\documentclass[3p,a4paper,10pt, sort&compress]{elsarticle}
\usepackage[utf8]{inputenc}
\usepackage[T1]{fontenc}
\usepackage{lmodern}
\usepackage{csquotes}
\usepackage[ngerman,english]{babel}
\useshorthands{"}
\addto\extrasenglish{\languageshorthands{ngerman}}
\usepackage[usenames,dvipsnames]{xcolor}
\usepackage{xstring}
\usepackage{enumitem}

\usepackage{mathtools}
\usepackage{amsmath}
\usepackage{amsthm}
\usepackage{amssymb}
\usepackage{booktabs}
\usepackage{hyperref}
\usepackage{todonotes}
\hypersetup{colorlinks,
}
\usepackage{graphicx}
\graphicspath{{../graphics/}}
\usepackage{pdfpages}
\usepackage{xparse}
\usepackage{mathdots}
\usepackage{hhline}
\usepackage{ifthen}
\usepackage[europeanresistors]{circuitikz}
\usepackage{tikz}

\usepackage{nomencl}
\makenomenclature
\usepackage{multicol}
\makeatletter
\@ifundefined{chapter}
  {\def\wilh@nomsection{section}}
  {\def\wilh@nomsection{chapter}}

\def\thenomenclature{%
  \begin{multicols}{2}[%
    \csname\wilh@nomsection\endcsname*{\nomname}
    \if@intoc\addcontentsline{toc}{\wilh@nomsection}{\nomname}\fi
    \nompreamble]
  \list{}{%
    \labelwidth\nom@tempdim
    \leftmargin\labelwidth
    \advance\leftmargin\labelsep
    \itemsep\nomitemsep
    \let\makelabel\nomlabel}%
}
\def\endthenomenclature{%
  \endlist
  \end{multicols}
  \nompostamble}
\makeatother
\renewcommand{\nomname}{Nomenclature}
\newcommand{\pipe}{\|} %

\theoremstyle{definition}
\newtheorem{definition}{Definition}[section]
\newtheorem{example}[definition]{Example}

\theoremstyle{plain}
\newtheorem{theorem}[definition]{Theorem}
\newtheorem{corollary}[definition]{Corollary}
\newtheorem{proposition}[definition]{Proposition}
\newtheorem{lemma}[definition]{Lemma}

\theoremstyle{remark}
\newtheorem{remark}[definition]{Remark}
\usepackage[cal=pxtx]{mathalfa}

\newcommand\Tstrut[1][2.6]{\rule{0pt}{#1ex}}  %
\newcommand\Bstrut{\rule[-0.9ex]{0pt}{0pt}}   %
\newcommand\shline[1]{\Bstrut \\ \hline \Tstrut[#1]}
\newcommand{\eg}{e.\,g.,\ } 
\newcommand{\ie}{i.\,e.,\ }
\newcommand{\eul}{\mathrm{e}}
\newcommand{\imag}{\mathrm{i}}
\newcommand{\A}{\mathcal{A}}
\newcommand{\B}{\mathcal{B}}

\newcommand{\D}{\mathcal{D}}

\newcommand{\E}{\mathcal{E}}

\newcommand{\W}{\mathcal{W}}
\newcommand{\Y}{\mathcal{Y}}
\newcommand{\cP}{\mathcal{P}}
\newcommand{\eio}[1][\omega]{\eul^{\imag #1}}
\newcommand{\emio}[1][\omega]{\eul^{-\imag #1}}
\newcommand{\N}{\mathbb{N}}
\newcommand{\R}{\mathbb{R}}
\newcommand{\C}{\mathbb{C}}
\newcommand{\K}{\mathbb{K}}
\newcommand{\cK}{{K}}
\newcommand{\cW}{\mathcal{W}}
\newcommand{\cR}{\mathcal{R}}
\newcommand{\cJ}{\mathcal{J}}

\newcommand{\mat}[3][K]{
\mathbb{#1}^{
\IfSubStr{#2}{+}{(#2)}{#2}
\times
\IfSubStr{#3}{+}{(#3)}{#3}
}
}
\newcommand{\matz}[3][K]{\mathbb{#1}[z]^
{
	\IfSubStr{#2}{+}{(#2)}{#2}
	\times
	\IfSubStr{#3}{+}{(#3)}{#3}
}}
\newcommand{\matrz}[3][K]{\mathbb{#1}(z)^
	{
		\IfSubStr{#2}{+}{(#2)}{#2}
		\times
		\IfSubStr{#3}{+}{(#3)}{#3}
	}}
\newcommand{\rkr}[1][K]{\operatorname{rk}_{\mathbb{#1}(z)}}

\newcommand{\rk}{ \operatorname{rk}}
\newcommand{\im}{ \operatorname{im}}
\newcommand{\In}{ \operatorname{In}}
\usepackage{bbm}

\newcommand{\diag}{ \operatorname{diag}}

\newcommand{\sign}{\varepsilon}

\DeclareDocumentCommand{\Sigmn}{O{m} O{n} }{\Sigma_{#1,#2}(\K)}
\DeclareDocumentCommand{\Sigmnq}{O{m} O{n} O{q}}{\Sigma_{#1,#2,#3}(\K)}
\DeclareDocumentCommand{\Sign}{ O{n} }{\Sigma_{#1}(\K)}
\DeclareDocumentCommand{\Sigmnw}{O{m} O{n} }{\Sigma_{#1,#2}^w(\K)}
\DeclareDocumentCommand{\system}{ O{E} O{A} O{B} O{m} O{n} s t{S}}{
\IfBooleanTF{#7}{(#1,\,#2)
 \IfBooleanTF{#6}{}
  {\in\Sigmn[0][#5]}
}
{(#1,\,#2,\,#3)
  \IfBooleanTF{#6}{}
  {\in\Sigmn[#4][#5]}}
}
\DeclareDocumentCommand{\outputsystem}{ O{E} O{A} O{B} O{C} O{D} O{m} O{n} s t{S}}{
	\IfBooleanTF{#9}{(#1,\,#2)
		\IfBooleanTF{#8}{}
		{\in\Sigmnq[0][#9]}
	}
	{(#1,\,#2,\,#3,\,#4,\,#5)
		\IfBooleanTF{#8}{}
		{\in\Sigmnq[#6][#7]}}
}
\DeclareDocumentCommand{\wsystem}{ O{E} O{A} O{B} O{Q} O{S} O{R} O{m} O{n} s}{(#1,\,#2,\,#3,\,#4,\,#5,\,#6)
\IfBooleanTF{#9}{}
  {\in\Sigmnw[#7][#8]}
}
\DeclareDocumentCommand{\bset}{ O{E} O{A} O{B} t{T} }{
\mathfrak{B}_{
\IfBooleanTF{#4}{(#1,\,#2)}
{
\system[#1][#2][#3]*} 
}
}
\DeclareDocumentCommand{\behavior}{s t{C} O{x} O{u} O{E} t{T}  }{
\IfBooleanTF{#1}{}
{
\IfBooleanTF{#6}{#3\in}
{\vect{#3,#4}\in}}
\IfBooleanTF{#6}{\bset[#5]T}{\bset[#5]}
}

\DeclareDocumentCommand{\behaviorC}{s}{CHANGE!}
\DeclareDocumentCommand{\wsystemF}{s}{
	\IfBooleanTF{#1}{\wsystem[E_F][A_F][B_F][Q_F][S_F][R_F][m][n]*}
	{\wsystem[E_F][A_F][B_F][Q_F][S_F][R_F][m][n]}
}

\DeclareDocumentCommand{\systemF}{s}{
	\IfBooleanTF{#1}{\system[E_F][A_F][B_F][m][n]*}
	{\wsystem[E_F][A_F][B_F][m][n]}
}
\DeclareDocumentCommand{\objfunc}{t{C} O{x} O{u} }{
\mathcal J(#2,#3)
}
\DeclareDocumentCommand{\inffunc}{t{C} O{E} O{x^0} }{
\mathcal W_+(#2 #3)
}
\DeclareDocumentCommand{\Vcons}{t{F} O{E} O{A} O{B}}{
\IfBooleanTF{#1}
{\mathcal{W}^c_F%
}
{\mathcal{W}^c%
}
}
\DeclareDocumentCommand{\cV}{O{\Sigma} t{F} }{
	\mathcal{V}
	\IfBooleanTF{#2}
	{_{\systemF*}}
	{_{\system*}}
}
\DeclareDocumentCommand{\Vshift}{O{E} O{A} O{B}}{\mathcal{W}_{\system[#1][#2][#3]*}
}
\DeclareDocumentCommand{\VshiftZD}{O{E} O{A} O{B} O{C} O{D}}{\mathcal{W}^0_{\outputsystem[#1][#2][#3][#4][#5]*}
}
\DeclareDocumentCommand{\ZD}{O{E} O{A} O{B} O{C} O{D}}{\frak{ZD}_{\outputsystem[#1][#2][#3][#4][#5]*}
}
\DeclareDocumentCommand{\Vdiff}{O{E} O{A} O{B}}{\mathcal{W}%
^\OpDiff}

\DeclareDocumentCommand{\wsystemS}{s}{
\IfBooleanTF{#1}{ \wsystem[I_{n_1}][A_s][B_s][Q_s][S_s][R_s][m][n_1]* }
  { \wsystem[I_{n_1}][A_s][B_s][Q_s][S_s][R_s][m][n_1] }
 }

 \DeclareDocumentCommand{\OpShift}{}{\sigma}
 \DeclareDocumentCommand{\OpDisc}{O{1}}{\varDelta_#1}
 \DeclareDocumentCommand{\OpDiff}{}{\frac{\mathrm{d}}{\mathrm{d}t}}

  \ExplSyntaxOn
\NewDocumentCommand \vect { s o m }
 {
  \IfBooleanTF {#1}
  {{ \vectauxstar{#3}}^*
  }
  {{ \vectauxn{#3}}}
}

\DeclarePairedDelimiterX \vectauxstar [1] {\big\lparen} {\big\rparen}
 { \, \dbacc_vect:n { #1 } \, }
 \DeclarePairedDelimiterX \vectauxn [1] {\lparen} {\rparen}
 { \, \dbacc_vectn:n { #1 } \, }

\cs_new_protected:Npn \dbacc_vect:n #1
 {
  \seq_clear:N \l_tmpb_seq
  \seq_set_split:Nnn \l_tmpa_seq { , } { #1 }
  \seq_map_inline:Nn \l_tmpa_seq { \seq_put_right:Nn \l_tmpb_seq { ##1^* } }
  \seq_use:Nn \l_tmpb_seq { \;\;  }
 }
 
 \cs_new_protected:Npn \dbacc_vectn:n #1
 {
  \seq_clear:N \l_tmpb_seq
  \seq_set_split:Nnn \l_tmpa_seq { , } { #1 }
  \seq_map_inline:Nn \l_tmpa_seq { \seq_put_right:Nn \l_tmpb_seq { ##1 } }
  \seq_use:Nn \l_tmpb_seq { ,\, }
 }
\ExplSyntaxOff

 \newcommand{\propfour}{
 	for every $\varepsilon>0$ and every $x^0\in\Vshift$ there exists a $\behavior\cap(\ell^2(\K^n) \times \ell^2(\K^m))$ such that $Ex_0 = Ex^0$ and   
 	$\| Kx+ Lu\|_{\ell^2} < \varepsilon,$
 	}
 \newcommand{\propfourI}{
 	for every $\varepsilon>0$ and every $x^0\in\K^n$ there exists a $\behavior\cap(\ell^2(\K^n) \times \ell^2(\K^m))$ such that $Ex_0 = Ex^0$ and   
 	$\| Kx+ Lu\|_{\ell^2} < \varepsilon,$
 	}

\begin{document}
\title{On linear-quadratic optimal control of implicit difference equations}
\author[rvt,fn0]{Daniel~Bankmann\corref{cor}}
\cortext[cor]{Corresponding author}
\ead{bankmann@math.tu-berlin.de}
\author[rvt,fn]{Matthias~Voigt}
\ead{mvoigt@math.tu-berlin.de}
\fntext[fn0]{The first author has been partially supported by the European Research Council through the Advanced Grant ``Modeling, Simulation and
Control of Multi-Physics Systems'' (MODSIMCONMP).}
\fntext[fn]{The second author has been supported by the project \textit{SE1: Reduced Order Modeling for Data Assimilation} within the framework of the Einstein Center for Mathematice (ECMath) funded by the Einstein Foundation Berlin.}

\address[rvt]{Technische Universit\"at Berlin, Institut für Mathematik, Sekretariat MA 4--5, Stra{\ss}e des 17. Juni 136, 10623 Berlin, Germany} 
\begin{abstract}
In this work we investigate explicit and implicit difference equations and the corresponding infinite time horizon linear-quadratic optimal control problem. We derive conditions for feasibility of the optimal control problem as well as existence and uniqueness of optimal controls under certain weaker assumptions compared to the standard approaches in the literature which are using algebraic Riccati equations. To this end, we introduce and analyze a discrete-time Lur'e equation and a corresponding Kalman-Yakubovich-Popov inequality. We show that solvability of the Kalman-Yakubovich-Popov inequality can be characterized via the spectral structure of a certain palindromic matrix pencil. The deflating subspaces of this pencil are finally used to construct solutions of the Lur'e equation. The results of this work are transferred from the continuous-time case. However, many additional technical difficulties arise in this context.
\end{abstract}
\maketitle

\section{Introduction}

In this article we revisit the discrete-time linear-quadratic optimal control problem, that is minimizing a quadratic cost functional given by
\begin{equation*}\label{eq:objectivefunctioncontintro}
 \sum_{j=0}^\infty{
  \begin{pmatrix}
   x_j\\
   u_j
  \end{pmatrix}^*
  \begin{bmatrix}
   Q	& S\\
   S^*	& R
  \end{bmatrix}
  \begin{pmatrix}
   x_j\\
   u_j
  \end{pmatrix}
 }
\end{equation*}
subject to the \emph{implicit difference equation} 
\begin{equation}\label{eq:linsystemdisc}
 E\OpShift x_j = Ax_j + Bu_j,
\end{equation}
with the initial condition $Ex_0 = Ex^0$ and the stabilization condition $\lim_{j \to \infty} Ex_j = 0$. Here $\OpShift$ denotes the shift operator, \ie $\OpShift x_j = x_{j+1}$. Moreover, $(x_j)_j\in(\K^n)^{\N_0}$ is the \emph{state sequence}, and $(u_j)_j\in(\K^m)^{\N_0}$ is the \emph{input sequence}. Throughout this work we will further assume that the matrix pencil $zE-A\in\matz{n}{n}$ is regular, \ie {${\det(zE-A)\not\equiv 0}$}. Such discrete-time systems often appear during the time-discretization \cite{BreCP96} or discrete-time lifting \cite{KahMP99} of continuous-time differential-algebraic equations, but many problems can also be directly modeled as implicit difference equations \cite{LueA77,PeaCS88}. 

There is a large body of work concerning the linear-quadratic optimal control problem for differential-algebraic equations, see, \eg \cite{mehrmann_autonomous_1991, lancaster_algebraic_1995, backes_extremalbedingungen_2006,kurina_linear-quadratic_2004,reis_kalmanyakubovichpopov_2015}, just to mention a few. 

So far, the discrete-time optimal control problem has only been discussed in a few works, most of which treat the case $E=I_n$. However, several technical assumptions are usually made. In \cite{laub_schur_1979}, the case of an invertible $A$ with $Q \succeq 0$ and $R \succ 0$ is discussed. The invertibility is needed to form a symplectic matrix that is associated with a discrete-time algebraic Riccati equation
\begin{equation}\label{eq:DARE}
 A^*XA-X - (A^*XB + S)(B^*XB+R)^{-1} (B^*XA + S^*) + Q = 0, \quad X=X^*
\end{equation}
and the necessary optimality conditions. The invertibility condition is relaxed in \cite{pappas_numerical_1980-1,laub_invariant_1991} where instead of the symplectic matrix, a symplectic matrix pencil is considered. Another difficulty arises, if $R$ is not invertible. Then also the symplectic pencil cannot be formed and one has to turn to an extended symplectic pencil \cite{ionescu_computing_1992}, which is essentially what we will later call the BVD matrix pencil. However, for the analysis it is still assumed that this pencil is regular.

The relation of the optimal control problem to a certain linear matrix inequality (the Kalman-Yakubovich-Popov inequality) and so-called Popov functions is discussed in \cite{stoorvogel_zeug_1998}. Here a generalized algebraic Riccati equation is considered, where $B^*XB+R$ is not assumed to be invertible. It is shown that the solutions of this equation fulfill a certain rank-minimization property of the associated linear matrix inequality which, in contrast to the continuous-time case, do not need to be solutions of the algebraic Riccati equation \eqref{eq:DARE}. 

The case of optimal control problems for implicit difference equations, \ie the case where $E$ might be singular, has only been briefly considered in the literature. In \cite{bender_laub_svdzeug_1987}, a discrete-time algebraic Riccati equation similar to \eqref{eq:DARE} is derived by transforming the system into SVD coordinates and modifying the cost functional accordingly. This analysis needs an index-1 condition on the system to ensure the solvability of the optimality system. The monograph \cite{mehrmann_autonomous_1991} treats the problem numerically, \ie structure-preserving algorithms for symplectic matrix pencils are devised.

The goal of this work is a full theoretical analysis of the infinite time horizon linear-quadratic optimal control problem for implicit difference equations. In contrast to most other works, we do not impose any definiteness conditions on the cost functional nor the index of the system. Also, our notion of rank-minimality turns out to be more general than in \cite{stoorvogel_zeug_1998}. The results obtained in this paper are motivated by recent achievements for the continuous-time case \cite{reis_lure_2011,reis_kalmanyakubovichpopov_2015,voigt_linear-quadratic_2015}.

This paper is structured as follows. In Section~\ref{chap:prelim} we recap basic matrix and control theoretic notations and results. In Section~\ref{chap:kyp} we introduce a variant of the Kalman-Yakubovich-Popov inequality for implicit difference equations
given by
\begin{equation*}\label{eq:introkyp}
\mathcal M (P):=
    \begin{bmatrix}
      A^*PA - E^*PE + Q		& A^*PB + S \\
      B^*PA + S^*		& B^*PB + R
    \end{bmatrix} \succeq_{\cV} 0,\qquad P=P^*,
    \end{equation*}
a discrete-time version of the inequality introduced in \cite{reis_kalmanyakubovichpopov_2015}, where $\succeq_{\cV}$ denotes an inequality projected on a certain subspace $\cV$, \ie $V^*\mathcal M (P)V \succeq 0$ holds for any basis matrix $V$ of $\cV$. We show statements which relate the solvability of this inequality 
to the non-negativity of the Popov function on the unit circle, a certain rational matrix function defined by
\begin{equation*}
\Phi(z):=
  \begin{bmatrix}
      (zE-A)^{-1}B	\\
      I_m		
  \end{bmatrix} ^\sim
  \begin{bmatrix}
      Q		& S	\\
      S^*	& R	
  \end{bmatrix} 
  \begin{bmatrix}
      (zE-A)^{-1}B	\\
      I_m		
  \end{bmatrix} \in\matrz{m}{m},
\end{equation*}
where $G^\sim(z):=G\left(\overline{z}^{-1}\right)^*$ for a rational matrix $G(z)\in\matrz{n}{n}$.

In Sections~\ref{chap:inertia} and \ref{sec:inertia} we introduce the notion of inertia for palindromic matrix pencils evaluated on the unit circle and provide spectral  characterizations regarding positivity of the Popov function, similar to the characterizations which were obtained in \cite{reis_lure_2011} and \cite{voigt_linear-quadratic_2015} for even matrix pencils in the continuous-time case.

In Section~\ref{chap:lure} we investigate the Lur'e equation for the discrete-time optimal control problem which is a generalization of the algebraic Riccati equation \eqref{eq:DARE}. This means that we seek solution triples $(X,\,K,\,L)\in\mat{n}{n}\times\mat{q}{n}\times\mat{q}{m}$ fulfilling
\begin{equation*}
 \begin{bmatrix}
  A^*XA-E^*XE + Q		&  A^*XB + S	\\
  B^*XA	    + S^*	&  B^*XB + R
 \end{bmatrix} =_{\cV}
 \begin{bmatrix}
  K^*\\
  L^*
 \end{bmatrix}
 \begin{bmatrix}
  K & L
 \end{bmatrix}, \quad X=X^*,
\end{equation*}
where $q:=\rkr\Phi(z)$.

We show that solvability of this equation can be related to the existence of certain deflating subspaces of a palindromic matrix pencil of the form %
\begin{equation*}z
  \begin{bmatrix}
    0		& E	& 0 		\\
    A^*		& Q	& S 	\\
	B^*	& S^*	& R
  \end{bmatrix}
    -
  \begin{bmatrix}
    0			& A		& B 		\\
    E^*		& Q	& S 	\\
   0		& S^*	& R
  \end{bmatrix}\in\matz{2n+m}{2n+m}.
  \end{equation*}
  The explicit construction of these deflating subspaces turns out to be much more demanding compared to the continuous-time setting.
  Furthermore, solutions of the Lur'e equation can be obtained from these subspaces. Due to the symmetry of the above pencil, robust and efficient structure-exploiting numerical methods \cite{byers_symplectic_2009,schroder_palindromic_2008} can be used.

Finally, in Section~\ref{chap:applications} we apply these results to the optimal control problem. Here we consider the stabilizing solutions of the Lur'e equations which can be used to construct optimal controls (in case of existence) and to determine the optimal value. In particular, we show that under some weak conditions, the existence of stabilizing solutions is equivalent to the feasibility of the optimal control problem. We further characterize existence and uniqueness of optimal controls in terms of the zero dynamics of the closed-loop system. Finally, we discuss how the deflating subspaces of the palindromic and BVD pencils appear in the solution of the corresponding boundary value problems.

\begin{table}[h!]
\footnotesize
\printnomenclature[2.4cm]
  \label{tab:nomenclature}
\end{table}
\section{Mathematical Preliminaries}\label{chap:prelim}

\subsection{Matrix Pencils}
In this section we briefly discuss some basic notions of matrix pencils $zE-A\in\matz{m}{n}$.
\begin{definition}[Equivalence of matrix pencils]
 Two matrix pencils $zE_1 - A_1, \, zE_2 - A_2 \in \matz{m}{n}$ are called equivalent if there exist invertible matrices $W \in \mat{m}{m}$ and $T \in \mat{n}{n}$ such that 
 \begin{equation*}
  zE_2 - A_2 = W (zE_1 - A_1) T.
 \end{equation*}
\end{definition}
Each matrix pencil $zE-A \in \matz[C]{m}{n}$ can be transformed to Kronecker canonical form via equivalence transformations. This is made precise in the next theorem.
\begin{theorem}[Kronecker canonical form (KCF)]\cite{gantmacher_theory_1960}
	For every matrix pencil $zE-A\in\matz[C]{m}{n}$, there exist invertible matrices $W\in\mat[C]{m}{m}$ and $T\in\mat[C]{n}{n}$ such that
	\begin{equation*}
	W(zE-A)T  =  \diag \left( \cK_1(z),\,\ldots,\,\cK_l(z)\right),\qquad l\in\N,
	\end{equation*}
where each block $\cK_j(z)$ is in one of the forms in Table~\ref{tab:kcf} and
\begin{equation*}
N_k=
\begin{bmatrix}
0& 		1	&		  	& \\
& 			\ddots&\ddots		&\\
&		& \ddots&1\\
&&&0   
\end{bmatrix}\in\mat[C]{k}{k},\,
K^R_k = \begin{bmatrix}
1&       &\\
0&\ddots&\\
 		&\ddots	 &1\\
&&0
\end{bmatrix}\in\mat[C]{k+1}{k},\,
K^L_k = \begin{bmatrix}
0& 1      & & \\
& \ddots&\ddots &\\
0 		&	 &0&1
\end{bmatrix}\in\mat[C]{k}{k+1}.
\end{equation*}
\begin{table}\centering
	\caption{Blocks in Kronecker canonical form}
	\begin{tabular}{cccc}
		type & size & $\cK_j(z)$  & parameters \\\toprule
		K1 &  $k\times k $ &
		$(z-\lambda)I_k - N_k
			$
			& $ k\in\N,\,\lambda\in\C$\\
			K2 & $k\times k$ & $zN_k - I_k $ & $ k\in\N$\\
				K3 & $k\times (k+1)$ & $z(K_k^R)^T - K_k^L $ & $ k\in\N_0$\\
				K4 & $(k+1)\times k$ & $zK_k^R - (K_k^L)^T $ & $ k\in\N_0$
		\\\bottomrule    
	\end{tabular}
	\label{tab:kcf}
\end{table}
The KCF is unique up to permutations of the blocks.
\end{theorem}
Equivalent matrix pencils share the same spectral structure which can be read off the KCF. Here blocks of type K1 and K2 correspond to finite eigenvalues and infinite eigenvalues, respectively. Blocks of these types and combinations of them are regular. Blocks of types K3 and K4 are rectangular and thus not regular. Note that we allow for blocks of type K3 or K4 to have zero rows or zero columns, respectively. Such blocks represent a zero row or zero column, respectively, in the KCF of $zE-A$. Based on the KCF we define the \emph{index} of the pencil $zE-A \in \matz[C]{m}{n}$ as the size $k$ of the largest block of type K2 or K4 in its KCF \cite{BerR13}. 

When characterizing the eigenstructure of matrices $A\in\mat{n}{n}$, often invariant subspaces are involved, \ie subspaces $\mathcal V \subseteq\K^n$ such that $A\mathcal V \subseteq \mathcal V$. The generalization of invariant subspaces to matrix pencils $zE-A\in\matz{m}{n}$ are so-called deflating subspaces. Here, we are using a general definition which is also suitable for singular matrix pencils, see \cite{voigt_linear-quadratic_2015, dooren_reducing_1983}.
\begin{definition}[Basis matrix, deflating subspaces]
  Let $zE-A\in\matz{n}{n}$ and some subspace $\Y\subseteq \K^n$ be given.
  \begin{enumerate}[label=(\alph*)]
   \item A matrix $Y\in\mat[C]{n}{k}$ with full column rank such that $\Y = \im Y$ is called \emph{basis matrix} of $\Y$.
   \item If for a basis matrix $Y\in\mat[C]{n}{k}$ of $\Y$ there exist $W\in\mat[C]{n}{p}$ and $z\hat E- \hat A\in\matz[C]{p}{n}$ such that
  \begin{equation*}
    (zE-A)V = W (z\hat E - \hat A)
  \end{equation*}
  and $\rkr[C] (z\hat E- \hat A) = p$, then $\Y$ is called \emph{deflating subspace} of $z E-  A$.
    \end{enumerate}
\end{definition}
Indeed, every invariant subspace $\mathcal V \subseteq\K^n$ of $A\in\mat{n}{n}$ with basis matrix $V\in\mat{n}{k}$  describes a deflating subspace for the associated matrix pencil $zI_n -A$ by setting $W=V$ and $(z\hat E - \hat A)= zI_k -\Lambda$, where $\Lambda\in\mat{k}{k}$ fulfills $AV=V\Lambda$.

An important property that deflating subspaces might have is $E$-neutrality.
\begin{definition}[$E$-neutrality]\cite{gohberg_indefinite_2006, reis_lure_2011}
Let $E\in\mat{n}{n}$ and some subspace $\Y\subseteq\K^n$ be given. Then $\Y$ is called \emph{$E$-neutral} if for all $x,\,y\in\Y$ it holds that
$
 x^*E y=0.
$
It is called \emph{maximally} $E$-neutral if every proper superspace $\W\supsetneq\Y$ is not  $E$-neutral. 
\end{definition}
For a subspace $\Y\subseteq\K^n$ we can check $E$-neutrality by testing whether $Y^*EY=0$,
where $Y\in\mat{n}{k}$ is given such that $\im Y = \Y$.
  \nomenclature[fasystems]{$\Sigmn$}{set of all $\system *\in\mat{n}{n}\times\mat{n}{n}\times\mat{n}{m}$ with regular $zE-A$}

\subsection{Feedback Equivalence}
Let $\Sigmn$ denote the set containing all the system triples $\system*\in\mat{n}{n}\times\mat{n}{n}\times\mat{n}{m}$ with regular $zE-A$, \ie $\det(zE-A)\neq 0$. Later, we will also use the set $\Sigmnw$ containing all systems $\wsystem *\in\mat{n}{n}\times\mat{n}{n}\times\mat{n}{m}\times\mat{n}{n}\times\mat{n}{m}\times\mat{m}{m}$, where $\system$ and $Q$ and $R$ are Hermitian.
Furthermore, we call the space of all $\vect{x, u}\in(\K^n)^{\N_0}\times(\K^m)^{\N_0}$ that solve the IDE \eqref{eq:linsystemdisc} the \emph{behavior} of the system $\system*$. The behavior is denoted by $\behavior*$.
In this subsection we introduce an equivalence relation on the set $\Sigmn$   which will be particularly useful in Sections~\ref{chap:kyp} and~\ref{chap:lure}. This subsection is mainly based on \cite[Section 2.3]{reis_kalmanyakubovichpopov_2015}.
\begin{definition}[Feedback equivalence]
  Two systems $(E_i,\,A_i,\,B_i)\in\Sigmn,\,i=1,2$, are said to be \emph{feedback equivalent} if there exist invertible matrices $W,\,T\in\mat{n}{n}$ and a feedback matrix $F\in\mat{m}{n}$ such that 
  \begin{equation*}
    \begin{bmatrix}
	zE_2-A_2 & -B_2
    \end{bmatrix}
    = W 
    \begin{bmatrix}
	zE_1-A_1 & -B_1
    \end{bmatrix}
    \mathcal{T}_F,
  \end{equation*}
  
  where 
  \[
    \mathcal{T}_F =
    \begin{bmatrix}
      T	& 0 \\
      FT & I_m
    \end{bmatrix}.
  \]
  If this is the case we say that $(E_1,\,A_1,\,B_1)$ is feedback equivalent to $(E_2,\,A_2,\,B_2)$ via $W$ and $\mathcal{T}_F$.
\end{definition}
Note that in the behavior sense, \ie looking at the system defined by $z\E-\A$, where
\begin{equation*}
 \E :=
 \begin{bmatrix}
  E	& 0
 \end{bmatrix},\qquad
 \A:=
 \begin{bmatrix}
  A	& B
 \end{bmatrix},
\end{equation*}
feedback equivalence corresponds to strong equivalence as introduced in \cite{kunkel_differential-algebraic_2006}.
In particular, this means that feedback equivalence is indeed an equivalence relation, see \cite[Lemma~2.2.]{kunkel_differential-algebraic_2006}.

Given such an equivalence relation, one is usually interested in some condensed form. The following result provides such a form.
\begin{theorem}[Feedback equivalence form] \cite[Proposition 2.12]{ilchmann_outer_2014}
  Let the system $\system$ be given. Then  $\system *$ is feedback equivalent to $\system[E_F][A_F][B_F]*$ via some $W$ and $\mathcal T _F$, where
  \begin{equation}\label{eq:fef}
  \begin{bmatrix}
   zE_F - A_F & -B_F 
  \end{bmatrix}=\left[
    \begin{array}{@{} ccc|c@{}}
      zI_{n_1}-A_{11}	& 0		& 0			& -B_1	\\
      0			& - I_{n_2}	& zE_{23}		& -B_2	\\
      0			& 0		& zE_{33} - I_{n_3}	& 0
    \end{array}\right],
  \end{equation}
  $n_1, n_2, n_3\in\N_0$, and $E_{33}$ is nilpotent. 
  \end{theorem}
  A similar form has also been achieved in \cite[Theorem 4.1]{byers_descriptor_1997} via unitary transformations.
  \begin{example}\label{ex:simplecircuitfeedback}
    Consider the system given by 
 \begin{equation}\label{eq:simplecircuitmatrices}
      E=
      \begin{bmatrix}
	  0	& 0\\
	  0	& 1
      \end{bmatrix},\quad A=
      \begin{bmatrix}
	-1 	& 1  \\
	1 	& 0\\
      \end{bmatrix},\quad
      B=
      \begin{bmatrix}
        -1\\
        0
      \end{bmatrix}.
    \end{equation}    
    We obtain that the system is feedback equivalent to
    \begin{equation}\label{eq:simplecircuitfeedback}
    \begin{bmatrix}
      z E_F - A_F & -B_F
    \end{bmatrix}=
    W
        \begin{bmatrix}
      z E- A & -B
    \end{bmatrix}\mathcal T_F= 
      \begin{bmatrix}
        z -1 	& 0	& 1\\
        0	& -1	& -1
      \end{bmatrix}
    \end{equation}
    via zero feedback, \ie $F=0$ and
    \begin{equation*}\label{eq:simplecircuitfeedbacktransform}
      W=
      \begin{bmatrix}
	  1	& 1\\
	  -1			& 0        
      \end{bmatrix}, \quad\mathcal T _F=
      \begin{bmatrix}
      1			& 1		 	& 0\\
      1			& 0			& 0\\
      0			& 0			& 1
      \end{bmatrix}.
    \end{equation*}
    Thus, we have $n_1=n_2=1$ and $n_3=0$ in \eqref{eq:fef}.
  \end{example}

  \begin{proposition}\label{prop:fef}
   Let the system $\system$ be  feedback equivalent to the system $\system[E_F][A_F][B_F]$ in feedback equivalence form \eqref{eq:fef}. Further, denote by
    \begin{equation*}
    \system[I_{n_1}][A_{11}][B_1][m][n_1]
    \end{equation*}
    the associated \emph{explicit difference equation (EDE)} system.
   Then for $\lambda\in\C$ we have $\det(\lambda E_F -A_F)\neq0$  if and only if $\det(\lambda I_{n_1}  - A_{11})\neq0$.
\end{proposition}
  \begin{proof}
  Note that every nilpotent matrix $E_{33}\in\mat{n_3}{n_3}$ has only eigenvalues zero and thus $\det (zE_{33}-I_{n_3}) = (-1)^{n_3}$. Then the assertion follows immediately from the block-triangular structure of $zE_F-A_F$.
  \end{proof}
\subsection{System Space}
In this subsection we investigate properties of the solution space of the IDEs given by a system $\system$. This section is based on \cite[Chapter 3]{reis_kalmanyakubovichpopov_2015}. 
\begin{definition}\label{def:systemspace}
  Let $\system$. The smallest subspace $\cV\subseteq\K^{n+m}$  such that 
  \[
    \begin{pmatrix}
      x_j	\\
      u_j	
    \end{pmatrix}
    \in \cV
  \]
  for all $j\in\N_0$ and for all $\behavior $ is called the system space of $\system *$.
\end{definition}
\begin{lemma}\label{lem:feedbacksysspace}
  Let $\system$. Further, assume that $\system[E_F][A_F][B_F]$ is feedback equivalent to $\system *$ via $W$ and $\mathcal{T}_F$. Then the system spaces $\cV$ and $\cV F$ of $\system *$ and $\system[E_F][A_F][B_F]*$, respectively,  are related via
  \begin{equation*}
    \cV = \mathcal{T}_F \cV F .
  \end{equation*}

\end{lemma}
\begin{proof}
 The assertion has been shown in \cite[Lemma 3.2]{reis_kalmanyakubovichpopov_2015}.
\end{proof}

\begin{proposition}\label{prop:sysspace}
Let $\system$ be given. Further, assume that $\system[E_F][A_F][B_F]$ is feedback equivalent to $\system *$ via $W$ and $\mathcal{T}_F$  such that  $\system[E_F][A_F][B_F]*$ is in feedback equivalence form \eqref{eq:fef}.  Then we have:
\begin{enumerate}[label=(\alph*)]
  \item \label{it:sysspacea} It holds that
  $\cV F = \im V_F$, where 
  \begin{equation}\label{eq:reprsysspace}
   V_F := 
   \begin{bmatrix}
    I_{n_1}	& 0 & 0	& 0 \\
    0		& 0 & 0	& -B_2\\
    0		& 0 & 0	& 0	\\
    0		& 0 & 0	& I_m
   \end{bmatrix}\in\mat{n+m}{n+m}.
  \end{equation}

  \item \label{it:sysspaceb}
  It holds that
  \begin{equation}\label{eq:sysspace}
    \begin{bmatrix}
      A	& B
    \end{bmatrix}
    \begin{bmatrix}
      (zE-A)^{-1}B\\
      I_m
    \end{bmatrix}
    = z
    \begin{bmatrix}
      E	& 0
    \end{bmatrix}
    \begin{bmatrix}
      (zE-A)^{-1}B\\
      I_m
    \end{bmatrix}.
  \end{equation}
  \item \label{it:sysspacec}
  For all $\lambda\in\C$ with $\det(\lambda E - A)\neq 0 $ it holds that
  \begin{equation*}
    \im
    \begin{bmatrix}
    (\lambda E-A)^{-1}B\\
    I_m
  \end{bmatrix}
  \subseteq \cV .
  \end{equation*}
  \item \label{it:sysspaced}  Consider $V_F$ as in~\ref{it:sysspacea} and let  $V:= \mathcal T_F V_F$. Then
  \begin{equation*}
   V_F \cV F= \cV F
  \end{equation*}
and
\begin{equation*}
 V \mathcal T_F^{-1} \cV = \cV.
\end{equation*}
\end{enumerate}
\end{proposition}
 \begin{proof}
  Assertion~\ref{it:sysspacea} is shown in the proof of \cite[Proposition 3.3]{reis_kalmanyakubovichpopov_2015}. Assertion~\ref{it:sysspacec} is shown in \cite[Lemma 3.5]{reis_kalmanyakubovichpopov_2015}, where  part~\ref{it:sysspaceb} is obtained in the proof of  \cite[Lemma 3.5]{reis_kalmanyakubovichpopov_2015}.
  For part~\ref{it:sysspaced} see \cite[Proposition 2.29(d)]{bankmann_linear-quadratic_2015}.
 \end{proof}

\subsection{Controllability and Asymptotic Stability}
Before we introduce the linear-quadratic optimal control problem, we first need to recap several concepts of controllability and asymptotic stability for the system given by $\system$. These concepts are similar to the continuous-time case as in \cite{dai_singular_1989,bunse-gerstner_feedback_1999} and are discussed in, \eg \cite{dai_singular_1989, stykel_input-output_2003}.
\begin{definition}
 The system $\system$ or $\wsystem$ is called
 \begin{enumerate}[label=(\alph*)]
  \item \emph{completely controllable (C-controllable)} if for every initial point $x^0\in\K^n$ and every final point $x^{\rm f}\in\K^n$ there exista a $\behavior$ such that $x_0=x^0$ and $x_{j_{\rm f}}=x^{\rm f}$ at some timepoint $j_{\rm f}\in\N_0$;
  \item \emph{controllable on the reachable set (R-controllable)} if for every initial point $x^0\in\Vshift$ and every final point $x_{\rm f}\in\Vshift$ there exist $\behavior$ such that $Ex_0 = Ex^0$ and $Ex_{j_{\rm f}}=Ex^{\rm f}$ at some timepoint $j_{\rm f}\in\N_0$;
  \item \emph{I-controllable} if for every initial point $x^0\in\K^n$ there exists $\behavior$ such that $Ex_0=Ex^0$, \ie $\Vshift=\K^n$;
  \item \emph{stabilizable} if for every initial point $x^0\in\Vshift$ there exists a $\behavior$ such that $Ex_0=Ex^0$ and $\lim\limits_{j\to\infty} Ex_j = 0$.
 \end{enumerate}
\end{definition}
In the case where $E=I_n$, the notions R-controllability and C-controllability coincide and thus for systems of the form $\system[I_n]$ we omit the prefix R or C and say that they are \emph{controllable}. 
Table~\ref{tab:algchar} shows well-known characterizations of the different controllability notions \cite{dai_singular_1989, stykel_input-output_2003, berger_differential-algebraic_2014}.
\begin{table}\centering
\caption{Algebraic characterizations of controllability and stabilizability, where $S_\infty$ is a basis matrix of $\ker E$}
  \begin{tabular}{cl}
  notion & algebraic characterization \\\toprule
   R-controllability& 
   $ \rk
   \begin{bsmallmatrix}
    \lambda E -A	& B
   \end{bsmallmatrix}
    =n$, $\lambda\in\C$
    \\
    C-controllability &
    $
       \rk\begin{bsmallmatrix}
     E 	& B
   \end{bsmallmatrix}
    =n$,    $ \rk
    \begin{bsmallmatrix}
    \lambda E -A	& B
    \end{bsmallmatrix}
    =n$, $\lambda\in\C$\\
    I-controllability 
    &$
       \rk \begin{bsmallmatrix}
     E & AS_\infty	& B
   \end{bsmallmatrix}
    =n$\\
    Stabilizability
    &
       $ \rk
   \begin{bsmallmatrix}
    \lambda E -A	& B
   \end{bsmallmatrix}
    =n$, $\lambda \in\C,\, |\lambda| \ge 1$\\\bottomrule    
  \end{tabular}
\label{tab:algchar}
\end{table}

Eigenvalues $\lambda\in\C$ of $zE-A\in\matz{n}{n}$ such that 
$
  \rk
   \begin{bmatrix}
    \lambda E -A	& B
   \end{bmatrix}
    \neq n$
 destroy the controllability property. In this case we thus say that there is an \emph{uncontrollable mode at $\lambda$}; otherwise we say that there is a \emph{controllable mode at $\lambda$}.
 \begin{lemma}\label{lem:fefcontr}
  Let the system $\system$ be feedback equivalent to the system $\system[E_F][A_F][B_F]$ in feedback equivalence form \eqref{eq:fef} via $W$ and $\mathcal T_F$. Furthermore, denote by  
      $\system[I_{n_1}][A_{11}][B_1][m][n_1]$ the associated EDE system. Then we have:
    \begin{enumerate}[label=(\alph*)]
   \item \label{it:fefcontra}
        Let $\lambda\in\C$. Then the system $\system[E_F][A_F][B_F]*$ has an uncontrollable mode at $\lambda$ if and only if the system $\system[I_{n_1}][A_{11}][B_1]*$ has an uncontrollable mode at $\lambda$.
     \item \label{it:fefcontrb}
    The system $\system[E_F][A_F][B_F]*$ is R-controllable if and only if $\system[I_{n_1}][A_{11}][B_1]*$ 
    is controllable.
    \item \label{it:fefcontrc} The system $\system *$ is I-controllable if and only if $W$ and $\mathcal T_F$ can be chosen such that for $\system[E_F][A_F][B_F]*$ it holds that $n_3=0$.
   \end{enumerate}   
 \end{lemma}
 \begin{proof}
 Assertion~\ref{it:fefcontra} is shown in  \cite[Lemma~2.9(c)]{reis_kalmanyakubovichpopov_2015}.
   Then  assertion~\ref{it:fefcontrb} is an immediate consequence of \ref{it:fefcontra}, since by Table~\ref{tab:algchar} the system $\system[E_F][A_F][B_F]*$ is R-controllable if and only if all $\lambda\in\C$ there are controllable modes of $\system[E_F][A_F][B_F]*$ at $\lambda$. Part \ref{it:fefcontrc} is shown in \cite[Proposition 2.12]{ilchmann_outer_2014}.
 \end{proof}

\subsection{Zero Dynamics}
In this subsection we consider implicit difference equations with an output of the form
\begin{equation} \label{eq:IDEout}
 E\OpShift x_j = Ax_j + Bu_j, \quad y_j = Cx_j + Du_j,
\end{equation}
where $C \in \K^{q \times n}$ and $D \in \K^{q \times m}$. The set of such systems with $\system$ is denoted by $\Sigma_{m,n,q}(\K)$ and we write $(E,A,B,C,D) \in \Sigma_{m,n,q}(\K)$. The zero dynamics $\ZD$ of \eqref{eq:IDEout} simply consists of all $(x,u) \in \mathfrak{B}_{(E,A,B)}$ that result in a zero output, \ie
\begin{equation*}
 \ZD := \left\{ (x,u) \in (\K^n)^{\N_0} \times (\K^m)^{\N_0} \; \bigg| \; \begin{bmatrix} E \sigma - A & -B \\ C & D \end{bmatrix} \begin{pmatrix} x_j \\ u_j \end{pmatrix} = 0,\, j = 0,\,1,\,2,\,\ldots \right\}.
\end{equation*}
The set of zero dynamics with ``initial state'' $x^0 \in \K^n$ is defined by
\begin{equation*}
 \ZD(x^0) := \left\{ (x,u) \in \ZD \; | \; Ex_0 = Ex^0 \right\}.
\end{equation*}
The set of consistent initial shift variables for the zero dynamics is given by 
\begin{equation*}
 \VshiftZD := \left\{ x^0 \in \K^n \; | \; \ZD(x^0) \neq \emptyset \right\}. 
\end{equation*}
The following definition is an adaptation of the definition for continupus-time systems, see \cite{ilchmann_outer_2014,voigt_linear-quadratic_2015}.
\begin{definition} The zero dynamics $\ZD$ with set of consistent initial shift variables $\VshiftZD$ is called
  \begin{enumerate}[label=(\alph*)]
   \item \emph{stabilizable}, if for all $x^0 \in \VshiftZD$, there exists an $(x,u) \in \ZD(x^0)$ such that $\lim_{j \to \infty} (x_j,u_j) = 0$;
   \item \emph{asymptotically stable}, if for all $(x,u) \in \ZD$ it holds $\lim_{j \to \infty} (x_j,u_j) = 0$;
   \item \emph{strongly stabilizable}, if it is stabilizable and $\VshiftZD = \K^n$;
   \item \emph{strongly asymptotically stable}, if it is asymptotically stable and $\VshiftZD = \K^n$.
  \end{enumerate}
\end{definition}

\begin{proposition}\label{prop:zd}
 Let $(E,A,B,C,D) \in \Sigma_{m,n,q}(\K)$ be given and define $\cR(z) := \left[\begin{smallmatrix} zE - A & -B \\ C & D \end{smallmatrix}\right] \in \K[z]^{(n+q) \times (n+m)}$. Then the zero dynamics $\ZD$ with the space of consistent shift variables $\VshiftZD$ is
 \begin{enumerate}[label=(\alph*)]
   \item \emph{stabilizable}, if and only if $\rk_{\K(z)} \cR(z) = \rk \cR(\lambda)$ for all $\lambda \in \C$ with $|\lambda| \ge 1$;
   \item \emph{stabilizable}, if and only if for all $x^0 \in \VshiftZD$ there exists a $(x,u) \in \ZD(x^0)$ such that $\lim_{j \to \infty} Ex_j = 0$;
   \item \emph{asymptotically stable}, if and only if $\rk \cR(\lambda) = n+m$ for all $\lambda \in \C$ with $|\lambda| \ge 1$;
   \item \emph{asymptotically stable}, if and only if for all $x^0 \in \VshiftZD$ there exists a unique $(x,u) \in \ZD(x^0)$ such that $\lim_{j \to \infty} Ex_j = 0$;
   \item \emph{strongly stabilizable}, if and only if $\rk_{\K(z)} \cR(z) = \rk \cR(\lambda)$ for all $\lambda \in \C$ with $|\lambda| \ge 1$ and the index of $\cR(z)$ is at most one;
   \item \emph{strongly asymptotically stable}, if and only if $\rk \cR(\lambda) = n+m$ for all $\lambda \in \C$ with $|\lambda| \ge 1$ and the index of $\cR(z)$ is at most one;
  \end{enumerate}
\end{proposition}

\begin{proof} The proof follows the lines of the proof of \cite[Prop.~4.3]{ilchmann_outer_2014}. It is analogously verified that the zero dynamics $\ZD$ is
 \begin{enumerate}[label=(\alph*)]
  \item stabilizable, if and only if all blocks of type K1 in the KCF of $\cR(z)$ correspond to eigenvalues with $|\lambda| < 1$;
  \item asymptotically stable, if and only if all blocks of type K1 in the KCF of $\cR(z)$ correspond to eigenvalues with $|\lambda| < 1$ and the number of blocks of type K3 in the KCF of $\cR(z)$ is zero.
 \end{enumerate}
 Together with \cite[Rem.~2.5 (d), (e)]{ilchmann_outer_2014}, this shows statements (a)--(d).
 Moreover, it is checked that $\VshiftZD = \K^n$ is equivalent to the condition that for every $x^0 \in \K^n$ the IDE
 \begin{equation*}
  \begin{bmatrix} E & 0 \\ 0 & 0 \end{bmatrix}\sigma \begin{pmatrix} x_j \\ u_j \end{pmatrix} = \begin{bmatrix} A & B \\ C & D \end{bmatrix}\begin{pmatrix} x_j \\ u_j \end{pmatrix}, \quad Ex_0 = Ex^0
 \end{equation*}
 has a solution. This is equivalent to $\cR(z)$ being of index at most one \cite{berger_differential-algebraic_2014} and shows statements (e) and (f).
\end{proof}

\subsection{Linear-Quadratic Optimal Control}

One main goal of this work is to provide tools for analyzing the discrete-time infinite horizon linear-quadratic control problem \cite{kurina_linear-quadratic_2004, backes_extremalbedingungen_2006, mehrmann_autonomous_1991}. It is given by:

For $x^0\in\Vshift$ find $\behavior$ such that $Ex_0=Ex^0$, $\lim\limits_{j\rightarrow\infty} Ex_j= 0,$ and the \emph{objective function}
\begin{equation}\label{eq:objectivefunction}
\objfunc :=
 \sum_{j=0}^\infty{
  \begin{pmatrix}
   x_j\\
   u_j
  \end{pmatrix}^*
  \begin{bmatrix}
   Q	& S\\
   S^*	& R
  \end{bmatrix}
  \begin{pmatrix}
   x_j\\
   u_j
  \end{pmatrix}
 }
\end{equation}
is minimized. In other words, we are interested in the value of the functional $\inffunc\!:E\Vshift\rightarrow\R^+_0\cup\{\pm\infty\}$ defined by
\begin{equation*}
 \inffunc := \inf \left\{    
 \objfunc \,\middle|\, \behavior,\, Ex_0=Ex^0,\,\lim\limits_{j\rightarrow\infty} Ex_j= 0
 \right\}.
\end{equation*}
The problem is called \emph{feasible} if $\infty>\inffunc>-\infty$. It is called \emph{solvable} if the infimum is actually a minimum. Note that for $x^0\in\Vshift$ the existence of $\behavior$ such that $Ex_0=Ex^0$ is  guaranteed by the definition of $\Vshift$. 
If further $\system*$ is stabilizable we can choose $u$ such that in addition $\lim\limits_{j\rightarrow\infty} Ex_j= 0$, \ie $\inffunc< \infty$.

\nomenclature[fasystems]{$\Sigmnw$}{set of all $\wsystem *\in\mat{n}{n}\times\mat{n}{n}\times\mat{n}{m}\times\mat{n}{n}\times\mat{n}{m}\times\mat{m}{m}$, where $\system$ and $Q$ and $R$ are Hermitian}

It can be easily seen that the objective function $\objfunc$ does not change if  the system $\system[E_F][A_F][B_F]$ is equivalent to $\system *$ via $W$ and $\mathcal{T}_F$ and we use the modified weights
\begin{equation*}\label{eq:feedbackweights}
  \begin{bmatrix}
    Q_F	& S_F\\
    S_F^*& R_F
  \end{bmatrix}:=
  \begin{bmatrix}
   T^*(Q + F^*S^* + SF + F^*RF)T	& T^*(S+F^*R)\\
   (S^*+RF)T				& R
  \end{bmatrix}.
\end{equation*}

If we assume that the system $\wsystem$ is I-controllable and that
\begin{equation}\label{eq:QSRassumption}
  \begin{bmatrix}
    Q & S\\
    S^*& R
  \end{bmatrix}\succeq 0
\end{equation}
then it is well-known that in this case solutions of the optimal control problem can be characterized via certain structured matrix pencils, see, \cite{mackey_structured_2006, byers_symplectic_2009, kunkel_optimal_2008, mehrmann_autonomous_1991}.
One main contribution of this work is that we actually can drop these assumptions. 

In the discrete-time case, applying  Pontryagin's maximum principle \cite{pontryagin_mathematical_1962, mehrmann_autonomous_1991} leads to 
\begin{gather}
  \begin{split}\label{eq:optconditionbvd}
\begin{bmatrix}
    0			& E	& 0 	\\
    A^* 		& 0	& 0 	\\
    B^*			& 0	& 0
  \end{bmatrix}
  \OpShift
  \begin{pmatrix}
   \mu\\
   x\\
   u
  \end{pmatrix}
    =
  \begin{bmatrix}
    0			&  A	& B 	\\
    E^*			& Q	& S 	\\
	0		& S^*	& R
  \end{bmatrix}
  \begin{pmatrix}
   \mu\\
   x\\
   u
  \end{pmatrix},\quad
  Ex_0 = E x^0, \quad \lim\limits_{j\to\infty}{E^*\mu_j}=0,
  \end{split}
\end{gather}
where $\behavior$, $x^0\in\Vshift$, and $\mu\in{(\K^n)^{\N_0}}$ denote some Lagrange multipliers. This IDE can be analyzed by means of the matrix pencil
\begin{equation}\label{eq:bvdpenc}
  z\E-\A  =
  \begin{bmatrix}
    0			& zE-A	& -B 		\\
    zA^*-E^*		& -Q	& -S 	\\
	zB^*		& -S^*	& -R
  \end{bmatrix}\in\matz{2n+m}{2n+m},
\end{equation}
a so-called \emph{BVD-pencil}; here BVD is an acronym for \emph{\textbf{B}oundary \textbf{V}alue problem for the optimal control of
\textbf{D}iscrete systems}. The structure of this pencil is not invariant under unitary transformations which leads to problems in the numerical treatment \cite{byers_symplectic_2009}. In \cite{byers_symplectic_2009, schroder_palindromic_2008}  it is shown how we can achieve a more structured version if we introduce new variables
\begin{equation*}\label{eq:lagrangetransformation}
  m_j:=\mu_j - \mu_{j+1}. 
\end{equation*}
This reformulation yields 
\begin{gather}
  \begin{split}\label{eq:optconditionpal}
\begin{bmatrix}
    0			& E	& 0 	\\
    A^* 		& Q	& S 	\\
    B^*			& S^*	& R
  \end{bmatrix}
  \OpShift
  \begin{pmatrix}
   \mu\\
   x\\
   u
  \end{pmatrix}
    =
  \begin{bmatrix}
    0			&  A	& B 	\\
    E^*			& Q	& S 	\\
	0		& S^*	& R
  \end{bmatrix}
  \begin{pmatrix}
   \mu\\
   x\\
   u
  \end{pmatrix},\quad  Ex_0 = E x^0, \quad \sum\limits_{j=0}^{\infty}{E^*m_j}=E^*\mu_j
\end{split}
  \end{gather}
with the corresponding matrix pencil
\begin{equation}\label{eq:palpenc}
  z\E-\A  = z\A^*-\A
  \begin{bmatrix}
    0			& zE-A		& -B 		\\
    zA^*-E^*		& (z-1)Q	& (z-1)S 	\\
	zB^*		& (z-1)S^*	& (z-1)R
  \end{bmatrix}\in\matz{2n+m}{2n+m}.
\end{equation}
This pencil has the special property of being \emph{palindromic}, \ie $\E=\A^*$. This structure is preserved under congruence transformation and there exist numerically stable and structure-preserving methods for the computation of eigenvalues and deflating subspaces. In particular, simple eigenvalues on the unit circle stay on the unit circle \cite{byers_symplectic_2009, schroder_palindromic_2008}.
In Section~\ref{chap:inertia} we  discuss properties of palindromic pencils in more detail.
However, in \cite{mehrmann_self-conjugate_2014} it is shown that in an abstract Banach space setting the operator associated to the palindromic pencil \eqref{eq:palpenc} is not self-adjoint;  in contrast to the operator associated to the so-called even pencil  arising in continuous-time, see also \cite{kunkel_self-adjoint_2014}.
We show in Sections~\ref{chap:lure} and \ref{chap:applications} that in analogy to the continuous-time case in  \cite{reis_kalmanyakubovichpopov_2015}, also in the discrete-time case we can drop assumption \eqref{eq:QSRassumption} to obtain the necessary optimality conditions \eqref{eq:optconditionbvd} and \eqref{eq:optconditionpal}.
\section{Kalman-Yakubovich-Popov Lemma}\label{chap:kyp}
Consider the weighted system $\wsystem$ and corresponding system space $\cV$. In this section we relate positive semi-definiteness on the unit circle of the Popov function -- a specific rational matrix function -- to the solvability  of a certain matrix inequality, namely the Kalman-Yakubovich-Popov inequality. We will see in Section~\ref{chap:applications} that positive semi-definiteness on the unit circle of the Popov function is necessary for feasibility of the optimal control problem \eqref{eq:objectivefunction}.

\begin{definition}
Let $\wsystem$ be given. Consider $P=P^*\in\mat{n}{n}$ and
  \begin{equation}\label{eq:KYPmatrix}
    \mathcal{M}(P):=
    \begin{bmatrix}
      A^*PA - E^*PE + Q		& A^*PB + S \\
      B^*PA + S^*		& B^*PB + R
    \end{bmatrix}.
  \end{equation}
  If  $\mathcal{M}(P)\succeq_{\cV}\!0$, then $P$ is called solution of the discrete-time \emph{Kalman-Yakubovich-Popov (KYP)} inequality
  \begin{equation}\label{eq:LMI}
    \mathcal{M}(P)\succeq_{\cV}\!0,\qquad P^*=P.
  \end{equation}
\end{definition}

Throughout this chapter we will make use of the system $\wsystem$  being transformed to $\wsystemF*$ via feedback equivalence,  \ie we have invertible $W,\,T\in\mat{n}{n}$ and a feedback matrix $F\in\mat{m}{n}$ such that
\begin{equation}\label{eq:fbsys}
  \begin{matrix}
     E_F=WET,\qquad A_F=W(A+BF)T ,\qquad B_F=WB,\\
     Q_F= T^*(Q+SF+F^*S^*+F^*RF)T,\qquad S_F=T^*(S+F^*R),\qquad R_F=R.     
  \end{matrix}
\end{equation}
These transformations will allow us to extract an EDE formulation from the IDE problem.
The next results are crucial for the proof of the KYP Lemma  in the IDE case and are mainly adaptions of the corresponding results in \cite[Section 4]{reis_kalmanyakubovichpopov_2015}.
\begin{lemma}\label{lem:zerokyp} Let $\wsystem$. Then we have
  \begin{equation*}\label{eq:zerokyp}
    \begin{bmatrix}
      (zE-A)^{-1}B	\\
      I_m		
    \end{bmatrix} ^\sim
    \begin{bmatrix}
      A^*PA-E^*PE	& A^*PB\\
      B^*PA		& B^*PB
    \end{bmatrix}
    \begin{bmatrix}
      (zE-A)^{-1}B	\\
      I_m		
    \end{bmatrix} = 0.
  \end{equation*}
\end{lemma}
\begin{proof}%
     By using \eqref{eq:sysspace} we obtain
     \begin{align*}
     &\begin{bmatrix}
       (zE-A)^{-1}B	\\
       I_m		
     \end{bmatrix} ^\sim
     \begin{bmatrix}
       A^*PA-E^*PE	& A^*PB\\
       B^*PA		& B^*PB
     \end{bmatrix}
     \begin{bmatrix}
       (zE-A)^{-1}B	\\
       I_m		
     \end{bmatrix} \\
     \overset{\hphantom{\eqref{eq:sysspace}}}=&
     \begin{bmatrix}
       (zE-A)^{-1}B	\\
       I_m		
     \end{bmatrix} ^\sim
     \left(
     \begin{bmatrix}
       A^*\\
       B^*
     \end{bmatrix}
     P
     \begin{bmatrix}
       A	& B
     \end{bmatrix}
     -
         \begin{bmatrix}
       E^*\\
       0
     \end{bmatrix}
     P
     \begin{bmatrix}
       E	& 0
     \end{bmatrix}
     \right)
     \begin{bmatrix}
       (zE-A)^{-1}B	\\
       I_m		
     \end{bmatrix}\\
     \overset{\eqref{eq:sysspace}}=& 
     \begin{bmatrix}
       (zE-A)^{-1}B	\\
       I_m		
     \end{bmatrix} ^\sim
     \left(
     \overline{z}^{-*}
     \begin{bmatrix}
       E^*\\
       0
     \end{bmatrix}
     P
     \begin{bmatrix}
       E	& 0
     \end{bmatrix}
     z
     -
         \begin{bmatrix}
       E^*\\
       0
     \end{bmatrix}
     P
     \begin{bmatrix}
       E	& 0
     \end{bmatrix}
     \right)
     \begin{bmatrix}
       (zE-A)^{-1}B	\\
       I_m		
     \end{bmatrix} = 0.%
   \end{align*}
\end{proof}

\begin{lemma}\label{lem:equivalence}
Let $\wsystem$ with corresponding  system $\wsystemF*$ in feedback equivalence form as in \eqref{eq:fbsys} be given.  Further, let $P=P^*\in\mat{n}{n}$ and set $P_F=W^{-*}PW^{-1}$ and
  \begin{equation*}
    \mathcal{T}_F=
    \begin{bmatrix}
      T	&	0 \\
      FT&	I_m
    \end{bmatrix}.
  \end{equation*}
Then
$
    \mathcal{M}_F(P_F)= \mathcal{T}_F^* \mathcal{M}(P) \mathcal{T}_F,
$
  where $\mathcal{M}_F(P_F)$ is the matrix in \eqref{eq:KYPmatrix} with respect to $\wsystemF*$.
\end{lemma}
For the generalization of the KYP inequality to implicit difference equations we first need to understand relations between the different Popov functions and KYP inequalities  corresponding to systems $\wsystem$ and 
$
  \wsystemF
$
as in  \eqref{eq:fbsys} and how they are related to explicit difference equations.

If $\wsystemF*$ is also in feedback equivalence form \eqref{eq:fef}, then the associated EDE part is given by
$\wsystemS$
which is defined by
 \begin{equation}\label{eq:EDEsys}
  \begin{matrix}
      A_s=A_{11},\qquad B_s=B_1,\\
     Q_s= Q_{11},\qquad S_s=S_1-Q_{12}B_2,\qquad R_s=B_2^*Q_{22}B_2 - B_2^*S_2-S_2^*B_2+R.     
  \end{matrix}
\end{equation}
\begin{proposition}\label{prop:popov}
  Consider the Popov function 
  \begin{equation*}
  \Phi_F(z)  
  :=
  \begin{bmatrix}
      (zE_F-A_F)^{-1}B_F	\\
      I_m		
  \end{bmatrix} ^\sim
  \begin{bmatrix}
      Q_F		& S_F	\\
      S_F^*	& R_F	
  \end{bmatrix} 
  \begin{bmatrix}
      (zE_F-A_F)^{-1}B_F	\\
      I_m		
  \end{bmatrix} \in \matrz{m}{m}
  \end{equation*}
  of the system $\wsystemF$ as in \eqref{eq:fbsys}.
  \begin{enumerate}[label=(\alph*)]
    \item \label{it:phifphia} The Popov functions $\Phi_F(z)$ and $\Phi(z)$ are related via
    \begin{equation*}\label{eq:phifphi}
      \Phi_F(z)=\Theta_F^\sim(z)\Phi(z)\Theta_F(z),
    \end{equation*}
    where $\Theta_F(z)=I_m+FT(zE_F-A_F)^{-1}B_F\in\matrz{m}{m}$ is invertible.
    \item \label{it:phifphib} Further, assume that $\wsystemF*$ is given in feedback equivalence form as in \eqref{eq:fef} and partitioned accordingly. Then it holds that
    $
    \Phi_F(z)=\Phi_s(z),
    $
    where $\Phi_s(z)$ is the Popov function corresponding to the EDE part
$
  \wsystemS$
as in \eqref{eq:EDEsys}.

  \end{enumerate}
\end{proposition}
\begin{proof} See \cite[Proposition 3.8]{bankmann_linear-quadratic_2015} and \cite{reis_kalmanyakubovichpopov_2015,voigt_linear-quadratic_2015}.
\end{proof}

We now turn to a reduction of the problem for systems $\wsystemF$ in feedback equivalence form as in \eqref{eq:fbsys} to the corresponding EDE system
\begin{equation*}
  \wsystemS
\end{equation*}
as in \eqref{eq:EDEsys}.

\begin{lemma}\label{lem:kypreduction}
Assume that $\wsystemF$ as in \eqref{eq:fbsys} is given in feedback equivalence form as in \eqref{eq:fef} and partitioned accordingly. Further, consider the  corresponding EDE part 
\begin{equation*}
\wsystemS  
\end{equation*}
 as in \eqref{eq:EDEsys} and partition the  Hermitian matrix
  \begin{equation*}
    P_F=
    \begin{bmatrix}
      P_{11} 	& P_{12} 	& P_{13}\\
      P_{12}^*	& P_{22} 	& P_{23}\\
      P_{13}^*	& P_{23}^*	& P_{33}
    \end{bmatrix}
    \in\mat{n}{n} 
  \end{equation*}
  accordingly. 
 Then $P_{11}\in\mat{n_1}{n_1}$ is a solution of the KYP inequality \eqref{eq:LMI} corresponding to the EDE part $\wsystemS *$
  if and only if 
  $P_F$ is a solution of the KYP inequality \eqref{eq:LMI} corresponding to 
  $
    \wsystemF *.
 $
\end{lemma}
\begin{proof}  
  We have
  \begin{multline}\label{eq:feedbackkyp}    
    \begin{bmatrix}
      A_F^*P_FA_F - E_F^*P_FE_F +Q_F		& A_F^*P_FB_F +S_F   \\
      B_F^*P_FA_F +S_F^*  			& B_F^*P_FB_F  +R_F
    \end{bmatrix}\\    
=\left[
    \begin{array}{@{}ccc|c@{}}
      A_{11}^*P_{11}A_{11}-P_{11} 		& A_{11}^*P_{12} 		& M_{13}	& A_{11}^*P_{11}B_1 + A_{11}^*P_{12}B_2 				\\
      P_{12}^*A_{11} 				& P_{22} 			& M_{23}	& P_{12}^*B_1 + P_{22}B_2 							\\
      M_{13}^*					& M_{23}^*				& M_{33}	& M_{34}								\shline{1.8}
      B_1^*P_{11}A_{11} + B_2^*P_{12}^*A_{11}  	& B_1^*P_{12} +B_2^*P_{22}		& M_{34}^*	& M_{44}
    \end{array}\right]
    \\+
\left[
    \begin{array}{@{}ccc|c@{}}
     Q_{11}		&  Q_{12}		& Q_{13}	&   S_1				\\
      Q_{12}^*		&  Q_{22}		& Q_{23}	&  S_2								\\
     Q_{13}^*		&  Q_{23}^*		& Q_{33}	& S_3		\shline{1.8}
      S_1^* 		&  S_2^*		& S_3^*		& R
    \end{array}\right]
  \end{multline}
  for some $M_{13}\in\mat{n_1}{n_3}$, $M_{23}\in\mat{n_2}{n_3}$, $M_{33}\in\mat{n_3}{n_3}$, $M_{34}\in\mat{n_3}{m}$, and
  \begin{equation*}
   M_{44}=B_1^*P_{11}B_1+B_1^*P_{12}B_2 +B_2^*P_{22}B_2 + B_2^*P_{12}^*B_1\in\mat{m}{m}.
  \end{equation*}

  Let
    $\vect*{x,u}\in\cV F$. Thus, by \eqref{eq:reprsysspace} there exists an $x_1\in\K^{n_1}$ such that 
   \begin{equation*}\label{eq:kypreduction1}
   x=\begin{pmatrix}
   x_1\\
   -B_2u\\
   0_{n_3\times1}
   \end{pmatrix}.
   \end{equation*}
   Then we obtain
   \begin{multline*}   \label{eq:reducekyp}
   \begin{pmatrix}
      x\\
      u
   \end{pmatrix}^*
    \begin{bmatrix}
      A_F^*P_FA_F - E_F^*P_FE_F + Q_F		& A_F^*P_FB_F + S_F  \\
      B_F^*P_FA_F + S_F^* 			& B_F^*P_FB_F + R_F 
    \end{bmatrix}
    \begin{pmatrix}
      x\\
      u
    \end{pmatrix}\\
    =
    \begin{pmatrix}
      x_1\\
      u
    \end{pmatrix}^*
    \begin{bmatrix}
      A_s^*P_{11}A_s-P_{11}+Q_s		&  A_s^*P_{11}B_s + S_s 				\\
      B_s^*P_{11}A_s + S_s^*		& B_s^*P_{11}B_s +  R_s
    \end{bmatrix}
    \begin{pmatrix}
      x_1\\
      u
    \end{pmatrix}.
   \end{multline*}
  Thus  
  \begin{equation*}\label{eq:EDEPARTKYP}
  \begin{pmatrix}
      x_1\\
      u
    \end{pmatrix}^*
    \begin{bmatrix}
      A_s^*P_{11}A_s-P_{11}+Q_s		&  A_s^*P_{11}B_s + S_s 				\\
      B_s^*P_{11}A_s + S_s^*		& B_s^*P_{11}B_s +  R_s
    \end{bmatrix}
    \begin{pmatrix}
      x_1\\
      u
    \end{pmatrix}
    \ge 0
  \end{equation*}
  for all $\vect{x_1,u}\in\K^{n_1+m}$ if and only if 
  \begin{equation*}
   \begin{pmatrix}
      x\\
      u
   \end{pmatrix}^*
    \begin{bmatrix}
      A_F^*P_FA_F - E_F^*P_FE_F + Q_F		& A_F^*P_FB_F + S_F  \\
      B_F^*P_FA_F + S_F^* 			& B_F^*P_FB_F + R_F 
    \end{bmatrix}
    \begin{pmatrix}
      x\\
      u
    \end{pmatrix}\ge0
  \end{equation*}
  for all $\vect*{x,u}\in\cV F$. Hence, $P_{11}$ is a solution of the KYP inequality \eqref{eq:LMI} corresponding to the EDE part if and only if  $P_F$ solves \eqref{eq:LMI} corresponding to $\system[E_F][A_F][B_F]*$.
\end{proof}

We are now ready to state the generalization of the KYP lemma for IDEs.
\begin{theorem}[KYP lemma for IDEs]\label{th:kyp}
  Let $\wsystem$ be given with corresponding Popov function $\Phi(z)\in\matrz{m}{m}$.
  \begin{enumerate}[label=(\alph*)]
    \item \label{it:kypa}If there exists some $P\in\mat{n}{n}$ that is a solution of \eqref{eq:LMI}, then
    $\Phi(\eio) \succeq 0$    for all $\omega \in \R$ with $\det (\eio E-A)\neq 0$.
     \item \label{it:kypb}If on the other hand $(E,A,B)$ is R-controllable and  $\Phi(\eio) \succeq 0$    for all $\omega \in \R$ with $\det (\eio E-A)\neq 0$,     
    then there exists a solution $P\in\mat{n}{n}$ of \eqref{eq:LMI}.
  \end{enumerate}
\end{theorem}
\begin{proof}We first show assertion \ref{it:kypa}.
  Assume that $P\in\mat{n}{n}$ fulfills the KYP inequality~\eqref{eq:LMI}, \ie $\mathcal{M}(P)\succeq_{\cV}\!0$. Further, let $\omega\in\R$ be such that $\det(\eio E-A)\neq0$. Then, by Lemma~\ref{lem:zerokyp}, together with Proposition~\ref{prop:sysspace}\ref{it:sysspaceb}, statement~\ref{it:kypa} follows due to
  \begin{align*}\label{eq:popovzerokyp}
    \begin{split}
    \Phi(\eio)=& 
    \begin{bmatrix}
      (\eio E-A)^{-1}B	\\
      I_m		
  \end{bmatrix} ^\sim
  \begin{bmatrix}
      Q		& S	\\
      S^*	& R	
  \end{bmatrix} 
  \begin{bmatrix}
      (\eio E-A)^{-1}B	\\
      I_m		
  \end{bmatrix}\\
  =&
  \begin{bmatrix}
      (\eio E-A)^{-1}B	\\
      I_m		
  \end{bmatrix} ^*
  \mathcal{M}(P)
  \begin{bmatrix}
      (\eio E-A)^{-1}B	\\
      I_m		
  \end{bmatrix} \succeq 0.
  \end{split}
  \end{align*}
  
  For part~\ref{it:kypb} assume that $\Phi(\eio)\succeq 0$ for all $\omega \in \R$ with $\det (\eio E-A)\neq 0$. For the system in feedback equivalence form $\wsystem[E_F][A_F][B_F][Q_F][S_F][R_F]$ and corresponding Popov function $\Phi_F(z)\in\matrz{m}{m}$ we obtain from Proposition~\ref{prop:popov}\ref{it:phifphib} that $\Phi_F(\eio)\succeq 0$ for all $\omega \in \R$ also fulfilling $\det (\eio E_F-A_F)\neq 0$. In particular, by Proposition~\ref{prop:fef} for such $\omega$ we have {${\det(\eio I_{n_1}-A_{11})\neq0}$}. Furthermore, by Proposition~\ref{lem:fefcontr}\ref{it:fefcontrb}
  the associated EDE system $\system[I_{n_1}][A_{11}][B_1]$ is controllable.
  
  This means we are in the situation of this theorem for explicit difference equations \cite{rantzer_kalman-yakubovich-popov_1996} for  the EDE system
  \begin{equation*}
  \wsystemS
  \end{equation*}
  as in \eqref{eq:EDEsys}.
  Thus, applying Lemma \ref{lem:kypreduction} gives a solution $P_F$ of the KYP inequality \eqref{eq:LMI} corresponding to the system $\wsystemF *$.  Then, using Lemma~\ref{lem:equivalence}  completes the proof.
  \end{proof}
  \begin{example}[Example~\ref{ex:simplecircuitfeedback} revisited]\label{ex:simplecircuitkyp}
    Consider the system $\system*$ as in Example~\ref{ex:simplecircuitfeedback}. 
    From its feedback equivalence form as in   \eqref{eq:simplecircuitfeedback}
    we obtain  with \eqref{eq:reprsysspace} that
    \begin{equation*}
      V_F=
      \begin{bmatrix}
        1 & 0	& 0\\
        0 & 0	& -1\\
        0 & 0	& 1
      \end{bmatrix}
    \end{equation*}
    spans the system space $\cV F$ and thus 
    \begin{equation*}
      \cV = \mathcal T_F \cV F=
      \im
       \begin{bmatrix}
        1 		& 0	& -1\\
        1 		& 0	& 0 \\
        0 		& 0	& 1
      \end{bmatrix}\quad
       \text{with}\quad
      \mathcal T _F=
      \begin{bmatrix}
      1			& 1		 	& 0\\
      1			& 0			& 0\\
      0			& 0			& 1
      \end{bmatrix}.
    \end{equation*}  
    From 
    \begin{equation*}
      \begin{bmatrix}
    Q & S\\
    S^* & R
  \end{bmatrix}=I_3
    \end{equation*}
    we obtain as modified weights 
    \begin{equation*}
\begin{bmatrix}
    Q_F	& S_F\\
    S_F^*& R_F
  \end{bmatrix}=
  \begin{bmatrix}
   T^*(Q + F^*S^* + SF + F^*RF)T	& T^*(S+F^*R)\\
   (S^*+RF)T				& R
  \end{bmatrix}
  = 
  \begin{bmatrix}
    2	& 1	& 0\\
    1	& 1	& 0\\
    0	& 0	& 1
  \end{bmatrix}.
    \end{equation*}
    Moreover, the associated EDE part as in \eqref{eq:EDEsys} is given by
     \begin{equation}\label{eq:simpleciruictEDE}
  \begin{matrix}
      A_s=1,\quad B_s=-1,\quad  Q_s= 2,\quad S_s=-1,\quad R_s=2.     
  \end{matrix}
\end{equation}
Thus, $P_{11}$ solves the KYP inequality
\begin{equation*}\label{eq:simplecircuitkyp}
  \begin{bmatrix}
    2		& -P_{11}-1 \\
     -P_{11}-1		& P_{11}+2
  \end{bmatrix}\succeq 0
\end{equation*} if and only if
\begin{equation*}
  -\sqrt{3}\le P_{11}\le \sqrt{3}.
\end{equation*}
Therefore, choosing $P_{11}=-1$, we have that
\begin{equation*}
    P=W^*P_FW=
    \begin{bmatrix}
      1		& -1\\
      1		& 0
    \end{bmatrix}
    \begin{bmatrix}
      -1		& 0\\\
      0		& 0
    \end{bmatrix}
    \begin{bmatrix}
      1		& 1\\
      -1	& 0
    \end{bmatrix}
    =
    \begin{bmatrix}
      -1		& -1\\
      -1		& -1
    \end{bmatrix}
\end{equation*}
solves the KYP inequality \eqref{eq:LMI}. In particular, by Theorem~\ref{th:kyp} we obtain that for the Popov functions $\Phi_F(z)\in\K(z)$ and $\Phi(z)\in\K(z)$ we have $\Phi_F(\eio)\succeq 0$ and $\Phi(\eio)\succeq 0$.
  \end{example}

  \begin{remark}
    The result of Theorem~\ref{th:kyp} is analogous to the continuous-time result in \cite{reis_kalmanyakubovichpopov_2015}. To see this, replace positivity of the Popov function on the unit circle by positivity on the imaginary axis in \ref{it:kypa} and replace $\mathcal M(P)$ by its continuous-time analog.
    However, in \cite{reis_kalmanyakubovichpopov_2015} the assumption of R-controllability was alternatively replaced by the condition that the Popov function has full rank and $\system*$ is \emph{sign-controllable}. To adapt this to the discrete-time setting we would need a discrete-time analog of \cite[Theorem 6.1]{clements_spectral_1997}, which provides  the characterizations via sign-controllability in the ODE case.
  \end{remark}

\section{Structure of Palindromic Matrix Pencils}\label{chap:inertia}
In this section we are concerned with palindromic matrix pencils $z\A^* - \A \in\matz{n}{n}$. For the investigation of these palindromic matrix pencils we first introduce so-called quasi-Hermitian matrices. Then  we show characterizations of the inertia of palindromic matrix pencils similar to what was done in \cite{reis_lure_2011,voigt_linear-quadratic_2015,clements_spectral_1997,clements_spectral_1989} in the case of so-called even matrix pencils.
The concept of quasi-Hermitian matrices  is an extension to  the notion of Hermitian and skew-Hermitian matrices. These are matrices $\A\in\mat{n}{n}$ with the property $\A=\eio \A ^*$ for some $\omega\in[0,2\pi)$. They have the special property that every eigenvalue lies on the line with angle $\omega/2$ through the origin.
We can extend the notion of inertia for Hermitian matrices to quasi-Hermitian matrices.
\begin{definition}
  Let $\A=\eio \A ^*$ be quasi-Hermitian with $\omega\in[0,2\pi)$. Then the inertia of $\A$ along $\frac{\omega}{2}$ is
  \begin{equation*}
   \In(\A) := \In_{\frac\omega2}(\A):=(n_+,\,n_0,\,n_-),
  \end{equation*}
  where $n_+$, $n_0$, and  $n_-$ denote the number of eigenvalues $\lambda=r\eio[\frac{\omega}{2}]$ where $r$ is positive, zero, or negative, respectively. We omit the subscript $\frac\omega2$ in $\In_{\frac\omega2}(\A)$ if the angle is clear from the context.
\end{definition}
Similar to the Hermitian case, also in the quasi-Hermitian case we have a canonical form under congruence transformations.
\begin{theorem}\label{th:inin}
  The inertia of a quasi-Hermitian matrix is invariant under congruence transformations. On the other hand, if $\A,\,\B\in\mat{n}{n}$ are two quasi-Hermitian matrices having the same inertia with respect to the same $\omega\in[0,2\pi)$, then there exists some invertible $U\in\mat{n}{n}$ such that
  \begin{equation*}
    U^*\A U=\B,
  \end{equation*}
  \ie $\A$ and $\B$ are congruent.
\end{theorem}
\begin{proof}
   See \cite{ikramov_inertia_2001}.
\end{proof}
We are now interested in a structure-preserving canonical form revealing the eigenstructure of a palindromic matrix pencil.

\begin{table}
\caption{Blocks $D(z)$ occurring in palindromic Kronecker canonical form and their inertia $\In\left( D(\eio)\right)$ for $\omega\in[0,2\pi)$, where the addition of two inertia tuples has to be understood component-wise}
\centering
\footnotesize
\arraycolsep=3pt
\medmuskip = 1mu
  \begin{tabular}{cccp{2.7cm}p{4cm}}\label{tab:PKCF}
  Type & Dimension & $D(z)$ & Parameters & $\In\left(D(\eio)\right)$\\\toprule
  P1 & $2k$ & 
  $\varepsilon\left[
    \begin{array}{@{}c|c@{}}
      &zF_k-J_k(\lambda) \shline{2.6}
     zJ_k(\overline\lambda)-F_k&
    \end{array}\right]
  $
 & ${k\in\N_0,}\,\varepsilon\in\{-1,1\},\newline\lambda=r\eio[\theta],\,{|r|<1,}\newline{\theta\in[0,2\pi)}$ &  $ \left(k,0,k\right)$
 \\[1.5em]
   P2 & $2k+1$ & 
   $\varepsilon\left[
    \begin{array}{@{}c|c|c@{}}
      & 			&zF_k-J_k(\lambda) \shline{3.8}
      &z\emio[\frac\theta2]- \eio[\frac\theta2] & -{\left(e^{ k}_1\right)}^T\shline{3.2}
     zJ_k(\overline\lambda)-F_k & ze^{ k}_1&
    \end{array}\right]
  $
  &${k\in\N_0,}\,\varepsilon\in\{-1,1\},\newline\lambda=\eio[\theta],{\theta\in[0,2\pi)}$ &  $\left(k,0,k\right)+\In\left(\sign(\omega-\theta)\right)$
  \\[3.5em]
    P3 & $2k$ & 
   $\varepsilon\left[
    \begin{array}{@{}c|c@{}}
      &zF_k-J_k(\lambda) \shline{3.2}
     zJ_k(\overline\lambda)-F_k& (z-1){e^{k}_1}{\left(e^{ k}_1\right)}^T
    \end{array}\right]
  $
  &${k\in\N_0,}\,\varepsilon\in\{-1,1\},\newline\lambda=\eio[\theta],{\theta\in[0,2\pi)}$ & 
  $\begin{cases}
    \left(k,0,k\right), & \omega \neq \theta \\
    \left(k-1,1,k-1\right) + \In(\varepsilon), &\omega = \theta
  \end{cases} $
  \\[2.5em]
    P4 & $2k$ & 
     $\varepsilon\left[
    \begin{array}{@{}c|c@{}}
      &zF_k-J_k(1) \shline{3.2}
     zJ_k(1)-F_k& \imag(z+1){e^{k}_1}{\left(e^{ k}_1\right)}^T
    \end{array}\right]
  $
  & $k\in\N_0,\,\varepsilon\in\{-1,1\}$ 
  &   $
    \begin{cases}
    \left(k,0,k\right), & \omega \neq 0 \\
    \left(k-1,1,k-1\right) + \In(\varepsilon), &\omega = 0
  \end{cases} $
  \\[2.5em]
    P5 & $2k+1$ & 
    $\left[
    \begin{array}{@{}c|c@{}}
      &zS_R - S_L^T \shline{2.6}
     zS_L - S_R^T&
    \end{array}\right]
  $
  &$k\in\N_0$ 
  &$\left(k,1,k\right)$
  \\\bottomrule
  \end{tabular}
  
\end{table}

\begin{theorem}[Palindromic Kronecker canonical form]\label{th:PKCF}\cite{schroder_palindromic_2008}
  Let $z\A^*-\A\in\matz{n}{n}$ be a palindromic matrix pencil. Then there exists some invertible $U\in\mat[C]{n}{n}$ such that 
  \begin{equation}\label{eq:PKCF}
  U^*(z\A^*-\A)U = \diag\left({D_1(z),\ldots,D_l(z)}
  \right)
  \end{equation}
  for some $l\in\N$ is in \emph{palindromic Kronecker canonical form (PKCF)}, where each block $D_j(z)\in\matz[C]{k_j}{k_j}$, $k_j\in\N$, is of one of the forms shown in Table~\ref{tab:PKCF} and
  \begin{gather*}
  F_k = 
  \begin{bmatrix}
    & & 1\\
    &\iddots&\\
    1&&
  \end{bmatrix}\in\mat[C]{k}{k},\qquad
  J_k(\lambda)=
  \begin{bmatrix}
    & 			&		  	& \lambda\\
    & 			&\iddots		& 1\\
    &\iddots 		& \iddots&\\
   \lambda & 1 & &  
  \end{bmatrix}\in\mat[C]{k}{k},\\
    S^R_k = \begin{bmatrix}
     			&       &1\\
    	 		& \iddots&0\\
     1 		&\iddots	 &\\
    0&&
  \end{bmatrix}\in\mat[C]{k+1}{k},\qquad
    S^L_k = \begin{bmatrix}
     			&       &0 & 1\\
    	 		& \iddots&\iddots &\\
     0 		&	1 &&
  \end{bmatrix}\in\mat[C]{k}{k+1}.
\end{gather*}
The PKCF is unique up to permutations of the blocks, and the quantities $\sign_j\in\{-1,1\}$ are called \emph{sign-characteristics}.
\end{theorem}
A closely related version of the above theorem was developed in \cite{horn_canonical_2006}.
\begin{remark}We have multiplied the sign-characteristics of the blocks of type P4 occurring in \cite{schroder_palindromic_2008} with $-1$ in order to simplify some of  the upcoming results. This is justified by the fact that if $\tilde{D}_j(z)$ with sign-characteristic $\tilde{\sign}_j$ corresponds to a block of type P4 introduced in \cite{schroder_palindromic_2008}, then $D_j(z)=-U^*\tilde{D}_j(z)U$ with
\begin{equation*}
U=\imag
  \begin{bmatrix}
   I_{k_j/2} & \\
	    & -I_{k_j/2}
  \end{bmatrix}\in\mat[C]{k_j}{k_j}
\end{equation*}
is a block of type P4 with sign-characteristic $\sign_j=-\tilde{\sign}_j$ according to Theorem~\ref{th:PKCF}.
\end{remark}

\begin{remark}By analyzing the eigenstructure of the blocks in the form \eqref{eq:PKCF} we obtain:
 \begin{enumerate}[label=(\alph*)]
   \item Blocks of type P1 correspond to eigenvalues $\lambda$ and ${1}/{\overline{\lambda}}$  with $|\lambda|\neq 1$, \ie these eigenvalues occur  in pairs $\left\{\lambda, {1}/{\overline{\lambda}}\right\}$. In particular, this holds for the pairing $\{0,\infty\}$.
   \item Blocks of type P2, P3, and P4 correspond to eigenvalues $\lambda$ with $|\lambda|= 1$. 
   \item Blocks of type P5 correspond to rank deficiency of the pencil, \ie they correspond to singular blocks.
 \end{enumerate}
\end{remark}
Consider the palindromic matrix pencil $\cP(z)=z\A^*-\A\in\matz{n}{n}$. By inserting $\eio$ for the polynomial variable $z$  we obtain 
\begin{equation*}\label{eq:palquherm}
  \cP(\eio)=\eio \A^*-\A = \imag\eio[\frac{\omega}{2}](\imag\emio[\frac{\omega}{2}]\A-\imag\eio[\frac{\omega}{2}]\A^*)
\end{equation*}
and hence
\begin{equation*}
  \cP(\eio)^*= -\imag\emio[\frac{\omega}{2}](\imag\emio[\frac{\omega}{2}]\A-\imag\eio[\frac{\omega}{2}]\A^*)= (-\imag\emio[\frac{\omega}{2}])^2\cP(\eio)=-\emio \cP(\eio).
\end{equation*}
Thus, $\cP(\eio)$ is quasi-Hermitian and has well-defined inertia. Investigating the block structure of the PKCF leads to the following result.

\begin{lemma}\label{lem:inertiapatterns} Assume that $z\A^*-\A\in\matz{n}{n}$ is in PKCF, \ie it holds that $z\A^*-\A = \diag\left({D_1(z),\ldots,D_l(z)}
  \right)$ for some $l\in\N$. Then the inertia pattern of each block $D_j(z)\in\matz[C]{k_j}{k_j}$, $k_j\in\N$, is given as in Table~\ref{tab:PKCF}.

\end{lemma}
\begin{proof}
	See \cite[Lemma 4.11]{bankmann_linear-quadratic_2015}.
\end{proof}
\begin{remark}
  The results from Lemma~\ref{lem:inertiapatterns} can be used to determine the block structure of a pencil $z\A^*-\A\in\matz{n}{n}$ in the form \eqref{eq:PKCF}, given the inertia patterns for $\omega\in[0,2\pi)$. 
Note that blocks of type P1 have a very simple inertia pattern and thus from a general pattern
  \begin{equation*}
    \In(\eio\A^*-\A) = (k_1,k_2,k_3)
  \end{equation*}
  -- except for the case where $k_2=0$, \ie all blocks are of type P1 -- we cannot tell whether or how many blocks of type P1 are present in the PKCF.. 
\end{remark}

\section{Inertia of Palindromic Pencils in Optimal Control}\label{sec:inertia}
  Let $\wsystem$ be given. We consider palindromic matrix pencils arising in the optimal control problem as in \eqref{eq:palpenc} of the form
\begin{equation*}\label{eq:palpenc2}
  z\A^*-\A  =
  \begin{bmatrix}
    0			& zE-A		& -B 		\\
    zA^*-E^*		& (z-1)Q	& (z-1)S 	\\
	zB^*		& (z-1)S^*	& (z-1)R
  \end{bmatrix}\in\matz{2n+m}{2n+m}.
\end{equation*}
If we insert $\eio$ into \eqref{eq:palpenc} for $z$ we obtain the quasi-Hermitian matrix
\begin{equation}\label{eq:palpencins}
  \D(\omega):= \imag\eio[\frac{\omega}{2}] (\imag\emio[\frac{\omega}{2}]\A-\imag\eio[\frac{\omega}{2}]\A^*)
  =  \imag\eio[\frac{\omega}{2}]
  \begin{bmatrix}
    0			& E_\omega-A_\omega		& B_\omega 		\\
    E_\omega^*-A_\omega^*	& Q_\omega			& S_\omega	 	\\
    B_\omega^*			& S_\omega^*			& R_\omega
  \end{bmatrix}\in\mat[C]{2n+m}{2n+m}
\end{equation}
with $E_\omega=-\imag\eio[\frac{\omega}{2}]E$, $A_\omega=-\imag\emio[\frac{\omega}{2}]A$, $B_\omega=\imag\emio[\frac{\omega}{2}]B$, $Q_\omega=s_\omega Q$, $S_\omega=s_\omega S$ and $R_\omega=s_\omega R$, where $s_\omega= \imag\emio[\frac{\omega}{2}]-\imag\eio[\frac{\omega}{2}]=2 \sin\left(\frac{\omega}{2}\right)$.
\begin{lemma}\label{lem:congrpp}
Let $\wsystem$ and consider the matrix $\D(\omega)$ as in \eqref{eq:palpencins} with $\omega$ such that $\det (E_\omega-A_\omega)\neq 0$. Furthermore, let
\begin{equation*}
  U= 
  \begin{bmatrix}
    I_n		&	0	&	(E_\omega^*-A_\omega^*)^{-1}(Q_\omega(E_\omega-A_\omega)^{-1}B_\omega - S_\omega)	\\
    0		&	I_n	&	-(E_\omega-A_\omega)^{-1}B_\omega							\\
    0		&	0	&	I_m		
  \end{bmatrix}\in\mat[C]{2n+m}{2n+m}.
\end{equation*}
Then $\D(\omega)$ is congruent to
\begin{equation*}
  U^*\D(\omega)U= \imag\eio[\frac{\omega}{2}]
  \begin{bmatrix}
    0				&		E_\omega-A_\omega		&	0	\\
    E_\omega^*-A_\omega^*		&		Q_\omega		&	0	\\
    0				&		0				&	2\sin\left(\frac{\omega}{2}\right)\Phi\left(\eio\right)		
  \end{bmatrix}.
\end{equation*}  
\end{lemma}
\begin{proof}
  See \cite[Lemma 4.13]{bankmann_linear-quadratic_2015}.
\end{proof}

\begin{theorem}\label{th:popovinertia}
  Let $\wsystem$ be given with corresponding Popov function $\Phi(z)\in\matrz{m}{m}$ and $\rkr \Phi(z)=q$ for some $q\in\N_0$. Assume that $\system*$ has no uncontrollable modes on the unit circle. Then the following are equivalent:
  \begin{enumerate}[label=(\alph*)]
    \item \label{it:popovinertiaa} The Popov function $\Phi(z)$ is positive semi-definite on the unit circle, \ie $\Phi(\eio)\succeq 0$ for all $\omega \in [0,2\pi)$.
    \item \label{it:popovinertiab} The following conditions     for the PKCF of $z\A^*-\A$ as in \eqref{eq:PKCF} hold:
    \begin{enumerate}[label=(\roman*)]
      \item \label{it:popovinertiabi} There are no blocks of type P2 corresponding to eigenvalues $\lambda=\eio[\theta],\, \theta\neq 0$, and all blocks of type P3 have negative sign-characteristic.
      \item \label{it:popovinertiabii} The number of blocks of type P2 corresponding to an eigenvalue $\lambda=1$ with positive sign-characteristic is greater by $q$ than the number of those with negative sign-characteristic.
    \end{enumerate}
    \item \label{it:popovinertiac} The following conditions     for the PKCF of $z\A^*-\A$ as in \eqref{eq:PKCF} hold:
    \begin{enumerate}[label=(\roman*')]
      \item \label{it:popovinertiaci} There are no blocks of type P2 corresponding to eigenvalues $\lambda=\eio[\theta],\, \theta\neq 0$.
      \item \label{it:popovinertiacii} The number of blocks of type P2 corresponding to an eigenvalue $\lambda=1$ with positive sign-characteristic is greater by $q$ than the number of those with negative sign-characteristic.
    \end{enumerate}
  \end{enumerate}
\end{theorem}
\begin{proof}
  The strategy of the proof is similar to the one in \cite[Theorem 3.4.2]{voigt_linear-quadratic_2015} for the continuous-time case. First note that since $\system*$ has no uncontrollable modes on the unit circle we can find a feedback matrix $F\in\mat{m}{n}$ such that $zE - (A+BF) \in \matz{n}{n}$ has no eigenvalues on the unit circle. Then by Lemma~\ref{lem:equivalence} and the fact that the palindromic pencil $z\A_F^*-\A_F$ corresponding to $\system[E][A+BF][B]*$ is connected to $z\A^*-\A$ via $\A_F= U_F^*\A U_F$, where
  \begin{equation*}
  U_F:=
   \begin{bmatrix}
   I_n		& 0 		& 0	 \\
   0		& I_n	 	& 0	 \\
   0		& F		& I_m 
  \end{bmatrix}\in\mat{2n+m}{2n+m},
  \end{equation*}
  we can assume without loss of generality that $(E,\,A)$ has no eigenvalues on the unit circle.
  
  Now we show that \ref{it:popovinertiaa} implies \ref{it:popovinertiab}. Therefore, assume that $\Phi(\eio)\succeq 0$  for all $\omega\in[0,2\pi)$ . Then in particular we have
  \begin{equation*}
    \In(\Phi(\eio))=\left(q-a(\omega),m-q+a(\omega),0\right)
  \end{equation*}
  for all $\omega \in(0,2\pi)$, where $a:(0,2\pi)\to\N_0$ is some function which is zero for almost all $\omega \in(0,2\pi)$. Hence, by Lemma~\ref{lem:congrpp} and Theorem~\ref{th:inin} we obtain
  \begin{equation*}
    \In\left(\eio\A^*-\A \right) = \left( n + q - a(\omega), m-q+a(\omega), n \right)
  \end{equation*}
  for $\omega\in(0,2\pi)$. Again, by Theorem \ref{th:inin} the inertia of $\eio\A^*-\A$ coincides with the inertia of the PKCF  of $z\A^*-\A$ as in \eqref{eq:PKCF} evaluated at $\eio$. Since by Theorem~\ref{th:PKCF} the block structure of the PKCF is uniquely determined, we can proceed by identifying blocks by their inertia patterns.\\
  Note that $\rkr (z\A^*-\A) = 2n + q$, since $z\A^*-\A$ can only have a finite amount of rank drops and due to Lemma \ref{lem:congrpp} and the regularity of $zE-A$ there exist infinitely many values $\lambda\in\C$ for which $\rk(\lambda\A^*-\A)=2n + q$.
  From Lemma \ref{lem:inertiapatterns}  we can infer that we have exactly $2n+m-(2n+q)=m-q$ blocks of type P5 in the PKCF of $z\A^*-\A$, since these are the only rank deficient blocks.\\
  Thus, since $\rk (\A^*-\A) = 2r$, where $r:=\rk \left[E-A\enspace B\right]$, the number of blocks of type P2 or P4 corresponding to an eigenvalue $\lambda=1$ is exactly $2(n-r)+q$.
  Then, removing the blocks of type P5 from the PKCF of $z\A^*-\A$  yields a matrix pencil $z\A^*_1-\A_1\in\matz{2n_1+q}{2n_1+q}$ in PKCF with full normal rank and inertia 
  \begin{equation*}
    \In(\eio\A^*_1-\A_1)=\left( n_1 + q - a(\omega), a(\omega), n_1 \right)
  \end{equation*}
  on $(0,2\pi)$.  
  Then, by Lemma \ref{lem:inertiapatterns}, there are $q$ blocks of type P2 with corresponding eigenvalue $\lambda=1$ and positive sign-characteristic, since these are 
present in every combination of blocks with an inertia pattern of the form 
$
    (k+1,0,k)$
  independent of $\omega>0$ . Removing these blocks leads to the pencil $z\A^*_2-\A_2\in\matz{2n_2}{2n_2}$ in PKCF with inertia 
  \begin{equation*}
    \In(\eio\A^*_2-\A_2)=\left( n_2  - a(\omega), a(\omega), n_2 \right)
  \end{equation*}
  on $(0,2\pi)$.
  Furthermore, from Lemma \ref{lem:inertiapatterns} we deduce that there are no blocks of type P2 corresponding to eigenvalues $\lambda=\eio[\theta],\, \theta\neq 0$. Thus, all blocks of type P3 have negative sign-characteristic, since these are
the only blocks with an inertia pattern of the form
    $(k-1,1,k)$
  for exactly one value of $\omega>0$. This shows statement \ref{it:popovinertiabi}. 
  Removing these blocks, we obtain a matrix pencil $z\A^*_3-\A_3\in\matz{2n_3}{2n_3}$ in PKCF with  inertia 
  \begin{equation*}
    \In(\eio\A^*_3-\A_3)=\left( n_3 , 0, n_3 \right)
  \end{equation*}
  on $(0,2\pi)$.
  The inertia pattern of $z\A_3^*-\A_3$ together with Lemma \ref{lem:inertiapatterns}
reveals that the remaining blocks of type P2 corresponding to an eigenvalue $\lambda=1$ are split up equally into those with positive and those with negative sign-characteristic. This shows \ref{it:popovinertiabii} and thus statement~\ref{it:popovinertiab}.
  
  The proof that \ref{it:popovinertiac} follows from \ref{it:popovinertiab} is clear, since condition~\ref{it:popovinertiaci} follows immediately from condition~\ref{it:popovinertiabi} and conditions~\ref{it:popovinertiabii} and~\ref{it:popovinertiacii} coincide. 
  
  Now let the conditions \ref{it:popovinertiaci}, and \ref{it:popovinertiacii} hold. 
  Again, by Lemma~\ref{lem:congrpp}
and Theorem~\ref{th:inin}, for $\omega\in(0,2\pi)$ we obtain
  \begin{align*}
    \In\left(\eio\A^*-\A \right) =& \left( n , 0, n \right) + \In\left(\Phi(\eio)\right)\\
    =& \left( n + m_1 - a_1(\omega) , m-m_1-m_2 + a_1(\omega) + a_2(\omega), n + m_2  - a_2(\omega) \right)
  \end{align*}
   and functions $a_i:(0,2\pi)\to \N_0$, $i=1,2$, which are zero for almost all $\omega \in(0,2\pi)$ such that $m_1  + m_2  = q$.
  We now have to show that $m_2=0= a_2(\omega)$. %
  Then the blocks of type P2 with positive sign-characteristic are the only ones leading to an inertia pattern of the form
   $(k+1,0,k)$
  for $\omega>\theta$. The only blocks that could compensate the additional positive eigenvalue for $\omega>\theta$ are
blocks of type P2 with negative sign-characteristic. By condition~\ref{it:popovinertiaci} we are only allowed to take such blocks with $\theta=0$. By condition~\ref{it:popovinertiacii} then we obtain that $n + m_1 = (n + m_2) + q$ and thus $m_2=0$, $m_1=q$. 
  Hence, we have 
  \begin{equation*}
    \In(\Phi(\eio))=\left( q - a_1(\omega) , m-q + a_1(\omega) + a_2(\omega), - a_2(\omega) \right).
  \end{equation*}
  Since the inertia of a quasi-Hermitian matrix by definition is a triple of non-negative integers, this implies $a_2\equiv0$ and thus $\Phi(\eio)\succeq 0$ for all $\omega\in(0,2\pi)$. Then, by continuity, we also have that ${\Phi(1)\succeq0}$.
\end{proof}

\begin{example}[Example~\ref{ex:simplecircuitfeedback} revisited]
 We consider the system $\wsystem*$  with corresponding system $\wsystemF*$ in feedback equivalence form as in \eqref{eq:simplecircuitmatrices}, \eqref{eq:simplecircuitfeedback}, and Example~\ref{ex:simplecircuitkyp}. The associated palindromic pencil $z\A^*-\A\in\matz{5}{5}$  as in \eqref{eq:palpenc} is given by
    \begin{equation}\label{eq:simplecircuitpalpenc}
    z\A^*-\A=\left[
    \begin{array}{@{}cc|cc|c@{}}
      0 & 0  & 1 	& -1 		& -1\\
      0 & 0  & -1 	& z 		&0\shline{1.6}
      -z & z & z-1 	&  0		&0\\
      z&-1 & 0	& z-1 		&0\shline{1.6}
      z & 0  & 0 	& 0 		&z-1
    \end{array}\right].
  \end{equation}
  Transforming the matrix $\A$ to the corresponding  matrix $\A_F$ of the system in feedback equivalence form \eqref{eq:simplecircuitfeedback}  via 
  \begin{equation*} U_F :=
     \begin{bmatrix}
   W^*		& 0 		& 0	 \\
   0		& T	 	& 0	 \\
   0		& FT 		& I_m 
  \end{bmatrix}=
 \left[
    \begin{array}{@{}cc|cc|c@{}}
   1		& -1		& 0		& 0		& 0	\\
   1		& 0		& 0		& 0		& 0	\shline{1.6}
   0		& 0		& 1		& 1 		& 0	 \\
   0		& 0		& 1		& 0	 	& 0	 \shline{1.6}
   0		& 0		& 0		& 0 		& 1 
 \end{array}\right]\in\mat{5}{5}
  \end{equation*}
  we obtain
  \begin{equation*}
    \A_F=  U_F^* \A U_F=
    \left[
    \begin{array}{@{}cc|cc|c@{}}
      0 & 0  & 1 	& 0 		& \phantom{-}1\\
      0 & 0  & 0 	& 1 		&-1\shline{1.6}
      1 & 0  & 2 	& 1 		&0\\
      0 & 0  & 1	& 1 		&0\shline{1.6}
      0 & 0  & 0 	& 0 		&1
    \end{array}\right].
    \end{equation*}
    The matrix $\A_F$ can can be further transformed to
    \begin{equation}\label{eq:simplecircuitPKCF}
      U^*(z\A_F^*-\A_F)U=
          \left[
\begin{array}{@{}cc|cc|c@{}}
 0 & -1 & 0 & 0 & 0 \\
 z & 0 & 0 & 0 & 0 \shline{1.6}
 0 & 0 & 0 & z-(2+\sqrt{3}) & 0 \\
 0 & 0 & (2+\sqrt{3})z-1 & 0 & 0 \shline{1.6}
 0 & 0 & 0 & 0 & z-1
\end{array}
\right]
  \end{equation}    
    in PKCF as in \eqref{eq:PKCF} via  
  \begin{equation*}  
  U=
\left[
\begin{array}{@{}ccccc@{}}
 0 & -1 & -1+\sqrt{3} & -\frac{3}{2}-\sqrt{3} & 0 \\
 1 & -1 & -1+\frac{1}{\sqrt{3}} & -1-\frac{\sqrt{3}}{2} & 0 \\
 0 & 0 & -\frac{1}{\sqrt{3}} & \frac{1}{4} \left(1+\sqrt{3}\right) & -\frac{1}{\sqrt{2}} \\
 0 & 1 & 1-\frac{1}{\sqrt{3}} & \frac{1}{2} \left(2+\sqrt{3}\right) & 0 \\
 0 & 0 & -1+\frac{1}{\sqrt{3}} & -1-\frac{\sqrt{3}}{2} & 0
\end{array}
\right].
  \end{equation*}
  From \eqref{eq:simplecircuitPKCF} we see that the PKCF of $z\A_F^*-\A_F$ and thus also of $z\A^*-\A$ consists of a $2\times2$ block of type P1 corresponding to the eigenvalues $\{0,\infty\}$, a $2\times2$ block of type P1 corresponding to the eigenvalues $\{2+\sqrt{3},2-\sqrt{3}\}$, and a $1\times1$ block of type P2 corresponding to the eigenvalue $1$.   
    Furthermore, for the Popov function $\Phi_F(z)$ it holds that $\rkr \Phi_F(z)=1$. Thus, we have shown that the assumptions of Theorem~\ref{th:popovinertia}\ref{it:popovinertiab} are fulfilled and hence $\Phi_F(\eio)\succeq 0$ for all $\omega\in[0,2\pi)$. Thus, we have confirmed the result obtained in Example~\ref{ex:simplecircuitkyp}.
\end{example}

\section{Lur'e Equations}\label{chap:lure}
In this section we characterize solvability of Lur'e equations for explicit as well as for implicit difference equations in a similar way as in \cite{reis_kalmanyakubovichpopov_2015} for continuous-time systems.
Finding a solution of the Lur'e equation means finding $X=X^*\in\mat{n}{n}$, $K\in\mat{q}{n}$, and $L\in\mat{q}{m}$ such that
\begin{equation}\label{eq:dlure}
 \mathcal M (X) =
 \begin{bmatrix}
  A^*XA-E^*XE + Q		&  A^*XB + S	\\
  B^*XA	    + S^*	&  B^*XB + R
 \end{bmatrix} =_{\cV}
 \begin{bmatrix}
  K^*\\
  L^*
 \end{bmatrix}
 \begin{bmatrix}
  K & L
 \end{bmatrix},
\end{equation}
where $q:=\rkr\Phi(z)$.

If $X$ is a solution of the KYP inequality \eqref{eq:LMI}, then we can always find $K\in\mat{p}{n}$ and $L\in\mat{p}{m}$ for some $p\in \N_0$ such that \eqref{eq:dlure} holds.
The next result shows that for such solutions it holds that $p\ge q$. Thus, in other words, we are interested in the existence of solutions of \eqref{eq:dlure} with minimal rank $q$.  
\begin{proposition}\label{prop:lurefullrankrhs}
  Let $\wsystem$ be given and  let $q=\rkr \Phi(z)$. Further, let
  \begin{equation*}
      (X,\,K,\,L)\in\mat{n}{n}\times\mat{q}{n}\times\mat{q}{m}
  \end{equation*}
  be a solution of the Lur'e equation \eqref{eq:dlure} and assume, that for $M\in\mat{p}{n}$ and $N\in\mat{p}{m}$ also the triple $(X,\,M,\,N)$ fulfills \eqref{eq:dlure}.
  Then we have $q\le p$ and 
  \begin{equation}\label{eq:lurerankcond}
  \rkr
  \begin{bmatrix}
      zE-A	& -B	\\
      (z-1)K	&(z-1)L
    \end{bmatrix} = n+q.
  \end{equation}
\end{proposition}
\begin{proof}
  See \cite[Proposition 5.1]{bankmann_linear-quadratic_2015}.
\end{proof}

Note that in the continuous-time case \cite{reis_kalmanyakubovichpopov_2015} instead of rank minimality the condition \eqref{eq:lurerankcond} was used to define solutions fo the Lur'e equation \eqref{eq:dlure}. Proposition~\ref{prop:lurefullrankrhs}  shows that both versions are indeed equivalent.
In the following we will derive certain deflating subspaces of BVD and palindromic matrix pencils, respectively, from a solution of the Lur'e equation \eqref{eq:dlure}. First, we do this for the case of explicit difference equations. Afterwards, based on these results, we do the generalization to the implicit case with the help of feedback transformations similarly to the approach in Section~\ref{chap:kyp}.
\subsection{Explicit Difference Equations}
In the EDE case, \ie systems $\wsystem[I_n]$ finding a solution of the Lur'e equation \eqref{eq:dlure} reduces to:

For $q:=\rkr \Phi(z)$ find $X\in\mat{n}{n}$, $K\in\mat{q}{n}$, and $L\in\mat{q}{m}$ such that

\begin{equation}\label{eq:clure}
\mathcal M (X)=
 \begin{bmatrix}
  A^*XA-X + Q		&  A^*XB + S	\\
  B^*XA	    + S^*	&  B^*XB + R
 \end{bmatrix}=
 \begin{bmatrix}
  K^*\\
  L^*
 \end{bmatrix}
 \begin{bmatrix}
  K & L
 \end{bmatrix}.
\end{equation}
The next result is an analogous version of \cite[Lemma 12]{reis_lure_2011} for  the discrete-time case. 
 \begin{lemma}\label{lem:minxkl}
  Let $\wsystem[I_n]$ be given and let $q=\rkr \Phi(z)$. Furthermore, let $\Phi(\eio)\succeq 0$ for all $\omega\in[0,2\pi)$ such that $\det(\eio I_n-A)\neq 0$ and let a Hermitian $X\in\mat{n}{n}$ be given with
  \begin{equation*}
  \rk
\mathcal M (X) =q.
 \end{equation*}
  Then \eqref{eq:clure} has a solution $(X,\,K,\,L)\in\mat{n}{n}\times\mat{q}{n}\times\mat{q}{m}$.
 \end{lemma}
\begin{proof}

The proof is analogous to the continuous-time case \cite{reis_lure_2011} and can be found in \cite[Lemma 5.2]{bankmann_linear-quadratic_2015}.
\end{proof}

\begin{example}[Example \ref{ex:simplecircuitfeedback} revisited]\label{ex:simplecircuitsdelure}
Consider the system $\wsystem*$ as in \eqref{eq:simplecircuitmatrices} and Example~\ref{ex:simplecircuitkyp}. We have seen in Example~\ref{ex:simplecircuitkyp} that with
\begin{equation*}
  \mathcal M_s (P_s)=
   \begin{bmatrix}
    2		& -P_s-1 \\
     -P_s-1		& P_s+2
  \end{bmatrix}
\end{equation*}
$P_s= \sqrt{3}$
solves the KYP inequality~\eqref{eq:LMI} for the EDE system $\wsystemS *$ as in \eqref{eq:simpleciruictEDE}.
In particular, we have that $\rk \mathcal M_s (P_s)= 1 = \rkr \Phi_s(z)$ for the Popov function $\Phi_s(z)\in\K(z)$ of the EDE system. Thus we obtain
\begin{equation*}
  \mathcal M_s (P_s)
    =
    \begin{bmatrix}
    {\sqrt2}	& 0\\
   - \frac{\sqrt3+1}{\sqrt2}	& 1
  \end{bmatrix}
  \begin{bmatrix}
    1 & 0\\
    0 & 0
  \end{bmatrix}
  \begin{bmatrix}
    {\sqrt2}	& - \frac{\sqrt3+1}{\sqrt2}\\
    0	& 1
  \end{bmatrix}  
\end{equation*}
and hence,
\begin{equation*}
 (P_s,\, K_s,\, L_s) = \left(\sqrt 3,\, {\sqrt2},\, -\frac{\sqrt3+1}{\sqrt2}\right)
\end{equation*}
is a solution of the Lur'e equation \eqref{eq:clure}. 
\end{example}
Now we are ready to show that the existence of a solution of the Lur'e equation \eqref{eq:clure} is equivalent to the existence of a certain deflating subspace of the palindromic matrix pencil as in \eqref{eq:palpenc}. This result is the continuous-time analog of \cite[Theorem 11]{reis_lure_2011}.
 \begin{theorem}\label{th:cluredefss}Let $\wsystem[I_n]$ be given and consider the associated palindromic pencil $z\A^*-\A$ as in \eqref{eq:palpenc}. Further, let $q=\rkr \Phi(z)$ and assume that $\rk\,\left[\,I_n-A \enspace -\!B \,\right] =n$.                                                                                                     
 Then the following are equivalent:
 \begin{enumerate}[label=(\alph*)]
  \item \label{it:clurepalpencdefssa} There exists a solution $(X,\,K,\,L)\in\mat{n}{n}\times\mat{q}{n}\times\in\mat{q}{m}$ of the Lur'e equation \eqref{eq:clure}.
  \item \label{it:clurepalpencdefssb} It holds that $\Phi(\eio)\succeq 0$ for all $\omega\in[0,2\pi)$ such that $\det(\eio I_n-A)\neq 0$. Furthermore, there exist matrices $Y_\mu,\,Y_x\in\mat{n}{n+m},\,Y_u\in\mat{m}{n+m}$ and $Z_\mu,\,Z_x\in\mat{n}{n+q},\,Z_u\in\mat{m}{n+q}$ such that for
  \begin{equation*}
   Y=\begin{bmatrix}
      Y_\mu\\
      Y_x\\
      Y_u
     \end{bmatrix},\qquad
   Z=\begin{bmatrix}
      Z_\mu\\
      Z_x\\
      Z_u
     \end{bmatrix}
  \end{equation*}
  the following holds:
  \begin{enumerate}[label=(\roman*)]
   \item \label{it:clurepalpencdefssbi} The matrix 
   \begin{equation*}
    \begin{bmatrix}
     I_n -A 	& -B
    \end{bmatrix}
    \begin{bmatrix}
     Y_x\\
     Y_u
    \end{bmatrix}
   \end{equation*}
 has full row rank $n$.
   \item \label{it:clurepalpencdefssbii} The space $\mathcal Y= \im Y$ is maximally $(\A^*-\A)$-neutral.
   \item \label{it:clurepalpencdefssbiii}There exist $\tilde E,\,\tilde A \in\mat{n+q}{n+m}$ such that $(z\A^*-\A)Y=Z(z\tilde E - \tilde A)$.
  \end{enumerate}

 \end{enumerate}

 \end{theorem}
 \begin{proof}
  Denote by $C\in \mat{n+m}{n}$ and $C_c\in \mat{n+m}{m}$ the right inverse and a basis matrix of the kernel of
  \begin{equation*}
    \begin{bmatrix}
      I_n -A & -B
    \end{bmatrix},
  \end{equation*}
 respectively. Further let 
  \begin{equation*}
    \begin{bmatrix}
     C_1^- \\ C_2^-
    \end{bmatrix}
    :=
    \begin{bmatrix}
     C_{11}^-	&C_{12}^-\\
     C_{21}^-	&C_{22}^-
    \end{bmatrix}
    :=
   \begin{bmatrix}
    C	& C_c
   \end{bmatrix}^{-1},
  \end{equation*}
  where $C_1^-=\left[\,I_n-A \quad -B \,\right]\in \mat{n}{n+m}$, $C_2^-\in \mat{m}{n+m}$, $C_{11}^-=I_n-A\in \mat{n}{n}$, $C_{12}^-=-B\in \mat{n}{m}$, $C_{21}^-\in \mat{m}{n}$, and $C_{22}^-\in \mat{m}{m}$. 
  
  First assume that there exists a solution $(X,\,K,\,L)$ of \eqref{eq:clure}. Then we have $\Phi(\eio)\succeq 0$ for all $\omega\in[0,2\pi)$ such that $\det(\eio I_n-A)\neq 0$ . Set
  \begin{equation}\label{eq:deflssc1}
   Y=
   \begin{bmatrix}
      Y_\mu\\
      Y_x\\
      Y_u
     \end{bmatrix}
     =
   \begin{bmatrix}
    X(A-I_n)	& XB \\
    I_n		& 0  \\
    0		& I_m
   \end{bmatrix}, \qquad
    Z=\begin{bmatrix}
      Z_\mu\\
      Z_x\\
      Z_u
     \end{bmatrix}
     =
     \begin{bmatrix}
      I_n			& 0 \\
      (I_n-A^*)X		& K^*  \\
      -B^*X			& L^*
    \end{bmatrix},
  \end{equation}
  and 
  \begin{equation}\label{eq:deflssc2}
   z\tilde E - \tilde A =
   \begin{bmatrix}
    zI_n-A	& -B		\\
    (z-1)K	&(z-1)L
   \end{bmatrix}.
  \end{equation}
  Property \ref{it:clurepalpencdefssbi} follows, since
  \begin{equation*}
    \rk 
    \begin{bmatrix}
     I_n -A 	& -B
    \end{bmatrix}
    \begin{bmatrix}
     Y_x\\
     Y_u
    \end{bmatrix}
    =
    \rk
    \begin{bmatrix}
     I_n -A 	& -B
    \end{bmatrix}=n
  \end{equation*}
  by assumption. For property \ref{it:clurepalpencdefssbii} we first note that for
    \begin{equation}\label{eq:clurepalpencV}
  V:=
    \begin{bmatrix}
    I_n		& 0	&0 \\
    0		& C_{11}^-	&C_{12}^-\\
    0		& C_{21}^-	&C_{22}^-
   \end{bmatrix}\in\mat{2n+m}{2n+m}
  \end{equation}
  we have
  \begin{equation}\label{eq:clurepalpencE}
    V^{-*}(\A^*-\A)V^{-1}
    =
    \begin{bmatrix}
      0 & I_n & 0 \\
      I_n & 0 & 0\\
      0 & 0 & 0
    \end{bmatrix}
    =:\E.
  \end{equation}
  Then $\im Y$ is maximally $(\A^*-\A)$-neutral if and only if 
  $\im \hat Y$  is maximally $\E$-neutral, where
   \begin{equation*}\label{eq:Yhat}
    \hat Y := V Y
   \begin{bmatrix}
    C	& C_c
  \end{bmatrix}
   =
   \begin{bmatrix}
    -X		& 0 \\
    I_n		& 0  \\
    0		& I_m
   \end{bmatrix}.   
  \end{equation*}
  On the one hand, $\im \hat Y$ is $\E$-neutral, since
  \begin{equation*}
    \hat Y ^*\E \hat Y =
    \begin{bmatrix}
      -X+X& 0\\
      0	 & 0
    \end{bmatrix}= 0.
  \end{equation*}
  On the other hand, we have that $n+m = \rk \hat Y$ and  the rank of every $\E$-neutral space is bounded  from above by $n+m$. Therefore, $\im \hat Y$ is  maximally $\E$-neutral which shows \ref{it:clurepalpencdefssbii}. Finally, we have \ref{it:clurepalpencdefssbiii} by
  \begin{align*}
    &(z\A^*-\A)Y\\
    =&
      \begin{bmatrix}
      zI_n 					- A			& -B			\\
      z \left(A^*X(A-I_n)+Q\right) 		-	 X(A-I_n)-Q	& z(A^*XB + S) - XB-S	\\
      z\left(B^*X(A-I_n)+S^*\right)		- S^*			& z(B^*XB + R) - R
      \end{bmatrix}     
      \\
     = & 
      \begin{bmatrix}
        zI_n 					- A			& -B						\\
        z\left((I_n-A^*)X+K^*K\right)		- (I_n-A^*)XA -K^*K	& zK^*L - (I_n-A^*)XB -K^*L 	\\
        z\left(-B^*X+L^*K\right)		- B^*XA -L^*K		& zL^*L + B^*XB - L^*L 
      \end{bmatrix}\\
      =&
    Z(z\tilde E - \tilde A).
  \end{align*}  
  Now assume that we are in the situation of \ref{it:dluredefssb}. Then by \ref{it:clurepalpencdefssbii}, $\im \hat Y$ is maximally $\E$-neutral for 
  \begin{equation*}
    \hat Y := 
    \begin{bmatrix}
      \hat Y_\mu\\
      \hat Y_x\\
      \hat Y_u
     \end{bmatrix}
    =V Y
   \begin{bmatrix}
    C	& C_c
   \end{bmatrix}
  \end{equation*}
  and $V$ and $\E$ as in \eqref{eq:clurepalpencV} and \eqref{eq:clurepalpencE}. By property \ref{it:clurepalpencdefssbi} we obtain
  \begin{equation*}
    \rk\hat Y_x = \rk 
    \begin{bmatrix}
      I_n	& 0
    \end{bmatrix}
    \begin{bmatrix}
      \hat Y_x\\
      \hat Y_u
    \end{bmatrix}
    = \rk
    \begin{bmatrix}
     I_n -A 	& -B
    \end{bmatrix}
    \begin{bmatrix}
     Y_x\\
     Y_u
    \end{bmatrix}
    \begin{bmatrix}
    C	& C_c
   \end{bmatrix} =n.
  \end{equation*}
 Thus, there exists an invertible $T_1\in\mat{n+m}{n+m}$ such that
  \begin{equation*}
   \hat Y T_1=
     \begin{bmatrix}
      \hat Y_{\mu_1}	&  \hat Y_{\mu_2}\\
       I_n		& 0		\\
      \hat Y_{u_1}	& \hat Y_{u_2}
     \end{bmatrix}.
  \end{equation*}
  Thus, $\hat YT_1$ is still maximally $\E$-neutral and we obtain
  \begin{equation*}
  0=(\hat YT_1)^*\E\hat YT_1=
   \begin{bmatrix}
    I_n	& - \hat Y_{\mu_1}^*	& 0 \\
    0	& - \hat Y_{\mu_2}^*	& 0
   \end{bmatrix}
    \hat Y T_1
   =
      \begin{bmatrix}
    \hat Y_{\mu_1} - \hat Y_{\mu_1}^*	&  \hat Y_{\mu_2} \\
     -\hat  Y_{\mu_2}^*			& 0
   \end{bmatrix};   
  \end{equation*}
   in particular $X:=-\hat Y_{\mu_1}$ is Hermitian.
  Hence, maximal $\E$-neutrality implies full rank of $\hat Y_{u_2}$. Applying another column transformation to $\hat Y$ via an invertible $T_2\in\mat{n+m}{n+m}$ yields
  \begin{equation*}
   \hat Y T_1 T_2=
     \begin{bmatrix}
       -X		& 0		\\
       I_n		& 0		\\
      0		& I_m
     \end{bmatrix}.
  \end{equation*}
  Doing the backtransformation for $Y$ we obtain
  \begin{equation*}
   Y=  V^{-1}\hat Y T_1 T_2 
   \begin{bmatrix}
    C	& C_c
   \end{bmatrix}^{-1}
\hat T,
  \end{equation*}
  where
  \begin{equation*}
   \hat T:=
   \begin{bmatrix}
    C	& C_c
   \end{bmatrix}   
   (T_1 T_2)^{-1}
   \begin{bmatrix}
    C	& C_c
   \end{bmatrix}^{-1}.
  \end{equation*}
 This implies
 \begin{equation*}
  Y\hat T^{-1} = 
  \begin{bmatrix}
    X(A-I_n)	& XB \\
    I_n		& 0  \\
    0		& I_m
   \end{bmatrix}.
 \end{equation*}
We partition $z\hat E - \hat A := (z\tilde E - \tilde A)\hat T^{-1}$ into 
 \begin{equation*}
  z\hat E - \hat A =
  \begin{bmatrix}
   z \hat E_{1} - \hat A_{1} 	& z \hat E_{2} - \hat A_{2} 
  \end{bmatrix},
 \end{equation*}
 where $z \hat E_{1} - \hat A_{1}\in\matz{n+q}{n}$ and $z \hat E_{2} - \hat A_{2}\in\mat{n+q}{m}$.
 Then property \ref{it:clurepalpencdefssbiii} implies
 \begin{align*}
 &
  \begin{bmatrix}
   zI_n -A				& -B \\
   z\left(A^*X(A-I_n) +  Q\right)   - X(A-I_n) -Q	& z(A^*XB +S) - XB - S \\
   z\left(B^*X(A-I_n) + S^*\right) - S^*					& z(B^*XB +R) - R			
  \end{bmatrix}\\
  =&
    \begin{bmatrix}
      Z_\mu\\
      Z_x\\
      Z_u
     \end{bmatrix}
    \begin{bmatrix}
      z \hat E_{1} - \hat A_{1} 	& z \hat E_{2} - \hat A_{2}
    \end{bmatrix},
 \end{align*}
 yielding $I_n = Z_\mu \hat E_1$ and thus $\rk Z_\mu=n$. Therefore, there exists invertible $T_3\in\mat{n+m}{n+m}$ such that $Z_\mu T_3 =\, [ \,I_n\quad 0\,]$. Then for 
 \begin{equation*}
  ZT_3=:
  \begin{bmatrix}
   I_n 		&0\\
   Z_{x_1}	&Z_{x_2}\\
   Z_{u_1}	&Z_{u_2}
  \end{bmatrix},\qquad
  T_3^{-1}
  ( z\hat E - \hat A )=:
   \begin{bmatrix}
   z \hat E_{11} - \hat A_{11} 	& z \hat E_{12} - \hat A_{12}\\ 
   z \hat E_{21} - \hat A_{21} 	& z \hat E_{22} - \hat A_{22}
   \end{bmatrix}
 \end{equation*}
partitioned accordingly,
 we obtain
  \begin{align}\label{eq:cluredefss1}
  \begin{split}
 &
  \begin{bmatrix}
   zI_n -A				& -B \\
   z\left(A^*X(A-I_n) +  Q\right)   - X(A-I_n) -Q	& z(A^*XB +S) - XB - S 			\\
   z\left(B^*X(A-I_n) + S^*\right) - S^*					& z(B^*XB +R) - R			
  \end{bmatrix}\\
  =&
   \begin{bmatrix}
   I_n 		&0\\
   Z_{x_1}	&Z_{x_2}\\
   Z_{u_1}	&Z_{u_2}
  \end{bmatrix}
    \begin{bmatrix}
   z \hat E_{11} - \hat A_{11} 	& z \hat E_{12} - \hat A_{12}\\ 
   z \hat E_{21} - \hat A_{21} 	& z \hat E_{22} - \hat A_{22}
   \end{bmatrix}.
   \end{split}
 \end{align}
 Thus, the first equation gives $\hat E_{11}=I_n$, $\hat A_{11}=A$, $\hat E_{12}=0$, and $\hat A_{12}=B$.
 For $z=1$ we obtain from \eqref{eq:cluredefss1} that
 \begin{equation*}
   \begin{bmatrix}
   I_n - A		& -B \\
   (A^*-I_n)X(A-I_n) 	& (A^*-I_n)XB \\
   B^*X(A-I_n) 		& B^*XB 			
  \end{bmatrix}\\
  =
   \begin{bmatrix}
   I_n 		&0\\
   Z_{x_1}	&Z_{x_2}\\
   Z_{u_1}	&Z_{u_2}
  \end{bmatrix}
    \begin{bmatrix}
   I_n - A 		&  -B\\ 
    \hat E_{21} -\hat A_{21}	&  \hat E_{22} - \hat A_{22}
   \end{bmatrix}.
 \end{equation*}
 Multiplying from the right with $C$ results in
  \begin{equation*}
   \begin{bmatrix}
   (I_n-A^*)X 	& 0 \\
  0		&-B^*X
  \end{bmatrix}
  \begin{bmatrix}
  I_n-A 	& -B\\
  I_n-A		& -B
  \end{bmatrix}
  C
  = 
  \begin{bmatrix}
   Z_{x_1}\\
   Z_{u_1}
  \end{bmatrix}
  +
  \begin{bmatrix}
   Z_{x_2}\\
   Z_{u_2}
  \end{bmatrix}
  \begin{bmatrix}
    \hat E_{21} -\hat A_{21}	&  \hat E_{22} - \hat A_{22}
  \end{bmatrix}
  C
 \end{equation*}
 and thus 
 \begin{equation*}
  \begin{bmatrix}
   Z_{x_1}\\
   Z_{u_1}
  \end{bmatrix}=
  \begin{bmatrix}
   (I_n-A^*)X\\
  -B^*X
  \end{bmatrix} -
  \begin{bmatrix}
   Z_{x_2}\\
   Z_{u_2}
  \end{bmatrix}
  \begin{bmatrix}
    \hat E_{21} -\hat A_{21}	&  \hat E_{22} - \hat A_{22}
  \end{bmatrix}
  C.
 \end{equation*}
  Inserting this relation into \eqref{eq:cluredefss1} for $z=\infty$ gives
  \begin{align*}&
   \begin{bmatrix}
    A^*X(A-I_n) + Q		&  A^*XB + S	\\
    B^*X(A-I_n)	    + S^*	&  B^*XB + R
  \end{bmatrix}\\
  =&
    \begin{bmatrix}
   Z_{x_2}\\
   Z_{u_2}
  \end{bmatrix}
  \begin{bmatrix}
   \hat E_{21}	&  \hat E_{22}
  \end{bmatrix}
  +\left(
  \begin{bmatrix}
   (I_n-A^*)X\\
  -B^*X
  \end{bmatrix} -
  \begin{bmatrix}
   Z_{x_2}\\
   Z_{u_2}
  \end{bmatrix}
  \begin{bmatrix}
    \hat E_{21} -\hat A_{21}	&  \hat E_{22} - \hat A_{22}
  \end{bmatrix}
  C\right)
  \begin{bmatrix}
    I_n & 0
  \end{bmatrix},
  \end{align*}
  which leads to
  \begin{align*}
  \mathcal M (X) &=
   \begin{bmatrix}
    A^*XA-X + Q		&  A^*XB + S	\\
    B^*XA	    + S^*	&  B^*XB + R
  \end{bmatrix}\\
  &=
    \begin{bmatrix}
   Z_{x_2}\\
   Z_{u_2}
  \end{bmatrix}\left(
  \begin{bmatrix}
   \hat E_{21}	&  \hat E_{22}
  \end{bmatrix}
  -
  \begin{bmatrix}
    \hat E_{21} -\hat A_{21}	&  \hat E_{22} - \hat A_{22}
  \end{bmatrix}
  \begin{bmatrix}
   C	& 0
  \end{bmatrix}
  \right).
  \end{align*}
  Thus we have
  \begin{equation}\label{eq:cluredefss2}
  \rk \mathcal M (X)
\le q.
  \end{equation}
  Further, by Lemma \ref{lem:zerokyp}, for $\omega\in[0,2\pi)$ we can rewrite $\Phi(\eio )$  as
  \begin{equation*}
    \Phi(\eio)=\begin{bmatrix}
      (\eio I_n-A)^{-1}B	\\
      I_m		
  \end{bmatrix} ^*
  \mathcal{M}(X)
  \begin{bmatrix}
      (\eio I_n-A)^{-1}B	\\
      I_m		
  \end{bmatrix}
  \end{equation*}
  and thus in \eqref{eq:cluredefss2} we even have equality. Therefore, we can apply Lemma~\ref{lem:minxkl} and hence, we have shown that \ref{it:clurepalpencdefssa} holds.

 \end{proof}

 In the case of a BVD pencil we can prove a similar statement.
 \begin{theorem}
   \label{prop:cluredefss}
   Let $\wsystem[I_n]$ be given and consider the associated BVD pencil $z\E-\A$ as in~\eqref{eq:bvdpenc}. Further, let $q=\rkr \Phi(z)$.
 Then the following are equivalent:
 \begin{enumerate}[label=(\alph*)]
  \item \label{it:clurebvdpencdefssa} There exists a solution $(X,\,K,\,L)\in\mat{n}{n}\times\mat{q}{n}\times\mat{q}{m}$ of \eqref{eq:clure}.
  \item \label{it:clurebvdpencdefssb} It holds that $\Phi(\eio)\succeq 0$ for all $\omega\in[0,2\pi)$ such that $\det(\eio I_n-A)\neq 0$. Furthermore, there exist matrices $Y_\mu,\,Y_x\in\mat{n}{n+m},\,Y_u\in\mat{m}{n+m}$ and $Z_\mu,\,Z_x\in\mat{n}{n+q},\,Z_u\in\mat{m}{n+q}$ such that for
  \begin{equation*}
   Y=\begin{bmatrix}
      Y_\mu\\
      Y_x\\
      Y_u
     \end{bmatrix},\qquad
   Z=\begin{bmatrix}
      Z_\mu\\
      Z_x\\
      Z_u
     \end{bmatrix}
  \end{equation*}
  the following hold:
  
  \begin{enumerate}[label=(\roman*)]
   \item \label{it:clurebvdpencdefssbi}The matrix 
   \begin{equation*}
   Y_x=
    \begin{bmatrix}
     I_n  	& 0
    \end{bmatrix}
    \begin{bmatrix}
     Y_x\\
     Y_u
    \end{bmatrix}
   \end{equation*}
 has full row rank $n$.
   \item \label{it:clurebvdpencdefssbii}  The space $\mathcal Y= \im Y$ is maximally $\E^e$-neutral, where 
   \begin{equation*}
   \E^e:=
    \begin{bmatrix}
     0 	& -I_n	& 0\\
     I_n& 0	& 0\\
     0	& 0	& 0
    \end{bmatrix}.
   \end{equation*}
   \item \label{it:clurebvdpencdefssbiii} There exist $\tilde E,\,\tilde A \in\mat{n+q}{n+m}$ such that $(z\E-\A)Y=Z(z\tilde E - \tilde A)$.
  \end{enumerate}

 \end{enumerate}

  \end{theorem}
\begin{proof}
  See \cite[Theorem 5.5]{bankmann_linear-quadratic_2015}.
\end{proof}

\subsection{Implicit Difference Equations}
In this section we generalize the results from the previous section to implicit difference equations. As for the KYP inequality we need relations between the Lur'e equation \eqref{eq:dlure} corresponding to the original system and the associated equation corresponding to the feedback equivalent system $\wsystemF$ as in \eqref{eq:fbsys}. These findings are related to the results in \cite{reis_kalmanyakubovichpopov_2015} in the continuous-time case.
\begin{lemma}\label{lem:dlure2flure}
 Let $\wsystem$ be given and $q=\rkr\Phi(z)$.
 Then  $(X\,,K\,,L)\in\mat{n}{n}\times\mat{q}{n}\times\mat{q}{m}$ is a solution of \eqref{eq:dlure} if and only if
 \begin{equation}\label{eq:flureequivalence}
  (X_F,\,K_F,\,L_F) := (W^{-*}XW^{-1},\,KT+LFT,\,L)
 \end{equation}
 is a solution of \eqref{eq:dlure} associated to the feedback system
 \begin{equation*}
   \wsystemF   
 \end{equation*}
 as in \eqref{eq:fbsys}, \ie
 \begin{equation*}
   \mathcal{M}_F(X_F)=
  \begin{bmatrix}
  A_F^*X_FA_F-E_F^*X_FE_F + Q_F		&  A_F^*X_FB_F + S_F	\\
  B^*_FX_FA_F	    + S_F^*	&  B_F^*X_FB_F + R_F
 \end{bmatrix} =_{\cV F}
 \begin{bmatrix}
  K_F^*\\
  L_F^*
 \end{bmatrix}
 \begin{bmatrix}
  K_F & L_F
 \end{bmatrix}.
 \end{equation*}
\end{lemma}
\begin{proof} First note that 
for 
\begin{equation*}
\mathcal{T}_F =
    \begin{bmatrix}
      T	& 0 \\
      FT & I_m
    \end{bmatrix}
\end{equation*}
we have
\begin{equation*}\label{eq:lurerhseq}
\begin{bmatrix}
  K	& L
\end{bmatrix}
\mathcal{T}_F
=
\begin{bmatrix}
  KT + LFT	& L
\end{bmatrix}  .
\end{equation*}
In addition, by  Proposition~\ref{prop:popov}\ref{it:phifphia} we obtain that $q=\rkr\Phi(z)=\rkr\Phi_F(z)$. Thus,  Lemma~\ref{lem:equivalence} %
 immediately yields the assertion. 
\end{proof}
Moreover, we now characterize the connection between the Lur'e equation \eqref{eq:dlure} corresponding to the system $\wsystemF$ in feedback equivalence form  as in \eqref{eq:fbsys} and the Lur'e equation \eqref{eq:clure} corresponding to the associated EDE part as in \eqref{eq:EDEsys}.
\begin{lemma}\label{lem:dlure2clure}
 Let $\wsystem$ be given and consider the system $\wsystemF$ as in \eqref{eq:fbsys} in feedback equivalence form \eqref{eq:fef}. Further, consider $(X_F\,,K_F,\,L_F)$ as in \eqref{eq:flureequivalence} partitioned according to the block structure of the feedback equivalence form. 
 
 Then with $q=\rkr\Phi(z)$ we have that
 $(X\,,K\,,L)\in\mat{n}{n}\times\mat{q}{n}\times\mat{q}{m}$ is a solution of \eqref{eq:dlure} if and only if $(X_{11},\,K_1,\,L-K_2B_2)\in\mat{n_1}{n_1}\times\mat{q}{n_1}\times\mat{q}{m}$ is a solution of \eqref{eq:clure} for the EDE system 
 \begin{equation*}
  \wsystemS 
  \end{equation*}
  as in \eqref{eq:EDEsys}.
\end{lemma}
\begin{proof}
	See \cite[Lemma 5.7]{bankmann_linear-quadratic_2015}.
\end{proof}

For the rest of this chapter we assume that $\wsystem$ is I-controllable, \ie there exists a feedback $F\in\mat{m}{n}$ such that the system 
\begin{equation*}
 \wsystemF
\end{equation*}
as in \eqref{eq:fbsys}  is in feedback equivalence form such that $n_3=0$. This is justified by the fact that the subsystem described by $\system[E_{33}][I_{n_3}][0][m][n_3]$ obtained from the feedback equivalence form \eqref{eq:fef} has only the zero solution and thus does not contribute to the dynamics of the system. Indeed, in the proofs of Lemma~\ref{lem:kypreduction} and Lemma~\ref{lem:dlure2clure} the parts of $\wsystemF*$ corresponding to the last $n_3$ variables do not contribute to the analysis. 
The following proposition makes this precise, using the same projection ansatz as in {\cite[Theorem 5.9]{reis_kalmanyakubovichpopov_2015}}. 
\begin{proposition}
 Let $\wsystem$ be given and consider the system $\wsystemF$ as in \eqref{eq:fbsys} in feedback equivalence form \eqref{eq:fef}. Further, let $q=\rkr \Phi(z)$.
 Define the projector
 \begin{equation*}\label{eq:projector}
\Pi:= W^{-1}
\begin{bmatrix}
  I_{n_1} & 0 & 0\\
  0	&0 &0\\
  0 & 0 &0
\end{bmatrix}
W
\in\mat{n}{n}.
\end{equation*}
Then we have 
 \begin{equation}\label{eq:projectorspace}
  \im \Pi = E\Vshift
 \end{equation}
and the following statements hold:
\begin{enumerate}[label=(\alph*)]
  \item \label{it:projectora} The system $\system[\Pi E]$ is I-controllable and 
  \begin{equation*}
    \behavior*=\behavior*[\Pi E][\Pi E][\Pi E].
  \end{equation*}
  In particular, the system space of $\system[\Pi E]*$ is $\cV$.
  \item \label{it:projectorb}  There exists a solution $P\in\mat{n}{n}$ of the KYP inequality \eqref{eq:LMI}, \ie $\mathcal M (P)\succeq_{\cV}0$, if and only if $\mathcal M_\Pi (P)\succeq_{\cV}0$,
  where $\mathcal{M}_\Pi(P)$ is the matrix in \eqref{eq:LMI} corresponding to the system $\system[\Pi E]$.
  \item \label{it:projectorc} There exists a solution $(X,\,K,\,L)\in\mat{n}{n}\times\mat{q}{n}\times\mat{q}{m}$ of the Lur'e equation \eqref{eq:dlure} if and only if $(X,\,K,\,L)$ also fulfills the Lur'e equation \eqref{eq:dlure} corresponding to the  system $\system[\Pi E]$. 
  
\end{enumerate}
\end{proposition}
\begin{proof}
 Part~\ref{it:projectora} and \eqref{eq:projectorspace} follow with the algebraic manipulations mentioned in the proof of  {\cite[Theorem 5.9]{reis_kalmanyakubovichpopov_2015}}.
 
 Now set 
 \begin{equation*}
   \Pi_F:= W\Pi W^{-1}=
   \begin{bmatrix}
  I_{n_1} & 0 & 0\\
  0	&0 &0\\
  0 & 0 &0
\end{bmatrix}.
 \end{equation*}
 For parts \ref{it:projectorb} and  \ref{it:projectorc} note that the system $\system[\Pi_FE_F][A_F][B_F]$ is in feedback equivalence form \eqref{eq:fef} where compared to $\system[E_F][A_F][B_F]$ the matrices $E_{23}$ and $E_{33}$ are set to zero. Looking carefully at the proofs of Lemma~\ref{lem:kypreduction} and Lemma~\ref{lem:dlure2clure} we see that these matrices have no effect in the respective results and thus the assertion follows. 
\end{proof}

As a next step, we  perform transformations of the palindromic or BVD pencils  corresponding to the system $\wsystemF$ in feedback equivalence form  as in \eqref{eq:fbsys} such that we obtain the respective palindromic or BVD pencils corresponding to the EDE system $\wsystemS$ as in \eqref{eq:EDEsys} in the first diagonal block of the transformed pencil. 
 \begin{lemma}\label{lem:dpal2cpal} Let $\wsystemF$ as in \eqref{eq:fbsys} be given in feedback equivalence form \eqref{eq:fef} such that $n_3=0$. Further, let the corresponding palindromic pencil $z\A_F^*-\A_F$ as in  \eqref{eq:palpenc} be given. Denote by $z\A_s^*-\A_s$ the pencil corresponding to the EDE system $\wsystemS$ as in \eqref{eq:EDEsys}, \ie
 \begin{equation}\label{eq:EDEpalpenc}
   z\A_s^*-\A_s=
      \begin{bmatrix}
    0			& zI_{n_1}-A_s		& -B_s 		\\
    zA_s^*-I_{n_1}	& (z-1)Q_s	& (z-1)S_s 	\\
	zB_s^*		& (z-1)S_s^*	& (z-1)R_s
  \end{bmatrix}\in\matz{2n_1+m}{2n_1+m}.
 \end{equation}
 Then there exists an invertible $\hat U\in\mat{2n+m}{2n+m}$ such that 
  \begin{equation*}\label{eq:dpal2cpalresult}
   \hat U^* (z\A_F^*-\A_F) \hat  U = \left[
   \begin{array}{@{}c|cc@{} }
    z\A_s^*-\A_s	& zD	& 0	\shline{1.6}
    -D^*		& 0	& zI_{n_2}	\\	
    0			& -I_{n_2}  & 0
   \end{array}\right], 
  \end{equation*}
  where
  \begin{equation*}\label{eq:DAFAS}
  D=
    \begin{bmatrix}
      0			\\
      Q_{12}		\\
      S_2^*-B_2^*Q_{22} 	      
    \end{bmatrix}
  \end{equation*}
  and
  \begin{align}\label{eq:UAFAS}
  \begin{split}
   \hat  U&=
    \begin{bmatrix}
    I_{n_1}		& 0		& 0			&  0		& 0		\\
    0			& -Q_{12}^*	& -S_2 +Q_{22}B_2	& -Q_{22}	& I_{n_2}	\\	
    0			& I_{n_1}	& 0			&  0		& 0		\\
    0			& 0		& -B_2			& I_{n_2}	& 0		\\
    0			& 0		& I_m			& 0             & 0		
   \end{bmatrix}\\
   &=
   \underbrace{
   \begin{bmatrix}
    I_{n_1}		& 0		& 0		&  0	& 0			\\
    0			& 0		& 0		&  0	& I_{n_2}		\\	
    0			& I_{n_1}	& 0		&  0	& 0			\\
    0			& 0		& 0		&I_{n_2}& 0			\\
    0			& 0		& I_m		& 0     & 0		
   \end{bmatrix}}_{=:\tilde P}
   \underbrace{
   \begin{bmatrix}
    I_{n_1}		& 0		& 0		&  0		& 0				\\
    0			& I_{n_1}	& 0		&  0		& 0				\\
    0			& 0		& I_m		& 0      	& 0			        \\
    0			& 0		& -B_2		& I_{n_2}	& 0				\\
    0			& -Q_{12}^*	& -S_2+Q_{22}B_2& -Q_{22}	& I_{n_2}				
   \end{bmatrix}.}_{=:\tilde U}
   \end{split}
  \end{align}

 \end{lemma}
 \begin{proof}
   See \cite[Lemma 5.9]{bankmann_linear-quadratic_2015}.
 \end{proof}

 \begin{corollary}\label{cor:dbvd2cbvd} Let $\wsystemF$ as in \eqref{eq:fbsys} be given in feedback equivalence form \eqref{eq:fef} such that $n_3=0$. Further, let the corresponding BVD pencil $z\E_F-\A_F$ as in \eqref{eq:bvdpenc} be given. Denote by $z\E_s-\A_s$ the BVD pencil corresponding to the EDE system $\wsystemS$ as in \eqref{eq:EDEsys}, \ie
 \begin{equation*}\label{eq:EDEbvdpenc}
   z\E_s-\A_s=
      \begin{bmatrix}
    0			& zI_{n_1}-A_s	& -B_s 		\\
    zA_s^*-I_{n_1}	& -Q_s		& -S_s	 	\\
	zB_s^*		& -S_s^*	& -R_s
  \end{bmatrix}\in\matz{2n_1+m}{2n_1+m}.
 \end{equation*}
Then there exist invertible $\hat U$ and $\check U$ such that
  \begin{equation*}\label{eq:dbvd2cbvd}
   \hat U^* (z\E-\A) \check U = \left[
   \begin{array}{@{}c|cc@{} }
    z\E_s-\A_s		& 0		& 0	\shline{1.6}
    -D^*		& 0		& zI_{n_2}	\\	
    0			& -I_{n_2}  	& 0
   \end{array}\right].
  \end{equation*}
 \end{corollary}
Now we are able to prove a generalization of Theorem~\ref{th:cluredefss}. This result is related to the result in \cite[Theorem 6.2]{reis_kalmanyakubovichpopov_2015} in the continuous-time case.

 \begin{theorem}\label{thm:dluredfss}  Let $\wsystem$ be I-controllable. Further, let the corresponding palindromic pencil $z\A^*-\A$ \eqref{eq:palpenc} be given.  In addition, let $q=\rkr\Phi(z)$ and assume that $\rk\,\left[
                                                                                                                    \,E-A \quad B \,\right]
                                                                                                                   =n$.
 Then the following are equivalent:
 \begin{enumerate}[label=(\alph*)]
  \item \label{it:dluredefssa} There exists a solution $(X,\,K,\,L)\in\mat{n}{n}\times\mat{q}{n}\times\mat{q}{m}$ of the Lur'e equation \eqref{eq:dlure}.
  \item \label{it:dluredefssb} It holds that $\Phi(\eio)\succeq 0$ for all $\omega\in[0,2\pi)$ such that $\det(\eio E-A)\neq 0$. Furthermore, there exist matrices $Y_\mu,\,Y_x\in\mat{n}{n+m},\,Y_u\in\mat{m}{n+m}$ and $Z_\mu,\,Z_x\in\mat{n}{n+q},\,Z_u\in\mat{m}{n+q}$ such that for
  \begin{equation*}
   Y=\begin{bmatrix}
      Y_\mu\\
      Y_x\\
      Y_u
     \end{bmatrix},\qquad
   Z=\begin{bmatrix}
      Z_\mu\\
      Z_x\\
      Z_u
     \end{bmatrix}
  \end{equation*}
  the following holds:
  \begin{enumerate}[label=(\roman*)]
   \item \label{it:dluredefssbi}The matrix 
   \begin{equation*}
    \begin{bmatrix}
     E-A	& -B
    \end{bmatrix}
    \begin{bmatrix}
     Y_x\\
     Y_u
    \end{bmatrix}
   \end{equation*}
 has  rank $n_1$.
   \item \label{it:dluredefssbii} The space $\mathcal Y= \im Y$ is of dimension $n+m$ and $(\A^*-\A)$-neutral.
   \item \label{it:dluredefssbiii} It holds that 
   \begin{equation*}
      \cV=\im
      \begin{bmatrix}
        Y_x\\
        Y_u
      \end{bmatrix}.
   \end{equation*}

   \item \label{it:dluredefssbiv} There exist $\tilde E,\,\tilde A \in\mat{n+q}{n+m}$ such that $(z\A^*-\A)Y=Z(z\tilde E - \tilde A)$.
  \end{enumerate}

 \end{enumerate}

 \end{theorem}
 \begin{proof} First we show that the statement is invariant under feedback transformations. Therefore, assume we have given the system $\wsystemF$ in feedback equivalence form as in \eqref{eq:fbsys} such that $n_3=0$ with corresponding transformation matrices $W$ and $\mathcal T_F$ and corresponding palindromic pencil $z\A_F^* -\A_F$ as in \eqref{eq:palpenc}. Then by Lemma~\ref{lem:dlure2flure}, part \ref{it:dluredefssa} is equivalent to the existence of a solution  $(X_F,\,K_F,\,L_F)$ as in \eqref{eq:flureequivalence} of the Lur'e equation \eqref{eq:dlure} corresponding to $\wsystemF*$. 
 
 To show the equivalence of statement \ref{it:dluredefssb} to according statements for the system in feedback equivalence form  let
 \begin{equation*}
 U_F:=
  \begin{bmatrix}
   W^*		& 0 		& 0	 \\
   0		& T	 	& 0	 \\
   0		& FT 		& I_m 
  \end{bmatrix}\in\mat{2n+m}{2n+m}
 \end{equation*}
  and set
 \begin{equation}\label{eq:transformdss2FB}  
  Y_F:=\begin{bmatrix}
       Y_{\mu,F}\\
      Y_{x,F}\\
      Y_{u,F}
     \end{bmatrix}
     := U_F^{-1} Y 
     ,\quad
   Z_F:=\begin{bmatrix}
      Z_{\mu,F}\\
       Z_{x,F}\\
      Z_{u,F}
     \end{bmatrix}
     = U_F^{*} Z .%
 \end{equation}
 Then $ \A_F = U_F^{*}\A U_F$  and
  statement \ref{it:dluredefssbi} is equivalent to 
  \begin{align*}
   &\rk 
   \begin{bmatrix}
     E_F-A_F	& -B_F
    \end{bmatrix}
    \begin{bmatrix}
     Y_{x,F}\\
     Y_{u,F}
    \end{bmatrix}\\
    =\,&\rk
    \begin{bmatrix}
     E-A	& -B
    \end{bmatrix}
    \begin{bmatrix}
     T	 	& 0	 \\
    FT 	& I_m 
    \end{bmatrix}
    \begin{bmatrix}
     T	 	& 0	 \\
    FT 		& I_m 
    \end{bmatrix}^{-1}
    \begin{bmatrix}
     Y_x\\
     Y_u
    \end{bmatrix}=n_1.
  \end{align*}
  Furthermore, we have that $\rk Y_F = \rk Y =n+m$ and  $\im Y$ is $(\A^*-\A)$-neutral if and only if  $\im Y_F$ is $(\A_F^*-\A_F)$-neutral.  In addition, by Proposition~\ref{prop:sysspace}\ref{it:sysspacea} we obtain that \ref{it:dluredefssbiii} is equivalent to
  \begin{equation*}
    \cV F=\im
    \begin{bmatrix}
     Y_{x,F}\\
     Y_{u,F}
    \end{bmatrix}.
  \end{equation*}
  Finally, statement \ref{it:dluredefssbiv} is equivalent to  $(z\A_F^*-\A_F)Y_F=Z_F(z\tilde E - \tilde A)$ by the definition of $\A_F,\,Y_F\,$ and $Z_F$.
 Hence, we have shown that it is sufficient to prove the equivalence between \ref{it:dluredefssa} and \ref{it:dluredefssb} for the system $\wsystemF*$ in feedback equivalence form.
             
 Now we show that statement \ref{it:dluredefssb} follows from statement \ref{it:dluredefssa}. From Lemma \ref{lem:dlure2clure} we infer that $(X_{11},\,K_1,\,L-K_2B_2)$ is a solution of the Lur'e equation \eqref{eq:clure} for the EDE system 
 \begin{equation*}
 \wsystemS
  \end{equation*}
  as in \eqref{eq:EDEsys}.
  By denoting the corresponding palindromic pencil arising in the optimal control problem by $z\A_s^* - \A_s$ as in \eqref{eq:EDEpalpenc}, Theorem~\ref{th:cluredefss} implies the existence of 
    \begin{equation*}
    Y_s=
   \begin{bmatrix}
    X_{11}(A_{11}-I_{n_1})	& X_{11}B_1 \\
    I_{n_1}			& 0  	\\
    0				& I_m
   \end{bmatrix}, \qquad
    Z_s =
     \begin{bmatrix}
      I_{n_1}			& 0 \\
      (I_{n_1}-A_{11}^*)X_{11}	& K_1^*  \\
      -B_1^*X_{11}		& (L-K_2B_2)^*
    \end{bmatrix}
  \end{equation*}
  as in \eqref{eq:deflssc1} and 
  \begin{equation*}
z\hat E_s - \hat A_s=
  \begin{bmatrix}
     zI_{n_1}-A_{11}	& -B_1				\\
    (z-1)K_1		&(z-1)(L-K_2B_2)		  
  \end{bmatrix}
\end{equation*}
as in \eqref{eq:deflssc2} such that $(z\A_s^*-\A_s)Y_s=Z_s(z\hat E_s - \hat A_s)$. Note that as in Theorem~\ref{th:cluredefss}, $\im Y_s$ is maximally $(\A_s^*-\A_s)$-neutral.
  
  From Lemma~\ref{lem:dpal2cpal} we obtain an invertible transformation matrix $\hat U\in\mat{2n+m}{2n+m}$ as in \eqref{eq:UAFAS} such that  
  \begin{equation}\label{eq:dluredefsspalpenc2}
   z\hat\A^*-\hat\A := \hat U^* (z\A_F^*-\A_F) \hat U = \left[
   \begin{array}{@{}c|cc@{} }
    z\A_s^*-\A_s	& zD	& 0	\shline{1.6}
    -D^*		& 0	& zI_{n_2}	\\	
    0			& -I_{n_2}  & 0
   \end{array}\right]
  \end{equation}
  with
    \begin{equation*}\
  D=
    \begin{bmatrix}
      0			\\
      Q_{12}		\\
      S_2^*-B_2^*Q_{22} 	      
    \end{bmatrix}\in\mat{2n_1+m}{n_2}.
  \end{equation*}
  By inspecting the proof of Theorem~\ref{th:cluredefss} we find that
  \begin{equation}\label{eq:dluredefss1}
   (z\hat\A^*-\hat\A) \hat Y = \hat Z (z\hat E - \hat A),
  \end{equation}
  where 
 \begin{equation*}
  \hat Y = \left[
   \begin{array}{@{}c|c@{} }
     Y_s	& 0 \shline{1.6}
     0		& 0		\\		
     0		& I_{n_2}	
   \end{array}\right], \qquad
   \hat Z =\left[
      \begin{array}{@{}c|c@{} }
      Z_s				& 0\shline{1.6}
      0					&I_{n_2}\\
      0					& 0	%
    \end{array}\right],\qquad
    z\hat E - \hat A=\left[
      \begin{array}{@{}c|c@{} }
      z\hat E_s-\hat A_s				& 0\shline{1.6}
      -D^*Y_s				& zI_{n_2}
    \end{array}\right].
  \end{equation*}

  Thus we have
  \begin{align*}
    \hat Y^*(\hat\A^*-\hat\A)\hat Y &= 
  \left[
   \begin{array}{@{}c|cc@{} }
     Y_s^*	& 0 		& 0 \shline{1.6}		
     0		& 0		& I_{n_2}	
   \end{array}\right]
   \left[
   \begin{array}{@{}c|cc@{} }
    \A_s^*-\A_s		& D	& 0	\shline{1.6}
    -D^*		& 0	& I_{n_2}	\\	
    0			& -I_{n_2}  & 0
   \end{array}\right]
    \left[
   \begin{array}{@{}c|c@{} }
     Y_s	& 0 \shline{1.6}
     0		& 0		\\		
     0		& I_{n_2}	
   \end{array}\right]=0,
  \end{align*}
  and we obtain that $\im \hat Y$ is $n+m$ dimensional and $(\hat \A^*-\hat \A)$-neutral. Set
  \begin{equation*}
    \hat V =
    \begin{bmatrix}
      I_{n_1}	& 0			& 0			\\
      0		& 0			& I_m			\\
      Q_{12}^*	& -I_{n_2}		& -B_2 + S_2-Q_{22}B_2	
    \end{bmatrix}.
  \end{equation*}
  Transforming the quantities in \eqref{eq:dluredefss1} to feedback equivalence form \eqref{eq:fbsys} we obtain
   \begin{equation*}\label{eq:dluredefss2}
   (z\A_F^*-\A_F)  Y_F\hat V =  Z_F (z\tilde E - \tilde A),
  \end{equation*}
  where
\begin{equation}\label{eq:deflss}
  Y_F \hat V=
  \begin{bmatrix}
       Y_{\mu,F}\\
      Y_{x,F}\\
      Y_{u,F}
     \end{bmatrix}\hat V
     :=
       \begin{bmatrix}
      Y_{\mu_1,F}\\
      Y_{\mu_2,F}\\
      Y_{x_1,F}\\
      Y_{x_2,F}\\
      Y_{u,F}
     \end{bmatrix}:=
  \hat U\hat Y\hat V=\left[
   \begin{array}{@{}ccc@{} }
    X_{11}(A_{11}-I_{n_1})	&0		& X_{11}B_1				\\
    0				& -I_{n_2}	& -B_2					\shline{1.6}
    I_{n_1}			& 0		& 0  					\\
    0				& 0		& -B_2					\\
    0				& 0		& I_m			
   \end{array}\right],
  \end{equation}
  $Z_F:= \hat U^{-*}\hat Z$, and
  \begin{equation*}\label{eq:ztildeE-tildeA}
    (z\tilde E - \tilde A):=(z\hat E - \hat A)\hat V=\left[
    \begin{array}{@{}ccc@{} }
     zI_{n_1}-A_{11}			& 0		& -B_1					\\
    (z-1)K_1				& 0		& (z-1)(L-K_2B_2)			\\		  
      (z-1)Q_{12}^*			& -zI_{n_2}	& -zB_2 +(z-1) (S_2 - Q_{22}B_2 )
    \end{array}\right].
  \end{equation*}

 Then we obtain property \ref{it:dluredefssbi} by
  \begin{align*}
    \rk 
   \begin{bmatrix}
     E_F-A_F	& -B_F
    \end{bmatrix}
    \begin{bmatrix}
     Y_{x,F}\\
     Y_{u,F}
    \end{bmatrix}
    &=
    \rk
    \begin{bmatrix}
      I_{n_1}-A_{11}	& 0		& -B_1\\
      0			&-I_{n_2}	& -B_2
    \end{bmatrix}
    \begin{bmatrix}
        I_{n_1}			& 0  		& 0		\\
    0				& -B_2		& 0		\\
    0				& I_m		& 0
    \end{bmatrix}\\
    &=\rk
    \begin{bmatrix}
       I_{n_1}-A_{11}	& -B_1 & 0\\
       0		& 0 	& 0
    \end{bmatrix}
    =n_1.
  \end{align*}
  Property \ref{it:dluredefssbii} follows from the fact that $\im \hat Y$ is $n+m$ dimensional and $(\hat \A^*-\hat \A)$-neutral. Furthermore, by Proposition~\ref{prop:sysspace}\ref{it:sysspacea} we have property \ref{it:dluredefssbiii}. Altogether, this shows statement \ref{it:dluredefssb}.
  
  Now assume that  \ref{it:dluredefssb} holds for the system $\wsystemF*$ in feedback equivalence form, \ie properties \ref{it:dluredefssbi}--\ref{it:dluredefssbiv} are satisfied. From these properties we construct a deflating subspace for the palindromic pencil $z\A_s^*-\A_s$ such that we can apply Theorem~\ref{th:cluredefss}. Therefore, with the help of by Proposition~\ref{prop:sysspace}\ref{it:sysspacea} and \ref{it:dluredefssbiii} we accordingly partition $Y_F$  into
  \begin{equation*}
  Y_F =
  \begin{bmatrix}
       Y_{\mu,F}\\
      Y_{x,F}\\
      Y_{u,F}
     \end{bmatrix}
     :=
       \begin{bmatrix}
      Y_{\mu_1,F}\\
      Y_{\mu_2,F}\\
      Y_{x_1,F}\\
      Y_{x_2,F}\\
      Y_{u,F}
     \end{bmatrix}:=
            \begin{bmatrix}
            Y_{\mu_{11},F} & Y_{\mu_{12},F} & Y_{\mu_{13},F} \\
            Y_{\mu_{21},F} & Y_{\mu_{22},F} & Y_{\mu_{23},F}\\
            I_{n_1} & 0 & 0\\
            0 & -B_2 & 0\\
            0 & I_m & 0
            \end{bmatrix}
     \end{equation*}
     and denote by $\hat \A$ and $\hat U$ the matrices %
     we obtain from Lemma~\ref{lem:dpal2cpal} such that \eqref{eq:dluredefsspalpenc2} holds.
     Then, for $\hat Y := \hat U^{-1} Y_F$ we have

  \begin{equation*} 
  \hat Y= \hat U^{-1} Y_F =
    \begin{bmatrix}
   I_{n_1}	& 0		& 0				&  0		& 0					\\
    0		& 0		& I_{n_1}			&  0		& 0					\\
    0		& 0		& 0				& 0      	& I_m					\\
    0		& 0		& 0				& I_{n_2}	& B_2					\\	
    0		& I_{n_2}	& Q_{12}^*			& Q_{22}	& S_2						
   \end{bmatrix}
     \begin{bmatrix}
      Y_{\mu_1,F}\\
      Y_{\mu_2,F}\\
      Y_{x_1,F}\\
      Y_{x_2,F}\\
      Y_{u,F}
     \end{bmatrix}
     =
    \begin{bmatrix}
      Y_{\mu_1,F}\\
      Y_{x_1,F}\\
      Y_{u,F}\\
      0\\
     \hat Y_{\mu_2,F}    
     \end{bmatrix}
  \end{equation*}
  for some
  \begin{equation*}
  \hat Y_{\mu_2,F}:=
  \begin{bmatrix}
  \hat Y_{\mu_{21},F} & \hat Y_{\mu_{22},F} & \hat Y_{\mu_{23},F}
  \end{bmatrix}.
  \end{equation*}
   Thus, $\im \hat Y$ is $n+m$ dimensional by property \ref{it:dluredefssbii} and $(\hat \A^*-\hat \A)$-neutral.
   In particular we obtain 
   \begin{align*}
   0&=
   \begin{bmatrix}
    Y_{\mu_{13},F}^*& 0 & 0 & 0 & \hat Y_{\mu_{23},F}^*
   \end{bmatrix}
   \left[
   \begin{array}{@{}ccc|cc@{} }
   0				& I_{n_1}-A_{11}	& -B_1 	 					& 0					& 0		\\
   A_{11}^*-I_{n_1}	& 0				& 0							& Q_{12} 			& 0 	\\
   B_1^*			& 0				& 0							& S_2^* -B_2^*Q_{22}& 0		\shline{1.6}
   0				& -Q_{12}^*		& Q_{22}B_2^* - S_2			& 0					& I_{n_2}\\	
   0				& 0				& 0							&-I_{n_2} 			& 0
   \end{array}\right]
            \begin{bmatrix}
            Y_{\mu_{11},F} & Y_{\mu_{12},F}  \\
            I_{n_1} & 0 \\
            0 & I_m \\
            0&0\\
                        \hat Y_{\mu_{21},F} & \hat Y_{\mu_{22},F} 
            \end{bmatrix}\\
   &=
 Y_{\mu_{13},F}^*
 \begin{bmatrix}
 I_{n_1} - A_s & - B_s
 \end{bmatrix},
   \end{align*}
   and hence $Y_{\mu_{13},F}=0$.
This shows that the matrix
   \begin{equation*}\label{eq:thmdlurefullrank}
     \begin{bmatrix}
      Y_{\mu_1,F}\\
      Y_{x_1,F}\\
       Y_{u,F}
     \end{bmatrix} 
   \end{equation*}
   has full column rank and thus its image is also maximally $(\A_s^*-\A_s)$-neutral. This, together with the fact that $\rk \hat Y = n+m$, allows us to perform a column transformation of $\hat Y$ via $T_1\in\mat{n+m}{n+m}$ such that
  \begin{equation*}
    \begin{bmatrix}
      Y_{\mu_1,F}\\
      Y_{x_1,F}\\
      Y_{u,F}\\
     \hat Y_{\mu_2,F}\\
     \hat Y_{x_2,F}
     \end{bmatrix}
     T_1=
     \hat Y T_1=
      \left[
   \begin{array}{@{}cc|c@{} }
    X_{11}(A_{11}-I_{n_1})	& X_{11}B_1 	&	0		\\
     I_{n_1}			& 0 		& 0			\\
     0				& I_m		& 0			\shline{1.6}		
     0				& 0		& 0		\\
     0				& 0		& I_{n_2}		
   \end{array}\right]
  \end{equation*}
  with some Hermitian $X_{11}\in\mat{n_1}{n_1}$, similar as in the proof of Theorem~\ref{th:cluredefss}.
  Set
    \begin{equation*}
  Y_s:=
    \begin{bmatrix}
     X_{11}(A_{11}-I_{n_1})	& X_{11}B_1 				\\
     I_{n_1}			& 0 					\\
     0				& I_m					
    \end{bmatrix}.
  \end{equation*}

From property \ref{it:dluredefssbiv} we obtain
\begin{align}\label{eq:thmdluredeflss}
\begin{split} 
  \left[
   \begin{array}{@{}c|cc@{} }
    z\A_s^*-\A_s	& zD	& 0	\shline{1.6}
    -D^*		& 0	& zI_{n_2}	\\	
    0			& -I_{n_2}  & 0
   \end{array}\right]
   \left[
   \begin{array}{@{}c|c@{} }
     Y_s	& 0 \shline{1.6}
     0		& 0\\		
     0		& I_{n_2}	
   \end{array}\right]&=
   \begin{bmatrix}
     Z_{11}	& Z_{12}	\\
     Z_{21}	& Z_{22}	
   \end{bmatrix}
   \begin{bmatrix}
     z\hat E_{11}-\hat A_{11}	& z\hat E_{12}-\hat A_{12}\\
     z\hat E_{21}-\hat A_{21}	& z\hat E_{22}-\hat A_{22}
   \end{bmatrix},
   \end{split}
\end{align}
where $Z_{11}\in\mat{2n_1+m}{n_1+q}$, $Z_{12}\in\mat{2n_1+m}{n_2}$, $Z_{21}\in\mat{2n_2}{n_1+q}$, $Z_{22}\in\mat{2n_2}{n_2}$, $z\hat E_{11}-\hat A_{11}\in\matz{n_1+q}{n_1+m}$, $z\hat E_{12}-\hat A_{12}\in\matz{n_1+q}{n_2}$, $z\hat E_{21}-\hat A_{21}\in\matz{n_2}{n_1+m}$, and $z\hat E_{22}-\hat A_{22}\in\matz{n_2}{n_2}$. From the last block column and block row of \eqref{eq:thmdluredeflss}  we obtain
\begin{equation}\label{eq:thmdlurefullnormalrank}
     \begin{bmatrix}
     zI_{n_2}\\
     0		
   \end{bmatrix}=
    \begin{bmatrix}
     Z_{21}	& Z_{22}	
   \end{bmatrix}
   \begin{bmatrix}
      z\hat E_{12}-\hat A_{12}\\
      z\hat E_{22}-\hat A_{22}
   \end{bmatrix}
\end{equation}
and thus we have 
\begin{equation*}
  \rk 
  \begin{bmatrix}
    Z_{21}	&Z_{22}
  \end{bmatrix}=n_2.
\end{equation*}
Therefore, we can determine a transformation matrix $T_2\in\mat{n+q}{n+q}$ such that
\begin{equation*}
  \rk 
  \begin{bmatrix}
    Z_{21}	&Z_{22}
  \end{bmatrix}T_2=
   \begin{bmatrix}
    0	& \tilde Z_{22}
  \end{bmatrix}
\end{equation*}
for some $\tilde Z_{22}\in\mat{2n_2}{n_2}$. Set 
\begin{equation*}
  \begin{bmatrix}
     \tilde Z_{11}	& \tilde Z_{12}	\\
     0			& \tilde Z_{22}	
   \end{bmatrix}:=
  \begin{bmatrix}
     Z_{11}	& Z_{12}	\\
     Z_{21}	& Z_{22}	
   \end{bmatrix}T_2
 \end{equation*}
 and
 \begin{equation*}
    \begin{bmatrix}
     z\tilde E_{11} - \tilde A_{11}	& z \tilde E_{12}-\tilde A_{12}\\
     z\tilde E_{21} - \tilde A_{21}	& z\tilde E_{22}-\tilde A_{22}
   \end{bmatrix}:=T_2^{-1}
   \begin{bmatrix}
     z\hat E_{11}-\hat A_{11}	& z\hat E_{12}-\hat A_{12}\\
     z\hat E_{21}-\hat A_{21}	& z\hat E_{22}-\hat A_{22}
   \end{bmatrix},
\end{equation*}
accordingly partitioned.
Thus reevaluating \eqref{eq:thmdlurefullnormalrank} for the transformed matrices we also obtain full normal rank $n_2$ of 
$z\tilde E_{22}-\tilde A_{22}$. Hence there exists some $\lambda_0\in\C$ such that
$\lambda_0 \tilde E_{22}-\tilde A_{22}$ is invertible.
From the last block column and first block row  of \eqref{eq:thmdluredeflss} we infer
\begin{equation*}
    0=\tilde Z_{11}(\lambda_0 \tilde E_{12}-\tilde A_{12}) + \tilde Z_{12}(\lambda_0 \tilde E_{22}-\tilde A_{22}).
\end{equation*}
Thus, $\tilde Z_{12}$ can be expressed as
\begin{equation*}
  \tilde Z_{12}=-\tilde Z_{11}(\lambda_0 \tilde E_{12}-\tilde A_{12})(\lambda_0 \tilde E_{22}-\tilde A_{22})^{-1}.
\end{equation*}
Inserting this relation into the first block row and block column of \eqref{eq:thmdluredeflss} we have
\begin{equation*}
   (z\A_s^*-\A_s)Y_s = \tilde Z_{11}\left(z\tilde E_{11}-\tilde A_{11} - (\lambda_0 \tilde E_{12}-\tilde A_{12})(\lambda_0 \tilde E_{22}-\tilde A_{22})^{-1}(z\tilde E_{21}-\tilde A_{21}) \right).
\end{equation*}
Hence, we are finally in the position to apply Theorem~\ref{th:cluredefss}. From this we obtain a solution $(X_s,\,K_s,\,L_s)$ of \eqref{eq:clure} corresponding to the system $\wsystemS*$. By Lemma~\ref{lem:dlure2clure} we then also find a solution $(X_F,\,K_F,\,L_F)$ of \eqref{eq:dlure} corresponding to the system in feedback equivalence form $\wsystemF*$.  
 \end{proof}

 \begin{remark}\label{rem:deflss} Let an I-controllable system $\wsystem$ be given and let $z\A^*-\A$ be the palindromic pencil as in \eqref{eq:palpenc}. Further, assume that there exists a solution $(X,\,K,\,L)$ of the Lur'e equation \eqref{eq:dlure}. 
  \begin{enumerate}[label=(\alph*)]
  \item \label{it:deflssa} The matrix pencil $z\tilde E - \tilde A\in\matz{n+q}{n+m}$ that we have obtained in the proof of Theorem~\ref{thm:dluredfss} fulfills $\rkr (z\tilde E - \tilde A)=n+q$ by Proposition~\ref{prop:lurefullrankrhs}, since
  \begin{align}\label{eq:lurerhsfullrank}
  \begin{split}
  n+q=&\rkr 
    \begin{bmatrix}
      zE-A	& -B	\\
      (z-1)K	&(z-1)L
    \end{bmatrix}\\    
    =& \rkr
     \begin{bmatrix}
       W & 0\\
       0 & I_m
     \end{bmatrix} 
     \begin{bmatrix}
      zE-A	& -B	\\
      (z-1)K	&(z-1)L
    \end{bmatrix}
    \mathcal{T}_F\\
    =& \rkr
    \begin{bmatrix}
      zI_{n_1}-A_{11}	& 	0	& -B_1	\\
      0			& - I_{n_2}	& -B_2  \\
      (z-1)K_1		& (z-1)K_2	&(z-1)L 
    \end{bmatrix}\\
    =& \rkr
    \begin{bmatrix}
      zI_{n_1}-A_{11}	& 	0	& -B_1	\\
      0			& - I_{n_2}	& 0  	\\
      (z-1)K_1		&  0		&(z-1)(L-K_2B_2) 
    \end{bmatrix}.
    \end{split}
  \end{align}
  In particular, this means that the existence of solutions of \eqref{eq:dlure} implies the existence of a deflating subspace of the palindromic pencil $z\A^*-\A$.
   \item \label{it:deflssb}  In the proof of Theorem~\ref{thm:dluredfss} we have constructed a deflating subspace $\im Y_F$ as in \eqref{eq:deflss} for the system 
   $\wsystemF*$ in feedback equivalence form \eqref{eq:fbsys} from a solution $(X_F,\, K_F,\,L_F)$ of the Lur'e equation \eqref{eq:dlure}. From here we can construct a deflating subspace $\im Y$ for the original system by using \eqref{eq:transformdss2FB}. By Lemma~\ref{lem:dlure2clure} it is justified to set
   \begin{equation*}\label{eq:X_FonlyX11}
   W^{-*}XW^{-1}=X_F:=
     \begin{bmatrix}
       X_{11} & 0\\
       0	&0
     \end{bmatrix}.
   \end{equation*}

   Thus, we have  
   \begin{align*}
   Y:=&U_FY_F\hat V
   \mathcal T_F^{-1}	
  \\
  =&\left[
   \begin{array}{@{}c|cc@{} }
   W^*		& 0 		& 0	 \shline{1.6}
   0		& T	 	& 0	 \\
   0		& FT 		& I_m 
  \end{array}\right]
   \left[
   \begin{array}{@{}ccc@{} }
    X_{11}(A_{11}-I_{n_1})	& 0		& X_{11}B_1		\\
    0				& -I_{n_2}	& -B_2			\shline{1.6}		
    I_{n_1}			& 0		& 0			\\  		
    0				& 0		& -B_2			\\
    0				& 0		& I_m		
   \end{array}\right]
   \mathcal T_F^{-1}
   \\ 
   =&
 \left[
   \begin{array}{@{}c|cc@{} }
   W^*		& 0 		& 0	 \shline{1.6}
   0		& T	 	& 0	 \\
   0		& FT 		& I_m 
  \end{array}\right]
  \left[
  \begin{array}{@{}c|cc@{} }
   -X_F + (I_n - E_F)	&  0 		& 0 	 \shline{1.6}
    0			& I_{n} 	& 0	 \\
      0			& 0		& I_m	 
  \end{array}
  \right]
    \left[
   \begin{array}{@{}ccc@{} }
   I_{n_1} -A_{11}		& 0			& -B_1		\\
    0				& -I_{n_2}		& -B_2			\shline{1.6}		
    I_{n_1}			& 0			& 0			\\  		
    0				& 0			& -B_2			\\
    0				& 0			& I_m		
   \end{array}\right]
   \mathcal T_F^{-1}
   \\
   =&
   \begin{bmatrix}
    X(A-E) + G_1 		& XB + G_2 		\\
    V_1			& V_2
   \end{bmatrix},
   \end{align*}
  where 
    \begin{equation}\label{eq:G1G2pal}
   \im
   \begin{bmatrix}
    G_1 & G_2  
   \end{bmatrix}
    = \im W^*
    \begin{bmatrix}
     0 & 0 \\
     0 & I_{n_2}
    \end{bmatrix}
    W
    \begin{bmatrix}
     E-A & -B
    \end{bmatrix} \subseteq \ker E^*,
  \end{equation}  
  \begin{equation}\label{eq:Vsigma1}
    \begin{bmatrix}
       V_1	& V_2
    \end{bmatrix}:=\mathcal T_F V_F \mathcal T_F^{-1},
  \end{equation}
  and 
  \begin{equation*}\label{eq:Vsigma2}
    V_F:=
    \begin{bmatrix}
       I_{n_1}			& 0			& 0			\\  		
    0				& 0			& -B_2			\\
    0				& 0			& I_m		
    \end{bmatrix}
  \end{equation*}
  spans the system space $\cV F$, see Proposition~\ref{prop:sysspace}\ref{it:sysspacea}.
  Altogether, this leads to $(z\A^*-\A) Y=Z(z\check E-\check A)$, where $Z=U_F^{*}Z_F$ and $z\check E-\check A:=(z\tilde E - \tilde A)\mathcal T_F^{-1}$.
  \end{enumerate}

 \end{remark}
 
 \begin{example}[Example \ref{ex:simplecircuitfeedback} revisited]\label{ex:simplecircuitidelure}
 Consider  the system $\wsystem*$ as in \eqref{eq:simplecircuitmatrices} and Example~\ref{ex:simplecircuitkyp}. Note that since $n_3=0$ in \eqref{eq:simplecircuitfeedback}, the system $\system*$ is I-controllable according to Table~\ref{tab:algchar}.
We have seen in Example~\ref{ex:simplecircuitsdelure} that 
\begin{equation*}
 (X_s,\, K_s,\, L_s) = \left(\sqrt 3,\, {\sqrt2},\, -\frac{\sqrt3+1}{\sqrt2}\right)
\end{equation*}
is a solution of the Lur'e equation \eqref{eq:clure}
corresponding to  the EDE system 
$
 \wsystemS *
$
 as in \eqref{eq:simpleciruictEDE}.
By Lemma~\ref{lem:dlure2clure} we obtain that 
\begin{equation*}\label{eq:simplecircuitfeedbacklure}
 X_F=
 \begin{bmatrix}
  \sqrt3 & 0\\
  0	& 0
 \end{bmatrix},
 \quad K_F=
 \begin{bmatrix}
  {\sqrt2} & 0
 \end{bmatrix}, \quad
  L_F = -\frac{\sqrt3+1}{\sqrt2}
\end{equation*}
solves the Lur'e equation of the system in feedback equivalence form. Therefore, by Lemma~\ref{lem:dlure2flure} we see that
\begin{gather*}
\begin{split}
\label{eq:simplecircuitlure}
 X=W^*X_FW=
       \begin{bmatrix}
	  1	& -1\\
	  1			& 0        
      \end{bmatrix}
 \begin{bmatrix}
  \sqrt{3} & 0\\
  0	& 0
 \end{bmatrix}
       \begin{bmatrix}
	  1	& 1\\
	  -1			& 0        
      \end{bmatrix}=
      \begin{bmatrix}
       \sqrt{3} & \sqrt{3}\\
       \sqrt{3} & \sqrt{3}
      \end{bmatrix},\\
 K=K_FT^{-1}
 \begin{bmatrix}
  {\sqrt2} & 0
 \end{bmatrix}
 \begin{bmatrix}
      0			& 1 \\
      1			& -1
      \end{bmatrix}
 =  \begin{bmatrix}
  0 &{\sqrt2}
 \end{bmatrix}, \quad
  L= -\frac{\sqrt3+1}{\sqrt2}
  \end{split}
\end{gather*}
solves the Lur'e equation \eqref{eq:dlure} corresponding to the original system.

Thus, according to Remark~\ref{rem:deflss} the matrix $Y\in\mat{5}{3}$ defined by
\begin{equation}\label{eq:simplecircuitdeflatingsubspace}
Y=
  \begin{bmatrix}
    X(A-E) + G_1 		& XB + G_2 		\\
    V_1			& V_2
 \end{bmatrix}
 =\left[
 \begin{array}{@{}cc|c@{}}
   0 + 1	& 0 - 1 	& -\sqrt{3} +1 \\
   0 + 0	& 0 + 0 	& -\sqrt{3} + 0 \shline{1.6}
   0		& 1    		& 1		\\
   0		& 1		& 0		\\
   0		& 0		& 1
 \end{array}\right]
\end{equation}
is a basis matrix of the deflating subspace of the palindromic  pencil in \eqref{eq:simplecircuitpalpenc}. 
\end{example}

 As in the EDE case we can show a similar statement for BVD pencils as in \eqref{eq:bvdpenc}.

  \begin{theorem}\label{prop:dluredfss}   Let $\wsystem$ be I-controllable. Further, let the corresponding BVD pencil $z\E-\A$ as in \eqref{eq:bvdpenc} be given.  In addition, let $q=\rkr\Phi(z)$.                                        
 Then the following are equivalent:
 \begin{enumerate}[label=(\alph*)]
  \item \label{it:dluredefssbvda} There exists a solution $(X,\,K,\,L)\in\mat{n}{n}\times\mat{q}{n}\times\mat{q}{m}$ of the Lur'e equation~\eqref{eq:dlure}.
  \item \label{it:dluredefssbvdb} It holds that $\Phi(\eio)\succeq 0$ for all $\omega\in[0,2\pi)$ such that $\det(\eio E-A)\neq 0$. Furthermore, there exist matrices $Y_\mu,\,Y_x\in\mat{n}{n+m},\,Y_u\in\mat{m}{n+m}$ and $Z_\mu,\,Z_x\in\mat{n}{n+q},\,Z_u\in\mat{m}{n+q}$ such that for
  \begin{equation*}
   Y=\begin{bmatrix}
      Y_\mu\\
      Y_x\\
      Y_u
     \end{bmatrix},\qquad
   Z=\begin{bmatrix}
      Z_\mu\\
      Z_x\\
      Z_u
     \end{bmatrix}
  \end{equation*}
  the following holds:
  \begin{enumerate}[label=(\roman*)]
   \item \label{it:dluredefssbvdbi}The matrix 
   \begin{equation*}
    \begin{bmatrix}
    E	& 0
    \end{bmatrix}
    \begin{bmatrix}
     Y_x\\
     Y_u
    \end{bmatrix}
   \end{equation*}
 has  rank $n_1$.
   \item \label{it:dluredefssbvdbii} The space $\mathcal Y= \im Y$ is of dimension $n+m$ and $\E^e$-neutral, where 
   \begin{equation*}
   \E^e:=
    \begin{bmatrix}
     0 	& -E	& 0\\
     E^*& 0	& 0\\
     0	& 0	& 0
    \end{bmatrix}.
   \end{equation*}
   \item \label{it:dluredefssbvdbiii} It holds that 
   \begin{equation*}
      \cV=\im
      \begin{bmatrix}
        Y_x\\
        Y_u
      \end{bmatrix}.
   \end{equation*}

   \item \label{it:dluredefssbvdbiv} There exist $\tilde E,\,\tilde A \in\mat{n+q}{n+m}$ such that $(z\E-\A)Y=Z(z\tilde E - \tilde A)$.
  \end{enumerate}

 \end{enumerate}

 \end{theorem}
 \begin{proof}
   See \cite[Theorem 5.14]{bankmann_linear-quadratic_2015}.
 \end{proof}

 \begin{remark}\label{rem:deflssbvd} Let an I-controllable system $\wsystem$ be given and let $z\E-\A$ be the BVD pencil as in \eqref{eq:bvdpenc}. Further, assume that there exists a solution $(X,\,K,\,L)$ of the Lur'e equation \eqref{eq:dlure}. 
  \begin{enumerate}[label=(\alph*)]
  \item \label{it:deflssabvd} The matrix pencil $(z\tilde E - \tilde A)\in\matz{n+q}{n+m}$ we obtain in the proof of Theorem~\ref{prop:dluredfss} fulfills $\rkr (z\tilde E - \tilde A)=n+q$, see \eqref{eq:lurerhsfullrank}.
  In particular, this means that the existence of solutions of \eqref{eq:dlure} implies the existence of a deflating subspace of the BVD pencil $z\E-\A$.
   \item \label{it:deflssbbvd}  In the proof of Theorem~\ref{prop:dluredfss} we have constructed a deflating subspace $Y_F$  for the weighted system 
   $\wsystemF*$ in feedback equivalence form \eqref{eq:fbsys} from a solution $(X_F,\, K_F,\,L_F)$ of the Lur'e equation \eqref{eq:dlure}. From here we can construct a deflating subspace $Y$ for the original system. By Lemma~\ref{lem:dlure2clure} it is justified to set
   \begin{equation*}\label{eq:X_FonlyX11bvd}
   W^{-*}XW^{-1}=X_F:=
     \begin{bmatrix}
       X_{11} & 0\\
       0	&0
     \end{bmatrix}.
   \end{equation*}
   
   Thus we have  
   \begin{align*}
   Y:=&U_FY_F \hat V
   \mathcal T_F^{-1}	
  =\left[
   \begin{array}{@{}c|cc@{} }
   W^*		& 0 		& 0	 \shline{1.6}
   0		& T	 	& 0	 \\
   0		& FT 		& I_m 
  \end{array}\right]
   \left[
   \begin{array}{@{}ccc@{} }
    -X_{11}			& 0		& 0		\\
    0			& -I_{n_2}	& -B_2			\shline{1.6}		
    I_{n_1}			& 0		& 0			\\  		
    0				& 0		& -B_2			\\
    0				& 0		& I_m		
   \end{array}\right]
   \mathcal T_F^{-1}
   =
   \begin{bmatrix}
    -XE  + G_1 		& G_2 		\\
    V_1		& V_2
   \end{bmatrix},
   \end{align*}
  where $V_1,\,V_2$ are as in \eqref{eq:Vsigma1} and 
  \begin{equation}\label{eq:G1G2}
    \im
   \begin{bmatrix}
    G_1 & G_2  
   \end{bmatrix}
    = \im  W^*
    \begin{bmatrix}
  0 & 0 & 0\\
  0			& -I_{n_2}	&  -B_2				
    \end{bmatrix}\mathcal T _F^{-1}	
  \subseteq \ker E^*.
  \end{equation}  

  Altogether, this leads to $(z\E-\A) Y=Z(z\check E-\check A)$, where $Z=U_F^{*}Z_F$ and $z\check E-\check A:=(z\tilde E - \tilde A)\mathcal T_F^{-1}$.
  \end{enumerate}

 \end{remark}
 \begin{remark} A major difference between Theorem~\ref{thm:dluredfss} and Theorem~\ref{prop:dluredfss} or Theorem~\ref{th:cluredefss} and Theorem~\ref{prop:cluredefss} is that in the BVD case we do not need the artificial assumption
\begin{equation*}\label{eq:rankcondition}
\rk
\begin{bmatrix}
  E -A & -B
\end{bmatrix}=n,
\end{equation*}
or equivalently
\begin{equation*}\label{eq:rankcondition2}
\rk
\begin{bmatrix}
  I_{n_1} -A_{11} & -B_1
\end{bmatrix}=n_1,
\end{equation*}
\ie controllability at one. If the system $\system*$ is obtained by discretization with the implicit Euler method, we see that in the limiting case $h\to 0$ this corresponds to
\begin{equation*}
  \rk
  \lim\limits_{h\to0}
  \begin{bmatrix}
  I_{n_1} -hA_{11} & -hB_1
\end{bmatrix}
=
  \rk  
  \begin{bmatrix}
  I_{n_1} & 0
\end{bmatrix}
=n_1,
\end{equation*}
which is trivially fulfilled. Therefore, for sufficiently small $h$ we may assume validity of this assumption.
\end{remark}

\section{Application to Optimal Control}\label{chap:applications}
In this section we discuss the structure of solutions of the discrete-time optimal control problem \eqref{eq:objectivefunction} corresponding to the system $\wsystem$  based on the results from the previous subsection. %
First we show relations between so-called stabilizing solutions of the Lur'e equation \eqref{eq:dlure} and feasibility of the optimal control problem. Then we show characterizations for the existence and uniqueness of the optimal control.

\subsection{Stabilizing solutions}

In this subsection we state several discrete-time versions of results from \cite{ilchmann_outer_2014}. If not explicitly stated otherwise, these results can be proven analogously, \ie by using the same algebraic transformations and the identical properties of the solution operators.

\begin{definition}[Stabilizing solution] Let $\wsystem$ be given and assume that a solution $(X,\, K,\, L)\in\mat{n}{n}\times\mat{p}{n}\times\mat{p}{m}$ of the corresponding Lur'e equation \eqref{eq:dlure} exists.
	Then $(X,\, K,\, L)$ is also called \emph{stabilizing} solution if in addition it holds
	\begin{equation}\label{eq:stablure}
	\rk
	\begin{bmatrix}
	\lambda E-A	& -B	\\
	K	&L
	\end{bmatrix} = n+q
	\end{equation}
	for all $\lambda\in\C$ with $|\lambda|\ge1$.	
\end{definition}
The following is an adaptation of \cite[Proposition 6.4(b)]{ilchmann_outer_2014}.
\begin{proposition}\label{prop:analogprpo64b}
	Let $\system$ be I-controllable. If \propfourI then  we have
	\begin{equation}\label{eq:kronrb}
	\rk \begin{bmatrix}
	E & A & B\\
	0 & K & L
	\end{bmatrix}= 
	\rk 
	\begin{bmatrix}
	\lambda E - A & -B\\
	K & L
	\end{bmatrix}
	\end{equation}
	for all $\lambda\in\C$ with $|\lambda|>1$.
\end{proposition}
\begin{proof}
	The major part of the proof is completely analogous to the continuous-time case, since only algebraic operations and the linearity of the shift operator is used.
	The basic idea is that the relation \eqref{eq:kronrb} is equivalent to the fact that the Kronecker canonical form of the matrix pencil
	\begin{equation*}
		\begin{bmatrix}
		z E - A & -B\\
		K & L
		\end{bmatrix}
	\end{equation*}
	consists only of blocks of type K1 with $|\lambda|\le1$ and blocks of type K4 of size $1\times0$ \cite[Remark 2.8(b)]{ilchmann_outer_2014}. Also note, that due to the block structure of the KCF we have that for every $\varepsilon>0$ and every $v^{j,0}\in\C^{k_j}$ and corresponding block $K_j(z) = zF_j - G_j$ in the KCF there exists $v^j\in\ell_2(\C^{k_j})$ such that
	\begin{equation*}
	 F_jv_0^j = F_jv^{j,0},\quad \lim_{i\to\infty}F_jv^j_i=0,\quad \left\|F_j\OpShift v^j - G_jv^j\right\|_{\ell_2}<\varepsilon.  
	\end{equation*}

	Here we show, that all blocks of type K1 fulfill $|\lambda| \le 1$. To this end we thus proof the following fact: Let a block of the KCF of the form $zF - G := K(z) = z I_k - (\lambda I_k + N_k) \in \matz[\C]{k}{k}$ be given. If for any $v^0 \in \C^k$ and any $\varepsilon > 0$ there exists a $v\in\ell^2(\C^k)$ such that $v_0 = v^0$ and $\|\sigma v - G v\| < \varepsilon$, then $|\lambda| \le 1$. 
	
	For the sake of a contradiction, assume that $|\lambda|>1$. Then the eigenvalue $\frac1\lambda$ of $G^{-1}$ fulfills $\frac1{|\lambda|}<1$ and thus there exists positive definite $P$ such that
	\begin{equation*}
	P - G^*PG = - I_k.
	\end{equation*}
	Set $\tilde\varepsilon = (v^0)^*Pv^0>0$ and $\varepsilon:= \|(P+P^2)^\frac12 G^{-1}\|_2^{-2}\tilde\varepsilon$. Then for any $v^0\in\C^k\setminus\{0\}$ and  $w:=\sigma v - G v$ with $\|w\|^2_{\ell_2}<\varepsilon$ one finds that
	\begin{align*}
	\begin{split}
	v_{0}^*Pv_{0}-v_{j}^*Pv_{j} 
	&= -\sum_{k=0}^{j-1}{\OpShift(v_{k}^*Pv_{k}) - v_{k}^*Pv_{k}}
	= -\sum_{k=0}^{j-1}{({Gv_{k} + w_{k})}^*P{(Gv_{k} + w_{k})} - v_{k}^*Pv_{k}}\\ 
	&=-\sum_{k=0}^{j-1}{
		\begin{pmatrix}
		v_k\\
		w_k
		\end{pmatrix}^*
		\begin{bmatrix}
		G^*PG-P	& G^*P\\
		PG		& P
		\end{bmatrix}
		\begin{pmatrix}
		v_k\\
		w_k
		\end{pmatrix}
	}\\
&	= -
	\sum_{k=0}^{j-1}{
		\begin{pmatrix}
		v_k+ (I+P)G^{-1}w_k\\
		G^{-1}w_k
		\end{pmatrix}^*
		\begin{bmatrix}
		I	& 0\\
		0		& -P-P^2
		\end{bmatrix}
		\begin{pmatrix}
		v_k+ (I+P)G^{-1}w_k\\
		G^{-1}w_k
		\end{pmatrix}
	}\\
	&\le \|(P+P^2)^\frac12 G^{-1}\|_2^2\|w\|^2_{\ell^2}<\tilde\varepsilon
	\end{split}
	\end{align*}
	and thus in the limit case $\tilde\varepsilon = \lim_{j\to\infty} v_{0}^*Pv_{0} - v_{j}^*Pv_{j}< \tilde\varepsilon$.
\end{proof}

\begin{lemma}\label{lem:approx}
	Let $\wsystem$ be given. Further, assume that $(X,\, K,\, L)\in\mat{n}{n}\times\mat{p}{n}\times\mat{p}{m}$ fulfills \eqref{eq:dlure} with
	\begin{equation}\label{eq:tildeKLrank}
	\rk \begin{bmatrix}
	 K &  L
	\end{bmatrix} = p.
	\end{equation}
	If \propfour
	then there exists a stabilizing solution of the Lur'e equation \eqref{eq:stablure}.
\end{lemma}
\begin{proof}
	By the rank condition \eqref{eq:tildeKLrank} we can use the analogous version of $"\Rightarrow"$ of \cite[Proposition 6.5]{ilchmann_outer_2014} for $Y=I_p$, where Proposition~\ref{prop:analogprpo64b} is used.
\end{proof}

\begin{lemma}\label{lem:approx2}
	Let $\wsystem$ be stabilizable. Further, assume that $(X,\, K,\, L)\in\mat{n}{n}\times\mat{q}{n}\times\mat{q}{m}$ is a stabilizing solution of the Lur'e equation \eqref{eq:dlure}.
	Then any given sequence $v\in\ell^2(\K^q)$ can be approximated arbitrarily well by $y = Kx + Lu$ and some $\behavior\cap\left(\ell_2(\K^n)\times\ell_2(\K^m)\right)$. In other words, for every $\varepsilon>0$ and every $x^0\in\Vshift$ there exists $\behavior\cap\left(\ell_2(\K^n)\times\ell_2(\K^m)\right)$ such that $Ex_0= Ex^0$ and
	\begin{equation*}
	\|v - Kx- Lu\|_{\ell^2} < \varepsilon.
	\end{equation*}
\end{lemma}
\begin{proof}
	Since $(X,\, K,\, L)$ is a stabilizing solution one can easily show that
	\begin{equation*}
	\rk
	\begin{bmatrix}
	E & A &B\\
	0 & K &L
	\end{bmatrix}= n+ q.
	\end{equation*}
Then the result is obtained by using the discrete-time version of \cite[Proposition 6.4(a)]{ilchmann_outer_2014}. The proof is tedious but analogous and relies on a discrete-time version  of \cite[{Theorem 5.1}]{ilchmann_outer_2014}, which relates the algebraic characterization \eqref{eq:stablure} to the fact, that arbitrary sequences $v\in\ell^2(\K)$ can be approximated arbitrary well by $y = Kx + Lu$ and given $\behavior\cap\left(\ell_2(\K^n)\times\ell_2(\K^m)\right)$. The argumentation uses the Hardy space $\mathcal H_2^{q}$ of analytic functions $\mathbf{G} : \{ \lambda \in \C \,:\, \left| \lambda \right| > 1 \} \to \C^{q}$ with $\int_0^{2\pi} \left\|\mathbf{G}(\mathrm{e}^{\mathrm{i} \omega})\right\|_{\rm F}^2 \mathrm{d} \omega < \infty$ and revolves around the fact, that the multiplication operator mapping the Z-transform of the input to the Z-transform of the output has dense range in $\mathcal H_2^{q}$.

\end{proof}

\subsection{Feasibility}

We characterize feasibility and structure of the optimal control problem with the existence of a stabilizing solution of the Lur'e equation. First we show that the existence of a stabilizing solution implies feasibility and an explicit characterization of the optimal value function $\inffunc$ for given $x^0\in\Vshift$.
\begin{theorem}
	Let $\wsystem$ be given with no uncontrollable modes on the unit circle and assume there exists a stabilizing solution $(X,\, K,\, L)\in\mat{n}{n}\times\mat{q}{n}\times\mat{q}{m}$ of the Lur'e equation \eqref{eq:stablure}. Then the optimal control problem is feasible, \ie $\inffunc\in\R$ for all $x^0\in\Vshift$ and $\inffunc = (x^0)^*E^*XEx^0$.
\end{theorem}
\begin{proof}
	First, since  $\wsystem$ has no uncontrollable modes on the unit circle and $(X,\, K,\, L)$ is a stabilizing solution we obtain that
	\begin{equation*}
	\rk 
	\begin{bmatrix}
	\lambda E - A & -B
	\end{bmatrix}
	+ q =
		\rk
		\begin{bmatrix}
		\lambda E-A	& -B	\\
		K	&L
		\end{bmatrix} = n+q
	\end{equation*}
	for all $\lambda\in\C$ with $|\lambda|\ge 1$. Thus, $\wsystem*$ is stabilizable.

Let $x^0\in\Vshift$. Then, we have $\behavior$ with $Ex_0=Ex^0$ and $\lim_{j\to\infty}Ex_j=0$. By the definition of the system space $\cV\subseteq\K^{n+m}$, we obtain $\vect*{x_j,u_j}\in\cV$ for all $j\in\N_0$. Thus, for $j_2\ge j_1$ we have that
\begin{align}\label{eq:manipulate_objfunc}
\begin{split}
&x_{j_2}^*E^*XEx_{j_2} - x_{j_1}^*E^*XEx_{j_1}
= \sum_{k=j_1}^{j_2-1}{({Ax_{k} + Bu_{k})}^*X{(Ax_{k} + Bu_{k})} - x_{k}^*E^*XEx_{k}}\\ 
=&\sum_{k=j_1}^{j_2-1}{
	\begin{pmatrix}
	x_k\\
	u_k
	\end{pmatrix}^*
	\begin{bmatrix}
	A^*XA-E^*XE	& A^*XB\\
	B^*XA		& B^*XB
	\end{bmatrix}
	\begin{pmatrix}
	x_k\\
	u_k
	\end{pmatrix}
}
\ge -
\sum_{k=j_1}^{j_2-1}{
	\begin{pmatrix}
	x_k\\
	u_k
	\end{pmatrix}^*
	\begin{bmatrix}
	Q		& S\\
	S^*	& R
	\end{bmatrix}
	\begin{pmatrix}
	x_k\\
	u_k
	\end{pmatrix}
}. 
\end{split}
\end{align}
For $j_1=0,\,j_2\to\infty$  we thus  obtain for the objective function $\objfunc$ that
\begin{equation*}
x_0^*E^*XEx_0\le \objfunc 
\end{equation*}
and thus
\begin{equation}\label{eq:optimalvalueboundedbelow}
x_0^*E^*XEx_0\le \inffunc < \infty.
\end{equation}
Furthermore, since $(X,\,K,\,L)$ is a stabilizing solution of the Lur'e equation \eqref{eq:dlure}, for every $x^0\in\Vshift$ and $\behavior\cap\left(\ell^2(\K^n) \times \ell^2(\K^m)\right)$ with $Ex_0=Ex^0$ we obtain in \eqref{eq:manipulate_objfunc} that
\begin{align*}\label{eq:lureoc}
\begin{split}
-x_{0}^*E^*XEx_{0} 
=&\sum_{k=0}^{\infty}{
	\begin{pmatrix}
	x_k\\
	u_k
	\end{pmatrix}^*
	\left(
	\mathcal M (X)
	-
	\begin{bmatrix}
	Q		& S\\
	S^*	& R
	\end{bmatrix}
	\right)
	\begin{pmatrix}
	x_k\\
	u_k
	\end{pmatrix}
}\\
=&\sum_{k=0}^{\infty}{
	\begin{pmatrix}
	x_k\\
	u_k
	\end{pmatrix}^*
	\begin{bmatrix}
	K^*\\
	L^*
	\end{bmatrix}
	\begin{bmatrix}
	K & L
	\end{bmatrix}
	\begin{pmatrix}
	x_k\\
	u_k
	\end{pmatrix}
} - \objfunc
\end{split}
\end{align*}
and thus
\begin{equation}\label{eq:objlure}
x_0^*E^*XEx_0 + \|Kx + Lu\|_{\ell^2}^2 = \objfunc.
\end{equation}

Since $(X,\,K,\,L)$ is a stabilizing solution, using Lemma~\ref{lem:approx2} for $v=0$ we further conclude that
\begin{equation*}\label{eq:assumptionoptcontrol}
(x^0)^*E^*XEx^0 = \inffunc
\end{equation*}
for all $x^0\in\Vshift$.
\end{proof}
Next we show, that the opposite implication is also true, \ie that feasibility implies the existence of a stabilizing solution of the Lur'e equation.
\begin{theorem}
	Let $\wsystem$ be given with no uncontrollable modes on the unit circle. Assume that the optimal control problem is feasible, \ie $\inffunc\in\R$. Then there exists a stabilizing solution of the Lur'e equation.
	
\end{theorem}
\begin{proof}
	First we have to show, that $\inffunc = (Ex^0)^*XEx^0$ for some Hermitian $X\in\mat nn$. This can be done in an analogous way as in \cite[Theorem 3.8.3]{voigt_linear-quadratic_2015}.
	With similar steps as in \cite[Theorem 3.8.3]{voigt_linear-quadratic_2015} we can also show, that $X$ solves the KYP inequality \eqref{eq:LMI}.
	
	It remains to show, that $X$ also induces a stabilizing solution of the Lur'e equation. Since $X$ solves the KYP inequality there exist $\tilde K \in\mat pn$ and $\tilde L \in\mat pm$ such that \eqref{eq:dlure} and \eqref{eq:tildeKLrank} hold. Thus, by \eqref{eq:objlure} we obtain that for every $x^0\in\Vshift$ and $\varepsilon>0$ there exists $\behavior$ such that $\|\tilde Kx + \tilde Lu\|_{\ell^2}<\varepsilon$. Then, by Lemma~\ref{lem:approx} we also obtain a stabilizing solution of the Lur'e equation.
\end{proof}

\subsection{Existence and Uniqueness of Optimal Controls}
In this section we discuss conditions for the existence and uniqueness of optimal controls. Based on the considerations  
of the previous subsections we pose these conditions in terms of the zero dynamics $\ZD[E][A][B][K][L]$ and the pencil $\left[\begin{smallmatrix} zE - A & -B \\ K & L \end{smallmatrix}\right]$.
\begin{theorem} Let $(E,A,B,Q,S,R) \in \Sigma_{m,n}^w(\K)$ be I-controllable and assume that $(X,K,L) \in \K^{n \times n} \times \K^{q \times n} \times \K^{q \times m}$ is a stabilizing solution of the Lur'e equation \eqref{eq:dlure}. Then the following statements are satisfied:
 \begin{enumerate}[label=(\alph*)]
  \item For every $x^0 \in \K^n$ there exists a trajectory $(x,u) \in \mathfrak{B}_{(E,A,B)}$ with $Ex_0 = Ex^0$ such that $\cW_+(Ex^0) = \cJ(x,u)$ if and only if $\ZD[E][A][B][K][L]$ is strongly stabilizable.
  \item For every $x^0 \in \K^n$ there exists a unique trajectory $(x,u) \in \mathfrak{B}_{(E,A,B)}$ with $Ex_0 = Ex^0$ such that $\cW_+(Ex^0) = \cJ(x,u)$ if and only if $\ZD[E][A][B][K][L]$ is strongly asymptotically stable.
 \end{enumerate}
\end{theorem}

\begin{proof}
 \begin{enumerate}[label=(\alph*)]
  \item In view of \eqref{eq:objlure}, we see that for every $x^0 \in \K^n$ there exists a trajectory $(x,u) \in \mathfrak{B}_{(E,A,B)}$ with $Ex_0 = Ex^0$ such that $\cW_+(Ex^0) = \cJ(x,u)$, if and only for each $x^0 \in \K^n$ there exists a $(x,u) \in \ZD[E][A][B][K][L](x^0)$ with $\lim_{j \to \infty} Ex_j = 0$. Using Proposition \ref{prop:zd}(b), this is equivalent to $\ZD[E][A][B][K][L]$ being strongly stabilizable.
  \item For every $x^0 \in \K^n$ there exists a unique trajectory $(x,u) \in \mathfrak{B}_{(E,A,B)}$ with $Ex_0 = Ex^0$ such that $\cW_+(Ex^0) = \cJ(x,u)$, if and only for each $x^0 \in \K^n$ there exists a unique $(x,u) \in \ZD[E][A][B][K][L](x^0)$ with $\lim_{j \to \infty} Ex_j = 0$. With Proposition \ref{prop:zd}(d), this is equivalent to $\ZD[E][A][B][K][L]$ being strongly asymptotically stable.
 \end{enumerate}
\end{proof}
Using the results of Proposition \ref{prop:zd}, we directly obtain the following corollary.
\begin{corollary}Let $(E,A,B,Q,S,R) \in \Sigma_{m,n}^w(\K)$ be I-controllable and assume that $(X,K,L) \in \K^{n \times n} \times \K^{q \times n} \times \K^{q \times m}$ is a stabilizing solution of the Lur'e equation \eqref{eq:dlure}. Define $\cR(z) := \left[\begin{smallmatrix} zE-A & -B \\ K & L \end{smallmatrix}\right] \in \K(z)^{(n+q) \times (n+m)}$. Then the following statements are satisfied:
 \begin{enumerate}[label=(\alph*)]
  \item For every $x^0 \in \K^n$ there exists a trajectory $(x,u) \in \mathfrak{B}_{(E,A,B)}$ with $Ex_0 = Ex^0$ such that $\cW_+(Ex^0) = \cJ(x,u)$, if and only if $\rk \cR(\lambda) = n+q$ for all $\lambda \in \C$ with $|\lambda| \ge 1$ and the index of $\cR(z)$ is at most one.
  \item For every $x^0 \in \K^n$ there exists a unique trajectory $(x,u) \in \mathfrak{B}_{(E,A,B)}$ with $Ex_0 = Ex^0$ such that $\cW_+(Ex^0) = \cJ(x,u)$, if and only if $\rk \cR(\lambda) = n+m$ for all $\lambda \in \C$ with $|\lambda| \ge 1$ and the index of $\cR(z)$ is at most one.
 \end{enumerate}
\end{corollary}
Existence and uniqueness of optimal controls can also be read of the PKCF of the palindromic matrix pencil $z \A^* - \A$ in \eqref{eq:palpenc}. For this, one would need to analyze the spectral properties of the matrix pencil $\left[\begin{smallmatrix} zE-A & -B \\ (z-1)K & (z-1)L \end{smallmatrix}\right]$ in \eqref{eq:lurerankcond} corresponding to a stabilizing solution of the Lur'e equation and the structure of the deflating subspaces of individual blocks of the PKCF in detail as in \cite{voigt_linear-quadratic_2015}. For brevity of the article we leave out this result here.

\subsection{Application to palindromic and BVD matrix pencils}
Let us now discuss implications of the aforementioned results for the structure of optimal control with respect to the associated palindromic and BVD matrix pencils.
Thus, assume that $(X,\, K,\, L)\in\mat{n}{n}\times\mat{q}{n}\times\mat{p}{m}$ is a stabilizing solution of the Lur'e equation \eqref{eq:dlure}.
If  $x^0\in\Vshift$ is given, then $\behavior$ with $Ex_0=Ex^0$ and 
$\lim\limits_{j\rightarrow\infty} Ex_j= 0$ is an optimal control if and only if $\|Kx + Lu\|_{\ell^2}=0$. If this is the case, then $\vect{x,u}$ fulfills
\begin{gather*}
\begin{bmatrix}
E	& 0 \\
0	& 0
\end{bmatrix}
\begin{bmatrix}
x_{j+1}\\
u_{j+1}
\end{bmatrix}=
\begin{bmatrix}
A	& B \\
K	& L
\end{bmatrix}
\begin{bmatrix}
x_{j}\\
u_{j}
\end{bmatrix},\quad
Ex_0=Ex^0,\quad \lim\limits_{j\to\infty} Ex_j=0.
\end{gather*}

For an I-controllable system $\wsystemF$ as in \eqref{eq:fbsys} in feedback equivalence form with corresponding transformation matrices $W,T,F$ we set
\begin{equation*}
 X_F= W^{-*}X W^{-1}=
 \begin{bmatrix}
  X_{11} & X_{12}\\
  X_{21} & X_{22}
 \end{bmatrix}.
\end{equation*}
Since
\begin{align*}
 E^*XE
 =&T^{-*}
  \begin{bmatrix}
  I_{n_1} 	& 0\\
    0 		& 0
 \end{bmatrix}
  \begin{bmatrix}
  X_{11} & X_{12}\\
  X_{12}^* & X_{22}
 \end{bmatrix}
  \begin{bmatrix}
  I_{n_1} 	& 0\\
    0 		& 0
 \end{bmatrix}
 T^{-1}
  = 
 T^{-*}
  \begin{bmatrix}
  I_{n_1} 	& 0\\
    0 		& 0
 \end{bmatrix}
  \begin{bmatrix}
  X_{11} &	0\\
      0	 &	 0
 \end{bmatrix}
  \begin{bmatrix}
  I_{n_1} 	& 0\\
    0 		& 0
 \end{bmatrix}
 T^{-1},
\end{align*}
by Lemma~\ref{lem:dlure2clure} we can without loss of generality set $X_{12}=0$ and $X_{22}=0$.

In addition, further assuming that 
$
 \rk \begin{bmatrix}E-A & B\end{bmatrix}=n,
$   
from Theorem~\ref{prop:dluredfss} we obtain a deflating subspace $Y\in\matz{2n+m}{n+m}$ of the BVD pencil $z\E-\A$ as in \eqref{eq:bvdpenc}, \ie we have $Z\in\mat{2n+m}{n+q}$ and a matrix pencil $z\check E-\check A\in\matz{n+q}{n+m}$ such that $(z\E-\A )Y=Z(z\check E-\check A)$. It can be constructed as in Remark~\ref{rem:deflssbvd}\ref{it:deflssbbvd}. Inserting $\OpShift$ for $z$  leads to
\begin{multline}\label{eq:optcontroldefssbvd}
\left[
    \begin{array}{@{}c|cc@{} }
    0				& \OpShift E-A		& -B 			\shline{1.6}
      \OpShift A^*-E^*		& -Q			& -S 	\\
	\OpShift B^*		& -S^*			& -R
  \end{array}\right]
  \left[
    \begin{array}{@{}c|c@{} }
    -XE + G_1 			& G_2 		\shline{2.5}
    V_1			& V_2
\end{array}\right]
   \begin{pmatrix}
     x_j\\
     u_j
   \end{pmatrix}
   \\
   =Z
   \left[
    \begin{array}{@{}ccc@{} }
     \OpShift I_{n_1}-A_{11}			& 0			& -B_1						\\
     K_1				& 0			& L-K_2B_2			\\		  
     -Q_{12}^*			&- \OpShift I_{n_2}	& -\OpShift B_2 +Q_{22}B_2-S_2 
    \end{array}\right]   \mathcal T_F^{-1}
    \begin{pmatrix}
     x_j\\
     u_j
   \end{pmatrix},
\end{multline}
 where $\im \left[\,G_1 \enspace G_2\,\right]\subseteq \ker E^*$, see \eqref{eq:G1G2},
  \begin{equation*}
    \begin{bmatrix}
       V_1	& V_2
    \end{bmatrix}:=\mathcal T_F V_F \mathcal T_F^{-1},
  \end{equation*}
  and 
  \begin{equation*}
    V_F:=
    \begin{bmatrix}
       I_{n_1}			& 0			& 0			\\  		
    0				& 0			& -B_2			\\
    0				& 0			& I_m		
    \end{bmatrix}.
  \end{equation*}
Since $\vect{x_j,u_j}\in\cV$,
it follows with Proposition~\ref{prop:sysspace}\ref{it:sysspacea} that 
\begin{equation*}
\mathcal T_F^{-1} 
\begin{pmatrix}
  x_j\\
  u_j
\end{pmatrix}
= 
\begin{pmatrix}
  x_{1,j}\\
  -B_2u_j\\
  u_j
\end{pmatrix}
\end{equation*}
 for some $x_{1,j}\in \K^{n_1}$.
Then by Lemma~\ref{lem:dlure2clure} and Lemma~\ref{lem:dlure2flure} we have that 
\begin{equation}\label{eq:bvdrhs1bvd}
  \begin{bmatrix}
    K_1	& L-K_2B_2
  \end{bmatrix}
  \begin{pmatrix}
    x_{1,j} \\
    u_j
  \end{pmatrix}
  =
  \begin{bmatrix}
    K_F	& L_F
  \end{bmatrix}\mathcal{T}_F^{-1}
  \begin{pmatrix}
    x_j \\
    u_j
  \end{pmatrix}
  =
    \begin{bmatrix}
    K	& L
  \end{bmatrix}
  \begin{pmatrix}
    x_j \\
    u_j
  \end{pmatrix}=0.
\end{equation}
In addition, from \eqref{eq:feedbackkyp} we obtain
\begin{equation*}
  \begin{bmatrix}
    Q_{12}^*	& Q_{22}	& S_2
  \end{bmatrix}=K_2^*
  \begin{bmatrix}
    K_1	& K_2 & L
  \end{bmatrix}
\end{equation*}
and thus
\begin{align}\label{eq:bvdrhs2bvd}
  \begin{split}
 \begin{bmatrix}
    Q_{12}^*	& S_2 -Q_{22}B_2
  \end{bmatrix}
  \begin{pmatrix}
    x_{1,j} \\
    u_j
  \end{pmatrix}
  =&
    \begin{bmatrix}
    Q_{12}^*	& Q_{22}	& S_2
  \end{bmatrix}
  \mathcal{T}_F^{-1}
  \begin{pmatrix}
    x_j \\
    u_j
  \end{pmatrix}
  =K_2^*
  \begin{bmatrix}
    K_1	& K_2 & L
  \end{bmatrix}
  \mathcal{T}_F^{-1}
  \begin{pmatrix}
    x_j \\
    u_j
  \end{pmatrix}\\
  =&K_2^*
  \begin{bmatrix}
    K	& L
  \end{bmatrix}
  \begin{pmatrix}
    x_j \\
    u_j
  \end{pmatrix}=0.
  \end{split}
\end{align}
Thus, by equations \eqref{eq:bvdrhs1bvd} and \eqref{eq:bvdrhs2bvd} the right-hand-side of \eqref{eq:optcontroldefssbvd} is zero. Furthermore, by Proposition~\ref{prop:sysspace}\ref{it:sysspaced} we have that 
\begin{equation*}
V
\begin{pmatrix}
  x_j\\
  u_j
\end{pmatrix}
=
\begin{pmatrix}
  x_j\\
  u_j
\end{pmatrix}.
\end{equation*}
Set
\begin{equation*}\label{eq:lagrangereprbvd}
\mu_j := 
\begin{bmatrix}
  -XE + G_1 & G_2 
\end{bmatrix}
\begin{pmatrix}
  x_j\\
  u_j
\end{pmatrix}.
\end{equation*}
Thus  
\begin{equation*}
 \lim\limits_{j\to\infty}{E^*\mu_j} =\lim\limits_{j\to\infty}{-E^*XEx_j}=0,
\end{equation*}
and hence, $(\mu_j)_j$
is part of a solution of the boundary value problem
\begin{gather*}
  \begin{bmatrix}
    0			&  E	&  	0		\\
       A^*		& 0	& 0 			\\
	 B^*		& 0	& 0
  \end{bmatrix}
  \OpShift
  \begin{pmatrix}
    \mu\\
    x\\
    u
  \end{pmatrix}
  =
    \begin{bmatrix}
    0		&  A	& B			\\
    E^*		& Q	& S 			\\
	0	& S^*	& R
  \end{bmatrix}
  \begin{pmatrix}
    \mu\\
    x\\
    u
  \end{pmatrix},\qquad
  Ex_0 = E x^0, \quad \lim\limits_{j\to\infty}{E^*\mu_j}=0.
\end{gather*}

Moreover, for an I-controllable system $\wsystem*$ we can  take the same approach for a deflating subspace $Y\in\matz{2n+m}{n+m}$ of the palindromic pencil $z\A^*-\A$ as in \eqref{eq:palpenc} obtained  in Theorem~\ref{thm:dluredfss}. There we have $Z\in\mat{2n+m}{n+q}$ and a matrix pencil $z\check E-\check A\in\matz{n+q}{n+m}$ such that $(z\A^*-\A) Y=Z(z\check E-\check A)$. It can be constructed as in Remark~\ref{rem:deflss}\ref{it:deflssb}. Inserting $\OpShift$ for $z$ leads to
\begin{multline}\label{eq:optcontroldefss}\left[
    \begin{array}{@{}c|cc@{} }
    0				& \OpShift E-A		& -B 			\shline{1.6}
      \OpShift A^*-E^*		& (\OpShift-1)Q		& (\OpShift-1)S 	\\
	\OpShift B^*		& (\OpShift-1)S^*	& (\OpShift-1)R
  \end{array}\right]\left[
  \begin{array}{@{}c|c@{} }
    -X(E-A) +G_1		& -XB + G_2 		\shline{2.3}
    V_1			& V_2
   \end{array}\right]
   \begin{pmatrix}
     x_j\\
     u_j
   \end{pmatrix}
   \\
   =Z
   \left[
    \begin{array}{@{}ccc@{} }
     \OpShift I_{n_1}-A_{11}			& 0			& -B_1						\\
    (\OpShift -1)K_1				& 0			& (\OpShift -1)(L-K_2B_2)			\\		  
      (\OpShift -1)Q_{12}^*			& -\OpShift I_{n_2}	& -\OpShift B_2 +(\OpShift -1) (S_2 - Q_{22}B_2 )
    \end{array}\right]   \mathcal T_F^{-1}
    \begin{pmatrix}
     x_j\\
     u_j
   \end{pmatrix},
\end{multline}
 where $\im \left[\,G_1 \enspace G_2\,\right]\subseteq \ker E^*$, see \eqref{eq:G1G2pal}.

Again by equations \eqref{eq:bvdrhs1bvd} and \eqref{eq:bvdrhs2bvd} the right-hand-side of \eqref{eq:optcontroldefss} is $0$. Set
\begin{equation}\label{eq:lagrangereprpal}
m_j :=  
\begin{bmatrix}
  X(A - E) + G_1 & XB + G_2
\end{bmatrix}
\begin{pmatrix}
  x_j\\
  u_j
\end{pmatrix}.
\end{equation}
Thus  
\begin{align*}
 \sum_{k=0}^\infty{E^*m_k}=&  \sum_{k=0}^\infty{-E^*XEx_k + E^*X
 \begin{bmatrix}
   A  & B
\end{bmatrix}
\begin{pmatrix}
  x_k\\
  u_k
\end{pmatrix} 
 }\\
 =&
 \sum_{k=0}^\infty{-E^*XE(x_k - x_{k+1})}\\
 =&-E^*XEx_0 = E^* \mu_0
\end{align*}
and hence, $(m_j)_j$
is part of a solution of the boundary value problem
\begin{gather}
\begin{split}\label{eq:bvppal}
  \begin{bmatrix}
    0			&  E	&  	0		\\
       A^*		& Q	& S 			\\
	 B^*		& S^*	& R
  \end{bmatrix}
  \OpShift
  \begin{pmatrix}
    m\\
    x\\
    u
  \end{pmatrix}
  =
    \begin{bmatrix}
    0		&  A	& B			\\
    E^*		& Q	& S 			\\
	0	& S^*	& R
  \end{bmatrix}
  \begin{pmatrix}
    m\\
    x\\
    u
  \end{pmatrix},\quad
  Ex_0 = E x^0, \quad \sum_{k=0}^\infty{E^*m_k}=E^*\mu_0.
  \end{split}
\end{gather}

\begin{example}[Example \ref{ex:simplecircuitfeedback} revisited]
 Consider  the system $\wsystem*$ as in \eqref{eq:simplecircuitmatrices} and Example~\ref{ex:simplecircuitkyp}. In Example~\ref{ex:simplecircuitidelure} we have seen that 
\begin{equation*}(X,\,K,\,L)=\left(
  \begin{bmatrix}
    \sqrt{3} & \sqrt{3}\\
    \sqrt{3} & \sqrt{3}
  \end{bmatrix},\,
    \begin{bmatrix}
  0 &{\sqrt2}
 \end{bmatrix},\,
   -\frac{\sqrt3+1}{\sqrt2}
  \right)
\end{equation*}
is a solution of the Lur'e equation \eqref{eq:dlure}. We have
\begin{equation*}
  E^*XE=
  \begin{bmatrix}
    0	& 0\\
    0	& \sqrt3
  \end{bmatrix}
\end{equation*}
and thus by \eqref{eq:optimalvalueboundedbelow} for every 
\begin{equation*}
  x^0=
  \begin{pmatrix}
    x^0_1\\
    x^0_2
  \end{pmatrix}
  \in\Vshift
\end{equation*}
the optimal value $\inffunc$ is bounded from below by $\sqrt3 \,|x^0_2|^2$.

Indeed, setting 
\begin{equation*}
 u_j = \frac{2}{\sqrt3 +1}\left(1-\frac{2}{\sqrt3 +1}\right)^j x^0_2,
\end{equation*}
we obtain that
\begin{equation*}
 x_j = 
 \begin{pmatrix}
  1-\frac{2}{\sqrt3 +1}\\
  1
 \end{pmatrix}
 \left(1-\frac{2}{\sqrt3 +1}\right)^j x^0_2
\end{equation*}
solves the system given by \eqref{eq:simplecircuitmatrices} with 
\begin{align*}
 \objfunc = &
 \sum_{j=0}^\infty{
  \|x_j\|^2 + \|u_j\|^2
 }
 =|x_2^0|^2\frac{12-6\sqrt3}{1-\left(1-\frac{2}{\sqrt3 +1}\right)^2}=|x_2^0|^2\sqrt3,
\end{align*}
\ie $(x,u)$ is an optimal control fulfilling $Ex_0=Ex^0$ and $\lim\limits_{j\to\infty}Ex_j=0$.

In particular, from \eqref{eq:simplecircuitdeflatingsubspace} we obtain
that
\begin{equation*}
m_j=
 \begin{bmatrix}
  1 & -1 & -\sqrt3 +1\\
  0 & 0 & -\sqrt3
 \end{bmatrix}
 \begin{pmatrix}
  x_j\\
  u_j
 \end{pmatrix}
 =
 -\sqrt{3} \frac{2}{\sqrt3 +1}
 \begin{pmatrix}
   1\\
  1
 \end{pmatrix}
\left(1-\frac{2}{\sqrt3 +1}\right)^j x^0_2
\end{equation*}
fulfills the boundary value problem \eqref{eq:bvppal}, where
\begin{equation*}
  \mu_0=  -\sqrt3
  \begin{pmatrix}
    1\\
    1
  \end{pmatrix}
x_2^0.
\end{equation*}

\end{example}

\section{Conclusions and Outlook}
We have discussed several problems arising in the discrete-time linear-quadratic optimal control problem and we have seen their relations to the results that have been obtained in the continuous-time setting.  In Section~\ref{chap:kyp} we have discussed an extension of the Kalman-Yakubovich-Popov inequality for standard difference equations to the case of implicit difference equations. The characterizations are  analogous to what was obtained in \cite{voigt_linear-quadratic_2015, reis_kalmanyakubovichpopov_2015} in the continuous-time case. Nonetheless, some more technical difficulties had to be tackled. For an analogous relaxation of the controllability assumption to sign-controllability we would need the discrete-time analog of of \cite[Theorem 6.1]{clements_spectral_1997}. 

In Section~\ref{chap:inertia} we further related the spectral properties of the palindromic pencil associated to the discrete-time optimal control problem \eqref{eq:objectivefunction} to the positivity of the Popov function on the unit circle. To this end, we introduced the notion of quasi-Hermitian matrices which allows for a generalization of the concept of inertia. 

In Section~\ref{chap:lure} we introduced Lur'e equations for explicit as well as for implicit difference equations. We have shown that solvability of these equations is equivalent to the existence of certain deflating subspaces of the BVD and palindromic pencil arising in the discrete-time control problem \eqref{eq:objectivefunction}. In the palindromic case we needed the additional assumption that the given system is controllable at the eigenvalue one, which can always be achieved for discrete-time systems originating from discretization. It is an open question whether this condition can be dropped if the latter is not the case.

In Section~\ref{chap:applications} we have seen how we can use these results to characterize feasibility of the optimal control problem as well as existence and uniqueness of optimal controls. Furthermore, we have shown how the deflating subspaces are related to the solutions of the related two-point boundary value problems.

\nomenclature[asets]{$\N$}{$=\{1,2,\ldots\}$; set of natural numbers}
\nomenclature[asets]{$\N_0$}{$=\N\cup\{0\}$}
\nomenclature[asets]{$\R$}{field of real numbers}
\nomenclature[asets]{$\R^+$}{set of positive real numbers}
\nomenclature[asets]{$\R^+_0$}{set of non-negative real numbers}
\nomenclature[asetsb]{$\C$}{field of complex numbers}
\nomenclature[asetsz]{$\K$}{$\in\{\C,\R\}$}
\nomenclature[av]{$\K[z]$}{ring of polynomials with coefficients in $\K$}
\nomenclature[aw]{$\K(z)$}{field of rational functions that can be expressed as fraction of elements of $\K[z]$}
\nomenclature[bmatra]{$e_i^k$}{$i$-th unit vector in $\K^k$}
\nomenclature[bmatri]{$K^{\N_0}$}{set of all sequences $x=(x_j)_j$ whose components lie in the space $K$}
\nomenclature[bmatria]{$\cR^{m \times n}$}{set of $m$ by $n$ matrices with entries in a ring $\cR$}
\nomenclature[bmatrix]{$A^*$}{conjugate transpose of a matrix $A\in\mat{m}{n}$}
\nomenclature[bmatrix]{$A^{-*}$}{conjugate transpose of the inverse of an invertible matrix $A\in\mat{n}{n}$}
\nomenclature[bmatrix]{$\det A$}{determinant of a matrix $A\in\mat{n}{n}$}
\nomenclature[bmatrix]{$A^+$}{Moore-Penrose pseudo inverse of a matrix $A\in\mat{m}{n}$}
\nomenclature[cb]{$\rkr A(z)$}{rank of a rational matrix $A(z)\in\matrz{m}{n}$}

\nomenclature[cb]{$G^\sim(z)$}{$:=G\left(\overline{z}^{-1}\right)^*$ for a rational matrix $G(z)\in\matrz{n}{n}$}

\nomenclature[dddddd]{$\pipe x\pipe_2$}{2-norm of a vector $x\in\K^n$}
\nomenclature[ddddde]{$\ell^2(\K^n)$}{space of quadratic-summable sequences $x\in(\K^n)^{\N_0}$, \ie $\sum_{k=0}^{\infty}{\pipe x_j\pipe_2}<\infty$}
\nomenclature[dddddf]{$\pipe x\pipe_{\ell^2}$}{$=\left(\sum_{k=0}^{\infty}{\pipe x_j\pipe^2_2}\right)^\frac12$; $\ell^2$-norm of a sequence $x\in\ell^2(\K^n)$}
\nomenclature[fasystemsc]{$\Vshift$}{set of all $x^0\in\K^n$ such that there exists $\behavior$ with $Ex_0=Ex^0$ }
\nomenclature[zzz]{$\mathcal W ^0_{\outputsystem*}$}{set of all $x^0\in\K^n$ such that there exists $\vect{x,u}\in\ZD$ with $Ex_0=Ex^0$}

\nomenclature[fasystemsa]{$\behavior *$}{set of all $\vect{x,u}$ which solve the discrete-time IDE \eqref{eq:linsystemdisc}}
\nomenclature[fasystemsb]{$\cV$}{system space of $\system$, see Def.~\ref{def:systemspace}}
\nomenclature[zz]{$\ZD$}{set of all $\behavior$ such that $Cx_j + Du_j =0,\,j\in\N_0$}

{%
\section*{Bibliography}
\bibliographystyle{abbrvnat}
\bibliography{../bib/master,../bib/literatureresearch}

\begin{thebibliography}{41}
\providecommand{\natexlab}[1]{#1}
\providecommand{\url}[1]{\texttt{#1}}
\expandafter\ifx\csname urlstyle\endcsname\relax
  \providecommand{\doi}[1]{doi: #1}\else
  \providecommand{\doi}{doi: \begingroup \urlstyle{rm}\Url}\fi

\bibitem[Backes(2006)]{backes_extremalbedingungen_2006}
A.~Backes.
\newblock \emph{{Extremalbedingungen f{\"u}r Optimierungs-Probleme mit
  Algebro-Differentialgleichungen}}.
\newblock {Logos-Verlag}, Berlin, June 2006.
\newblock ISBN 978-3-8325-1268-2.
\newblock Also as Dissertation, Institut f\"ur Mathematik,
  Humboldt-Universit\"at zu Berlin, 2006.

\bibitem[Bankmann(2016)]{bankmann_linear-quadratic_2015}
D.~Bankmann.
\newblock \emph{On Linear-Quadratic Control Theory of Implicit Difference
  Equations}.
\newblock Master's thesis, Technische Universit{\"a}t Berlin, Berlin, Aug.
  2016.
\newblock Available from \url{http://dx.doi.org/10.14279/depositonce-5440}.

\bibitem[Bender and Laub(1987)]{bender_laub_svdzeug_1987}
D.~J. Bender and A.~J. Laub.
\newblock The linear-quadratic optimal regulator for descriptor systems:
  Discrete-time case.
\newblock \emph{Automatica}, 23\penalty0 (1):\penalty0 71--85, 1987.

\bibitem[Berger(2014)]{berger_differential-algebraic_2014}
T.~Berger.
\newblock \emph{On {{Differential}}-{{Algebraic Control Systems}}}.
\newblock {Universit{\"a}tsverlag}, Ilmenau, 2014.
\newblock Fakult{\"a}t f{\"u}r Mathematik und Naturwissenschaften, Technische
  Universit{\"a}t Ilmenau.

\bibitem[Berger and Reis(2013)]{BerR13}
T.~Berger and T.~Reis.
\newblock Controllability of linear differential-algebraic systems -- a survey.
\newblock In A.~Ilchmann and T.~Reis, editors, \emph{Surveys in
  Differential-Algebraic Equations I}, Differ.-Algebr. Equ. Forum, pages 1--61.
  Springer-Verlag, Berlin, Heidelberg, 2013.

\bibitem[Brenan et~al.(1996)Brenan, Campbell, and Petzold]{BreCP96}
K.~E. Brenan, S.~L. Campbell, and L.~R. Petzold.
\newblock \emph{Numerical Solution of Initial-Value Problems in
  Differential-Algebraic Equations}, volume~14 of \emph{Classics in Applied
  Mathematics}.
\newblock SIAM, 1996.
\newblock ISBN 978-0-89871-353-4.

\bibitem[Bunse-Gerstner et~al.(1999)Bunse-Gerstner, Byers, Mehrmann, and
  Nichols]{bunse-gerstner_feedback_1999}
A.~Bunse-Gerstner, R.~Byers, V.~Mehrmann, and N.~K. Nichols.
\newblock Feedback design for regularizing descriptor systems.
\newblock \emph{Linear Algebra and its Applications}, 299\penalty0
  (1\textendash{}3):\penalty0 119--151, Sept. 1999.

\bibitem[Byers et~al.(1997)Byers, Geerts, and Mehrmann]{byers_descriptor_1997}
R.~Byers, T.~Geerts, and V.~Mehrmann.
\newblock Descriptor systems without controllability at infinity.
\newblock \emph{SIAM Journal on Control and Optimization}, 35\penalty0
  (2):\penalty0 462--479, Mar. 1997.

\bibitem[Byers et~al.(2009)Byers, Mackey, Mehrmann, and
  Xu]{byers_symplectic_2009}
R.~Byers, D.~S. Mackey, V.~Mehrmann, and H.~Xu.
\newblock Symplectic, {{BVD}}, and palindromic approaches to discrete-time
  control problems.
\newblock In \emph{Collection of {{Papers Dedicated}} to the 60-th
  {{Anniversary}} of {{Mihail Konstantinov}}}, pages 81--102. {Publishing House
  RODINA}, Sofia, 2009.

\bibitem[Clements and Glover(1989)]{clements_spectral_1989}
D.~J. Clements and K.~Glover.
\newblock Spectral factorization via {{Hermitian}} pencils.
\newblock \emph{Linear Algebra and its Applications},
  122\textendash{}124:\penalty0 797--846, Sept. 1989.

\bibitem[Clements et~al.(1997)Clements, Anderson, Laub, and
  Matson]{clements_spectral_1997}
D.~J. Clements, B.~D.~O. Anderson, A.~J. Laub, and J.~B. Matson.
\newblock Spectral factorization with imaginary-axis zeros.
\newblock \emph{Linear Algebra and its Applications}, 250:\penalty0 225--252,
  Jan. 1997.

\bibitem[Dai(1989)]{dai_singular_1989}
L.~Dai.
\newblock \emph{Singular {{Control Systems}}}, volume 118 of \emph{Lecture
  Notes in Control and Information Sciences}.
\newblock {Springer}, Berlin, 1989.

\bibitem[Gantmacher(1960)]{gantmacher_theory_1960}
F.~R. Gantmacher.
\newblock \emph{Theory of {{Matrices Vol}}. 2}.
\newblock {Chelsea}, New York, 1960.

\bibitem[Gohberg et~al.(2006)Gohberg, Lancaster, and
  Rodman]{gohberg_indefinite_2006}
I.~Gohberg, P.~Lancaster, and L.~Rodman.
\newblock \emph{Indefinite {{Linear Algebra}} and {{Applications}}}.
\newblock {Birkh{\"a}user}, Basel, Feb. 2006.

\bibitem[Horn and Sergeichuk(2006)]{horn_canonical_2006}
R.~A. Horn and V.~V. Sergeichuk.
\newblock Canonical forms for complex matrix congruence and *congruence.
\newblock \emph{Linear Algebra and its Applications}, 416\penalty0
  (2\textendash{}3):\penalty0 1010--1032, July 2006.

\bibitem[Ikramov(2001)]{ikramov_inertia_2001}
K.~D. Ikramov.
\newblock On the inertia law for normal matrices.
\newblock \emph{Doklady Mathematics}, 64\penalty0 (2):\penalty0 141--142, 2001.

\bibitem[Ilchmann and Reis(2017)]{ilchmann_outer_2014}
A.~Ilchmann and T.~Reis.
\newblock Outer transfer functions of differential-algebraic systems.
\newblock \emph{ESAIM: Control, Optimisation and Calculus of Variations},
  23\penalty0 (2):\penalty0 391--425, 2017.

\bibitem[Ionescu and Weiss(1992)]{ionescu_computing_1992}
V.~Ionescu and M.~Weiss.
\newblock On computing the stabilizing solution of the discrete-time
  {{Riccati}} equation.
\newblock \emph{Linear Algebra and its Applications}, 174:\penalty0 229--238,
  Sept. 1992.

\bibitem[Kahane et~al.(1999)Kahane, Mirkin, and Palmor]{KahMP99}
A.~C. Kahane, L.~Mirkin, and Z.~J. Palmor.
\newblock Discrete-time lifting via implicit descriptor systems.
\newblock In \emph{Proc. European Control Conference}. Karlsruhe, Germany,
  1999.

\bibitem[Kunkel and Mehrmann(2006)]{kunkel_differential-algebraic_2006}
P.~Kunkel and V.~Mehrmann.
\newblock \emph{{Differential-Algebraic Equations: Analysis and Numerical
  Solution}}.
\newblock {European Mathematical Society Publishing House}, Z{\"u}rich, Feb.
  2006.

\bibitem[Kunkel and Mehrmann(2008)]{kunkel_optimal_2008}
P.~Kunkel and V.~Mehrmann.
\newblock Optimal control for unstructured nonlinear differential-algebraic
  equations of arbitrary index.
\newblock \emph{Mathematics of Control, Signals, and Systems}, 20\penalty0
  (3):\penalty0 227--269, 2008.

\bibitem[Kunkel et~al.(2014)Kunkel, Mehrmann, and
  Scholz]{kunkel_self-adjoint_2014}
P.~Kunkel, V.~Mehrmann, and L.~Scholz.
\newblock Self-adjoint differential-algebraic equations.
\newblock \emph{Mathematics of Control, Signals, and Systems}, 26\penalty0
  (1):\penalty0 47--76, 2014.

\bibitem[Kurina and M{\"a}rz(2004)]{kurina_linear-quadratic_2004}
G.~A. Kurina and R.~M{\"a}rz.
\newblock On linear-quadratic optimal control problems for time-varying
  descriptor systems.
\newblock \emph{SIAM Journal on Control and Optimization}, 42\penalty0
  (6):\penalty0 2062--2077, 2004.

\bibitem[Lancaster and Rodman(1995)]{lancaster_algebraic_1995}
P.~Lancaster and L.~Rodman.
\newblock \emph{Algebraic {{Riccati Equations}}}.
\newblock {Clarendon Press}, Oxford, 1995.

\bibitem[Laub(1979)]{laub_schur_1979}
A.~Laub.
\newblock A {{Schur}} method for solving algebraic {{Riccati}} equations.
\newblock \emph{IEEE Transactions on Automatic Control}, 24\penalty0
  (6):\penalty0 913--921, 1979.

\bibitem[Laub(1991)]{laub_invariant_1991}
A.~J. Laub.
\newblock Invariant subspace methods for the numerical solution of {R}iccati
  equations.
\newblock In S.~Bittanti, A.~J. Laub, and J.~C. Willems, editors, \emph{The
  {R}iccati Equation}, Communications and Control Engineering Series, pages
  163--196. Springer, Berlin, Heidelberg, 1991.
\newblock ISBN 978-3-642-63508-3 978-3-642-58223-3.

\bibitem[Luenberger and Arbel(1977)]{LueA77}
D.~G. Luenberger and A.~Arbel.
\newblock Singular dynamic {L}eontief systems.
\newblock \emph{Econometrica}, 45\penalty0 (4):\penalty0 991--995, 1977.

\bibitem[Mackey et~al.(2006)Mackey, Mackey, Mehl, and
  Mehrmann]{mackey_structured_2006}
D.~S. Mackey, N.~Mackey, C.~Mehl, and V.~Mehrmann.
\newblock Structured polynomial eigenvalue problems: Good vibrations from good
  linearizations.
\newblock \emph{SIAM Journal on Matrix Analysis and Applications}, 28\penalty0
  (4):\penalty0 1029--1051, Jan. 2006.

\bibitem[Mehrmann(1991)]{mehrmann_autonomous_1991}
V.~Mehrmann.
\newblock \emph{The {{Autonomous Linear Quadratic Control Problem}}}, volume
  163 of \emph{Lecture Notes in Control and Information Sciences}.
\newblock {Springer}, Heidelberg, 1991.

\bibitem[Mehrmann and Scholz(2014)]{mehrmann_self-conjugate_2014}
V.~Mehrmann and L.~Scholz.
\newblock Self-conjugate differential and difference operators arising in the
  optimal control of descriptor systems.
\newblock \emph{Operators and Matrices}, 8\penalty0 (3):\penalty0 659--682,
  2014.

\bibitem[Pappas et~al.(1980)Pappas, Laub, and Sandell]{pappas_numerical_1980-1}
T.~Pappas, A.~J. Laub, and N.~R. Sandell.
\newblock On the numerical solution of the discrete-time algebraic {R}iccati
  equation.
\newblock \emph{IEEE Transactions on Automatic Control}, 25\penalty0
  (4):\penalty0 631--641, 1980.

\bibitem[Pearson et~al.(1988)Pearson, Chapman, and Shields]{PeaCS88}
D.~W. Pearson, M.~J. Chapman, and D.~N. Shields.
\newblock Partial singular-value assignment in the design of robust observers
  for discrete-time descriptor systems.
\newblock \emph{IMA Journal of Mathematical Control and Information},
  5:\penalty0 203--213, 1988.

\bibitem[Pontryagin et~al.(1962)Pontryagin, Boltyanskii, Gamkrelidze, and
  Mishchenko]{pontryagin_mathematical_1962}
L.~S. Pontryagin, V.~G. Boltyanskii, R.~V. Gamkrelidze, and E.~F. Mishchenko.
\newblock \emph{The {{Mathematical Theory}} of {{Optimal Processes}}}.
\newblock {Interscience Publishers}, Jan. 1962.

\bibitem[Rantzer(1996)]{rantzer_kalman-yakubovich-popov_1996}
A.~Rantzer.
\newblock On the {{Kalman}}-{{Yakubovich}}-{{Popov}} lemma.
\newblock \emph{Systems \& Control Letters}, 28\penalty0 (1):\penalty0 7--10,
  June 1996.

\bibitem[Reis(2011)]{reis_lure_2011}
T.~Reis.
\newblock Lur'e equations and even matrix pencils.
\newblock \emph{Linear Algebra and its Applications}, 434\penalty0
  (1):\penalty0 152--173, Jan. 2011.

\bibitem[Reis et~al.(2015)Reis, Rendel, and
  Voigt]{reis_kalmanyakubovichpopov_2015}
T.~Reis, O.~Rendel, and M.~Voigt.
\newblock The {{Kalman}}\textendash{}{{Yakubovich}}\textendash{}{{Popov}}
  inequality for differential-algebraic systems.
\newblock \emph{Linear Algebra and its Applications}, 485:\penalty0 153--193,
  Nov. 2015.

\bibitem[Schr{\"o}der(2008)]{schroder_palindromic_2008}
C.~Schr{\"o}der.
\newblock \emph{Palindromic and Even Eigenvalue Problems -- Analysis and
  Numerical Methods}.
\newblock Dissertation, Institut f{\"u}r Mathematik, Technische Universit{\"a}t
  Berlin, 2008.

\bibitem[Stoorvogel and Saberi(1998)]{stoorvogel_zeug_1998}
A.~A. Stoorvogel and A.~Saberi.
\newblock The discrete algebraic {R}iccati equation and linear matrix
  inequality.
\newblock \emph{Linear Algebra and its Applications}, 274:\penalty0 317--365,
  1998.

\bibitem[Stykel(2003)]{stykel_input-output_2003}
T.~Stykel.
\newblock Input-output invariants for descriptor systems.
\newblock Preprint PIMS-03-1, Pacific Institute for the Mathematical Sciences,
  University of Calgary, Calgary, 2003.

\bibitem[Van~Dooren(1983)]{dooren_reducing_1983}
P.~Van~Dooren.
\newblock Reducing subspaces: Definitions, properties and algorithms.
\newblock In B.~K\aa{}gstr{\"o}m and A.~Ruhe, editors, \emph{Matrix
  {{Pencils}}}, number 973 in Lecture Notes in Mathematics, pages 58--73.
  {Springer}, Berlin, 1983.

\bibitem[Voigt(2015)]{voigt_linear-quadratic_2015}
M.~Voigt.
\newblock \emph{On Linear-Quadratic Optimal Control and Robustness of
  Differential-Algebraic Systems}.
\newblock {Logos-Verlag}, Berlin, 2015.
\newblock ISBN 978-3-8325-4118-7.
\newblock Also as Dissertation, Fakult{\"a}t f{\"u}r Mathematik,
  Otto-von-Guericke-Universit{\"a}t Magdeburg, 2015.

\end{thebibliography}
}
\end{document}